\documentclass[11pt]{article}
\usepackage{amssymb, verbatim}
\usepackage{amsmath}	
\usepackage{amsfonts}
\usepackage{amsthm}
\usepackage{epsfig}  
\usepackage{subfigure}
\usepackage{changebar}
\usepackage{mathrsfs}
\usepackage{euscript, euler, latexsym, amssymb}
\usepackage{delarray}
\usepackage{commands}
\usepackage[all]{xy}
\usepackage{verbatim}

\newcommand{\marked}{\mathcal{M}\hspace{-1pt}\word{arked}}
\newcommand{\Marked}{\mathcal{M}\hspace{-1pt}\word{arked}}
\newcommand{\pecial}{\word{special}}
\newcommand{\oc}{\word{oc}}
\newcommand{\Fatoc}{\Fat^{\word{oc}}}
\newcommand{\closed}{\word{closed}}
\newcommand{\Closed}{\mathcal{C}\hspace{-2pt}\mathit{losed}}
\newcommand{\Open}{\mathcal{O}\hspace{-2pt}\mathit{pen}}
\newcommand{\Out}{\mathcal{O}\hspace{-2pt}\mathit{ut}}
\newcommand{\In}{\mathcal{I}\hspace{-2pt}\mathit{n}}
\newcommand{\Outgoing}{\mathcal{O}\hspace{-2pt}\mathit{ut}}
\newcommand{\Incoming}{\mathcal{I}\hspace{-2pt}\mathit{n}}
\newcommand{\Special}{\mathcal{S}\hspace{-2pt}\mathit{pecial}}
\newcommand{\win}{\mathit{in}}
\newcommand{\out}{\mathit{out}}
\newcommand{\free}{\mathit{free}}
\newcommand{\pants}{\mathit{pants}}
\newcommand{\aFat}{\Fat^a}
\newcommand{\Bigoplus}[1]{\underset{#1}{\bigoplus}}
\newcommand{\Map}{\mathit{Map}}
\newcommand{\Split}{\mathit{Split}}

\newcommand{\Mod}{\mathit{Mod}}
\newcommand{\Tub}{\cT\hspace{-2.5pt} \mathit{ub}}
\newcommand{\Thom}{\mathit{Thom}}
\newcommand{\VB}{\mathcal{V\hspace{-1pt}B}}
\newcommand{\tilfatop}{\til{\left(\aFat\right)^{op}}}
\newcommand{\tilfat}{\til{\aFat}}

\newcommand{\Mtop}{\cM^{\mathit{top}}_{\aFat}(M)}

\newcommand{\cylinder}{\mathit{cyl}}
\newcommand{\tilsigma}{{\tilde{\sigma}}}

\begin{document}

\title{Higher string topology operations.}
\author{V\'eronique Godin\\ Harvard University\\ godin@math.harvard.edu}
\maketitle
\abstract{In \cite{ChasSullivan} Chas and Sullivan defined an intersection-type product on the homology of the free loop space $LM$ of an oriented manifold $M$. In this paper we show how to extend this construction to a homological conformal field theory of degree $d$. In particular, we get operations on $H_*LM$ which are parameterized by the twisted homology of the moduli space of Riemann surfaces.}

%\tableofcontents

 \section{Introduction}

Let $M$ be an oriented manifold of dimension $d$. Let the free loop space 
\[LM=\Map(S^1, M)\]
of $M$ be the space of piecewise smooth maps from the circle into $M$. In \cite{ChasSullivan} Chas and Sullivan defined a Batalin-Vilkovisky algebra structure on the homology $H_*LM$ of the free loop space of any closed oriented manifold. In particular, they have constructed an intersection-type product
\begin{equation}\label{eqn:CS product}
\bullet : H_pLM\tens H_qLM\map H_{p+q-d}LM.
\end{equation}
Their paper started string topology; the study of the algebraic structures of the space $LM$ and various stringy spaces associated to $M$.

\begin{figure} \begin{center}
\mbox{\epsfig{file=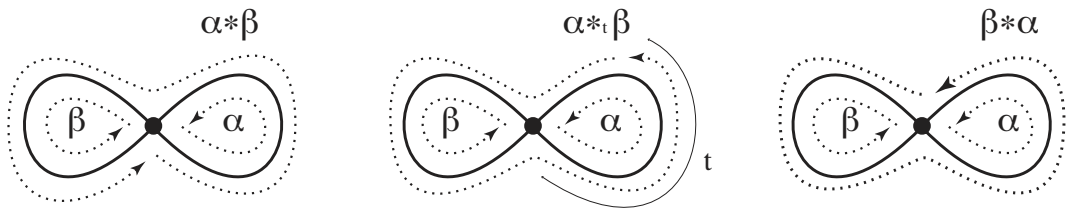, width=400pt}}
\caption{A family of compositions giving a homotopy between $a \bullet b$ and $b\bullet a$
 }\label{fig:firstfamily}
\end{center}
\end{figure}
The loop product of \eqref{eqn:CS product} is constructed by first intersecting chains in $LM\times LM$ with the submanifold $\Map(\infty, M)$ of composable loops and then by composing these composable loops. Their paper suggests a richer algebraic structure on $H_*LM$. For example, to prove that the loop product is graded commutative, they used the various compositions illlustrated by graphs in figure \ref{fig:firstfamily}. Hence there should be a space of graphs whose homology parameterized string topology operations,

This idea has already been pursued by Voronov in \cite{Voronov}, by Cohen and Jones in \cite{CohenJones} and by Cohen and the author in \cite{CohenGodin}. This paper extends these constructions to an action on $H_*LM$ by the homology of a space of graphs whose homotopy type is related to the moduli space of bordered Riemann surfaces. 

\subsection{Main result and examples}
\label{sec:main result}

\begin{theorem}\label{thm:HCFT}\label{thm:main}
The pair $(H_*LM,H_*LM)$ has the structure of a degree $d$ open-closed homological conformal field theory (HCFT) with positive boundary.
\end{theorem}

This theorem is the main result of this paper which, amongst other things, settles a conjecture from \cite{CohenVoronov}. We will first of all spell out what structure this gives on the pair $(H_*LM, H_*M)$. Specific examples will then be given.

By an \emph{open-closed cobordism} $S$, we mean an oriented cobordism between two 1-dimensional manifolds. More precisely, $S$ is a compact oriented surface with boundary. Its boundary is divided into three parts : the incoming part $\p_\win S$, the outgoing part $\p_\out S$ and the free part $\p_\free S$. The incoming and the outgoing boundaries are the 1-dimensional manifolds. The free boundary is a cobordism between the boundary of $\p_\win S$ and the boundary of $\p_\out S$. We also choose parameterizing diffeomorphisms 
\[\varphi_\win : \p_\win S\map N \qquad \varphi_\out :\p_\out S\map P\]
to ordered disjoint unions $N$ and $P$ of $I$ or $S^1$. Note that this gives the connected components of both $\p_\win S$ and $\p_\out S$ an ordering and a parameterization. 
The \emph{mapping class group} 
\[\Mod(S)=\Mod^{\oc}(S)=\pi_0 \Diff\left(S; \p_\win S \du \p_\out S\right)\]
 of such an open-closed cobordism $S$ is the group of isotopy classes of orientation-preserving diffeomorphisms of $S$ which fix both the incoming boundary and the outgoing boundary pointwise. 
Whenever the outgoing boundary of $S_1$ is identified with the same 1-manifold as the incoming boundary of $S_2$, we get a well-defined gluing $S_1\#S_2$ and  a group homomorphism
\[\Mod^{\oc}(S_1)\times \Mod^{\oc}(S_2) \map \Mod^{\oc}(S_1\#S_2).\]

\begin{figure} \begin{center}
\mbox{\epsfig{file=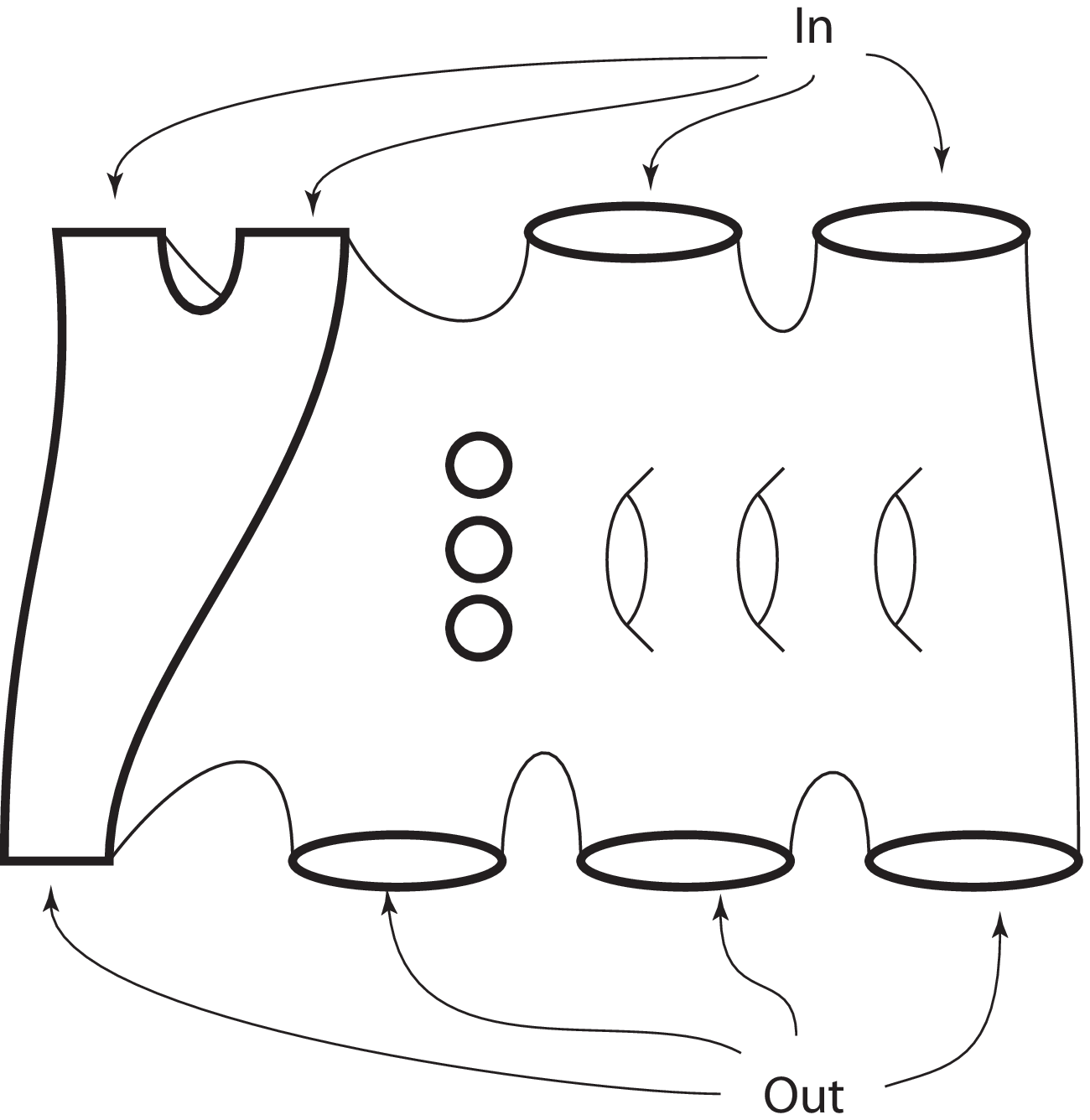, height=150pt}}\qquad\qquad
\mbox{\epsfig{file=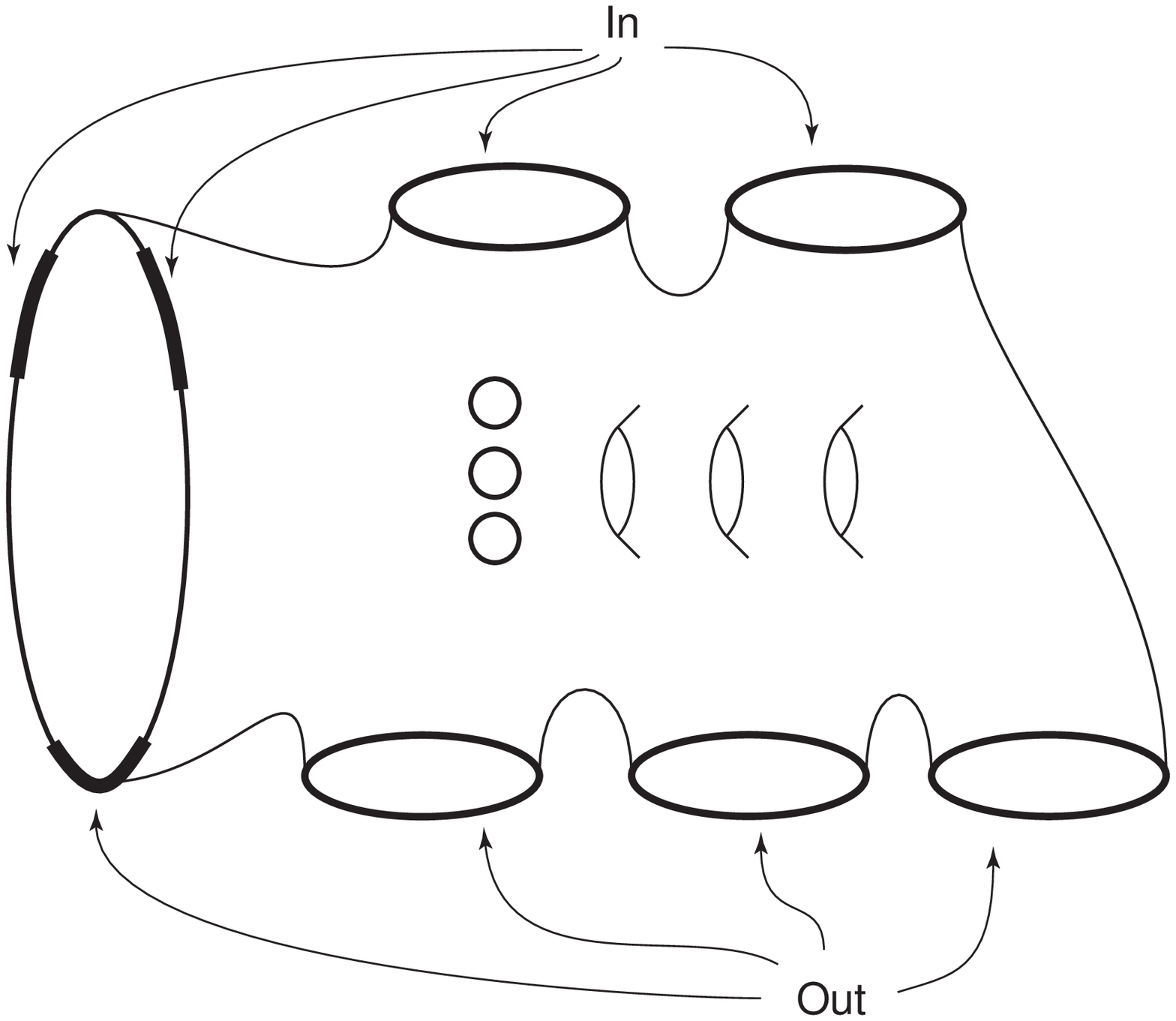, height=150pt}}
\caption{Two different pictures representing the same open-closed cobordism $S$.} \label{fig:oc cobordism}
\end{center}
\end{figure}
\begin{example}
The surface of figure \ref{fig:oc cobordism} has genus 3 and nine boundary components. The two drawings show the same open-closed cobordism between
\[\p_\win S \cong  I \du I \du S^1\du S^1 \qquad \p_\out S \cong I \du S^1\du S^1\du S^1.\]
\end{example}

Throughout this paper, we will assume that the boundary of every component of an open-closed cobordism is not completely contained in the incoming boundary. This eliminates, for example, the disk whose unique boundary component is completely incoming.
The term \emph{positive boundary} comes from this restriction.

Theorem \ref{thm:HCFT} says that for any open-closed cobordism $S$ there is a map
\[\mu_S: H_*(B\Mod^{\oc}(S);\det(\chi_{S})^{\tens d} \tens H_*LM^{\tens p}\tens H_*M^{\tens q} \map H_*LM^{\tens m} \tens H_*M^{\tens n}\]
which preserves degree.  Here $p$ and $q$ (respectively $m$ and $n$) are the number of circles and intervals in the incoming (respectively outgoing) boundary of $S$. We therefore get operations parameterized by the homology of the mapping class groups with twisted coefficients. 

The twisting $\det(chi_S)^{d}$ will be defined in section \ref{sec:orient}. We will consider first the virtual vector space
\[\chi_S = H_*(S;\p_\win S)= \left(H_1(S;\p_\win S), H_0(S,\p_\win S)\right)\]
with the appropriate action of $\Mod(S)$. We get a graded vector bundle $\det(\chi_S)$ above $B\Mod(S)$ which we think as lying in degree minus the Euler characteristic of $S$ relative to its boundary. 
When we glue two open-closed cobordism $S_1$ and $S_2$, the long exact sequence associated to the triple $(S_1\# S_2, S_1,\p_{\win} S_1)$ gives an identification
\[\det(\chi_{S_1})^{\tens d}\tens  \det(\chi_{S})^{\tens d} \cong \det(\chi_{S_1\# S_2})^{\tens d}.\]

We ask that the operations of $\mu_{S_1}$ and $\mu_{S_2}$ compose to give a commutative diagram.
\[\xymatrix{
H_*(B\Mod^{\oc}(S_2); \chi_{S_2}^d) \tens H_*(B\Mod^{\oc}(S_1);\chi_{S_1}^d) \tens H_*LM^{\tens p}\tens H_*M^{\tens q} \ar[d]\ar@/_175pt/[dd]_{\mu_2\cdot\mu_1}\\
 H_*(B\Mod^{\oc}(S_1\#S_2),\chi_{S_1\# S_2}^{d}) \tens H_*LM^{\tens p}\tens H_* M^{\tens q}\ar[d]^{\mu_{1\# 2}}\\
 H_*LM^{\tens k} \tens H_*M^{\tens j}
}\]

\begin{remark}
If the cobordism $S$ has no more than one boundary component that is completely free then $\chi_{S,d}$ is a trivial twisting. However there is no way of picking trivialization for each of these components in a way compatible with the gluing.
\end{remark}

\begin{figure} \begin{center}
\mbox{
\subfigure[Cylinder]{\epsfig{file=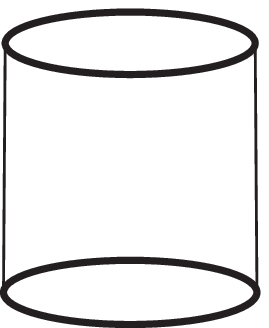, height=50pt}}
\qquad
\subfigure[Pair of pants]{\epsfig{file=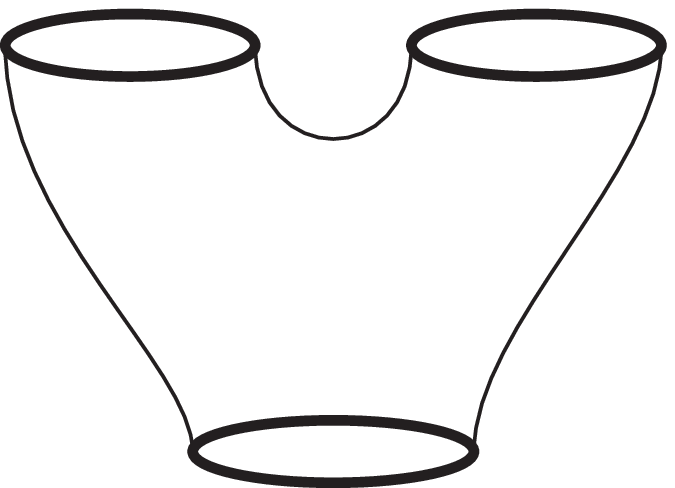, height=50pt}}
\qquad
\subfigure[Mouth piece]{\epsfig{file=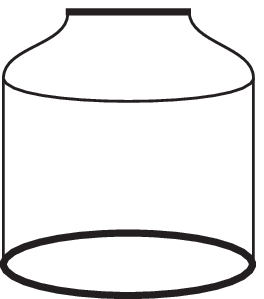, height=50pt}}
\qquad
\subfigure[Mouth piece backwards]{\epsfig{file=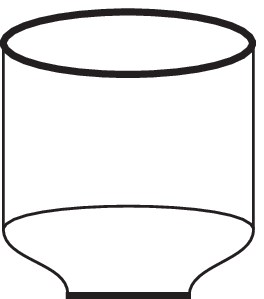, height=50pt}}
\qquad
\subfigure[Pair of flaps]{\epsfig{file=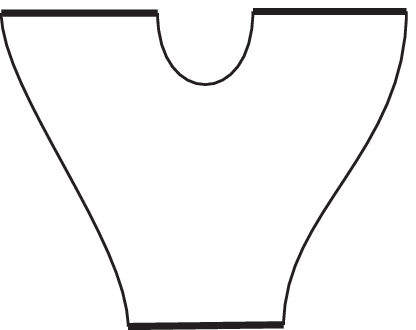, height=50pt}}
}
\caption{Some interesting cobordisms}\label{fig:int cobordisms}
\end{center}
\end{figure}

\begin{example}
The mapping class group of the cylinder $S=S_\cylinder$ of figure \ref{fig:int cobordisms}  is an infinite cyclic group generated by a Dehn twist around any curve homotopic to a boundary component. The twisting 
\[H_*(S;\p_{\win} S) = (0,0)\]
which is naturally oriented. We therefore get
\[H_0(\Mod^\oc(S_{\cylinder});\det(\chi_{S})^{\tens d}) \cong \bZ \qquad H_1(\Mod^\oc(S_{\cylinder});\det(\chi_S)^{\tens d}) \cong \bZ.\] 
Theorem \ref{thm:main} gives two interesting operations 
\[Id : H_*LM\map H_*LM \qquad \Delta : H_*LM\map H_{*+1} LM\]
corresponding respectively to the generator of $H_0$ and $H_1$. Here $\Delta$ is the BV operator of Chas and Sullivan which is given by the composition
\[\xymatrix{
H_*LM \ar[rr]^-{[S^1]\tens Id}&& H_1S^1\tens H_*LM \ar[r] & H_{*+1} (S^1\times LM) \ar[r]& H_{*+1}LM
}\]
where the last map is induced from the usual $S^1$ action.
\end{example}

\begin{example}
Lets now consider the pair of pants.
Its mapping class group is generated by three Dehn twists : one for each boundary. These Dehn twists commute and hence
\[\Mod(S_{\pants}) \cong \bZ\times\bZ\times\bZ.\]
The relative homology
\[H_*(S_\pants, \p_\win S_\pants) = (0, \bZ)\]
of the pair of pants is generated by a curve linking the first incoming boundary component to the second one. This choice of a generator determines an identification
\begin{alignat*}{8}
H_0(\Mod(S_{\pants});\det(\chi_{S})^{\tens d}) &\cong \bZ \qquad &H_1(\Mod(S_\pants);\det(\chi_S)^{\tens d}) &\cong \bZ^{\oplus 3} \\
H_2(\Mod(S_\pants);\det(\chi_S)^{\tens d})
 &\cong \bZ^{\oplus 3}&
H_3(\Mod(S_\pants);\det(\chi_S)^{\tens d}) &\cong  \bZ.
\end{alignat*}
The generator of $H_0(\Mod(S_{\pants});\det(\chi_S)^{\tens d})$ determines an operation
\[\star_{CS} : H_*LM\tens H_*LM \map H_*LM\]
which is, up to a sign, the Chas and Sullivan product. The operations corresponding to higher homology groups are simply composition of this product with various $\Delta$. 
\end{example}

\begin{remark}
As we shall see in section \ref{sub:ChasSullivan},
 the operation $\star_{CS}$  is anti-associative for an odd-dimensional manifold $M$. By regrading the group $H_*LM$, Chas and Sullivan defined a strictly associative product. However, this regrading would render our coproduct anti-co-associative. 
\end{remark}

\begin{example}
The mapping class group of both mouth pieces in figure \ref{fig:int cobordisms} is again $\bZ$ generated by a Dehn twist along a curve which is homotopic to a boundary component. Once again the bundle $\chi_S$ is oriented and hence these cobordisms therefore give two operations each. 

First the relative homology of the forward mouthpiece
\[H_*(S_{\mathit{mouthpiece}};\p_\win S_{\mathit{mouthpiece}}) \cong (0,\bZ).\]
is oriented by the orientation on the outgoing boundary component.
We obtain maps
\[\mu_{S_{\mathit{mouthpiece}},0} : H_*M \map H_{*-d} LM
\qquad \mu_{S_{\mathit{mouthpiece}},1} : H_*M \map H_{*-d+1} LM.\]
The first one sends
\[\mu_{S_{\mathit{mouthpiece}},0}([M]) = \chi(M)\ [c]\]
where $[c]$ is the class in $H_0LM$ corresponding to a single constant loop and where $\chi(M)$ is the Euler characteristic of $M$. The generator in degree one is obtained from the generator in degree zero by gluing the generator of $H_1(\Mod(S_{cylinder}))$ and hence the map
\[\mu_{S_{\mathit{mouthpiece}},1} : H_*M \map H_{*-d+1}LM\]
is the composition of $\mu_{S_\mathit{mouthpiece},0}$ with $\Delta$. Since $\Delta$ is zero on constant loops and hence 
\[\mu_{S_{\mathit{mouthpiece}},1} \equiv 0.\]

We pick the positive orientation on the relative homology of the backwards mouthpiece 
\[H_*(S_{\mathit{eceiphtuom}};\p_\win S_{\mathit{eceiphtuom}})\cong(0,0).\]
The first of the two maps
\[\mu_{S_{\mathit{eceiphtuom}},0} : H_*LM\map H_*M \qquad  \mu_{S_{\mathit{eceiphtuom}},1}: H_*LM \map H_{*+1} M.\]
is the induced by the evaluation map 
\[ev^0 : LM \map M\]
at the base point. As before the second map is the composition of $\Delta$ with the first $\mu_{S_{\mathit{eceiphtuom}},0}$ and therefore is 
\[H_*LM \cong H_*\left(M^{S^1}\right) \map H_{*+1} M.\]
\end{example}

\begin{example}
By Smale's theorem, the mapping class group of  the pair of flaps $S_\mathit{flaps}$ is trivial. Hence the twisting $\chi_{S_{\mathit{flaps}}}$ is trivial and the twisted homology of  $\Mod(S_\mathit{flaps})$ is trivial except that
\[H_0( \Mod_\oc(S_\mathit{flaps});\det(\chi_S)^{\tens d})\cong\bZ.\] 
The operation corresponding to a generator is the intersection product
\[H_pM\tens H_q M \map H_{p+q-d} M\]
of $M$. Again this intersection-product is not strictly associative.
\end{example}

%--------------Future Extensions
\subsection{Conjectures}
\subsubsection{Relative string topology}
Our construction should be a special case of a more general structure. From the work of Sullivan in \cite{SullivanOpenClosed} and Ram\'irez in \cite{Ramirez}, one expects a relative version of string topology. More precisely, fix any collection $\cB$ of submanifolds of $M$. For any $A, B$ in $\cB$, let
\[P_{AB}M = \left\{ \gamma : I \map M \qquad \gamma(0)\in A \quad \gamma(1)\in B\right\}\]
be the space of piecewise smooth paths in $M$ which starts in $A$ and ends in $B$. Using the ideas of string topology, Sullivan constructed an operation
\[H_p(P_{AB}M) \tens H_*(P_{BC}M) \map H_*(P_{AC} M)\]
which decreases degree by the dimension of $B$. 
We conjecture that this relative string product is part of a more general structure including the operations of theorem \ref{thm:main} as follows. 

An $\cB$-manifold is a parameterized 1-dimensional manifold $N$ whose boundary is labeled
\[ l_N :\p N\map \cB\]
by elements of $\cB$. An $\cB$-open-closed cobordism $S$ is an open-closed cobordism between two $\cB$-manifolds. In particular, $S$ comes with a labeling
\[l_S:\pi_0\p_\free S \map \cB\]
of the connected components of its free boundary by elements of $\cB$. The mapping class groups 
\[\Mod^\cB(S) =\pi_0\Diff(S;\p_\win S\du \p_\out S; l_S)\]
of such an $S$ is the group of isotopy classes of diffeomorphism of $S$ which preserves the incoming and the outgoing boundary pointwise and also respects the labeling. In particular, these diffeomorphisms are allowed to permute entire circles of the free boundary only if they are labeled by the same element of $\cB$.

An $\cB$-cobordism has \emph{positive-boundary} if each of its component has some outgoing boundary.
For each $\cB$-cobordism $S$, there should be a twisting $\xi_S$ above each $B\Mod^{\cB}(S)$. These should be compatible operations
\[\xymatrix{
 H_*(B\Mod^{\cB}(S);\xi_S)\tens H_*(LM)^{\tens p} \tens \bigotimes H_*(P_{A_{i}B_{i}}M)\ar[r]^-{\mu_S}&
 H_*(LM)^{\tens q} \tens \bigotimes H_*(P_{C_{j}D_{j}}M)
}\]
where the incoming boundary of $S$ has $p$ circles and then intervals $I_i$ labeled by $A_i$ and $B_{i}$ and similarly for the outgoing. 
 
\begin{conjecture}\label{conj:relative}
The tuple 
\[\{H_*(P_{AB}M) \qquad A,\, B\in \cB\}\]
has the structure of a $\cB$-homological conformal field theory.
\end{conjecture}

Theorem  \ref{thm:main} proves this conjecture in the case where $\cB = \{M\}$. It actually constructs extra operations because of the existence of a trace map
\[H_*P_{MM}M \cong H_*M \map \bZ.\]
We believe that these operations in this conjecture can be constructed similarly using an appropriately decorated graph model.

Such a construction, especially a chain-level construction, would have application to contact homology through the work of NG, Sullivan and Sullivan. The existence of such a construction would also have consequences for the string operations of theorem \ref{thm:main}.

\subsubsection{Homotopy invariance}
The construction of the string operations $\mu_{S}$ presented in this paper uses the tangent bundle of the manifold $M$ as well as its differentiable structure. At first glance, we have therefore constructed diffeomorphism invariants of $M$. However in \cite{CohenKleinSullivan}, Cohen, Klein and Sullivan show that the BV structure defined by Chas and Sullivan does not depend on the smooth structure of $M$. They even show that any homotopy equivalence
\[f:M_1\map M_2\]
which sends the orientation of $M_1$ to an orientation of $M_2$ induces an isomorphism
\[Lf_* : H_*LM_1\map H_*LM_2\]
of BV algebra. Such an homotopy equivalence is called an \emph{oriented homotopy equivalence}. Hence the Chas and Sullivan structure is an oriented-homotopy invariant.

\begin{conjecture}\label{conj:he}
The operations of theorem \ref{thm:main} and the conjectured operations of \ref{conj:relative}  are oriented-homotopy invariant.
\end{conjecture}

This conjecture follow morally from work of Costello. In \cite{CostelloCalabiYau}, Costello constructed a positive-boundary open-closed TQFT structure on the pair $(A, HC_*A)$
where $A$ is a Frobenius algebra and $HC_*A$ are its Hochshild chains. He also showed that this TQFT structure was universal amongst all such structure on $(A,B)$.
 
If one manages to apply his theory applies to $A=C^*M$, there would be a purely algebraic construction of the dual of the string operations as 
 \[HH_*(C^*M) \cong H^*LM.\]
Hence the structure on $H_*LM$ would depend only on a choice of an $A_\infty$ Frobenius algebra structure on $C^*M$. However, Costello's definition of an $A_{\infty}$ algebra is too restrictive to apply to $C^*M$. It would require to have a non-degerate pairing
 \[<,> : C^*M\tens C^*M \map k \]
on the \emph{infinite}-dimensional vector space $C^*M$.

Note that to apply the universality to this conjecture, we would need   to lift the string operations conjectured in \ref{conj:relative} to the chain level or to the spectrum level and define them in a slightly more general setting.

\subsection{Tools and methods}

This paper will start by the construction of a fat graph model for the classifying space of the mapping class groups of our cobordism. A fat graph is a graph with cyclic ordering of the edges coming into each vertex. These fat graphs are spines of Riemann surfaces and the cyclic ordering comes from the orientation of the surfaces.

\begin{remark}
In \cite{CohenGodin}, Cohen and the author also used a space of metric Sullivan chord diagrams to construct string operations. These chord diagrams were introduced by Sullivan and first appeared in \cite{CohenGodin}. In that paper, Cohen and the author had conjectured that the space of metric Sullivan chord diagrams is a classifying space for the mapping class groups and then hinted about how to get operations over families of chord diagrams. 
However the conjecture is wrong : the space of chord diagrams is too small to be such a model. More precisely it's cellular dimension is smaller than the cohomological dimension of the mapping class group as was shown in \cite{godinthesis}. This explains the bigger model and the more complicated operations defined in this paper.
\end{remark}

In this paper, we use the category $\aFat$ of admissible fat graphs.  These are fat graphs whose incoming boundary cycles are disjoint circles. This category is introduced in details in the next section. In that section, we  also prove that it is, in fact, a model for the classifying space of the mapping class groups.

We also show how to construct a symmetric monoidal partial 2-category ${\mathit{Cob}}^{\Fat}$ enriched over categories. The objects of ${\mathit{Cob}}^{\Fat}$ are diffeomorphism classes of parameterized 1-manifolds. The category of morphisms is the category  $\aFat$ of admissible fat graphs.
The composition in ${\mathit{Cob}}^{\Fat}$  is defined by gluing fat graphs along their boundary components. 

We then attack the construction of the string operations. For each admissible fat graph $\G$, we show how to generalized the idea of \cite{CohenGodin} to get a map of spaces
\begin{equation}\label{eq:operation one}
M^{\p_\win \G} \times W_{\G} \map \Thom\left(\kappa_\G\right)
\end{equation}
where $\kappa_\G$ is a  bundle above the space $M^{\G}$ of piecewise smooth maps of $\G$ into $M$ and $W_{\G}$ is an Euclidean space. This gives an operation
\[\mu_\G : H_*M^{\p_\win\G} \map H_*(\Thom(\kappa_\G)) \cong H_*M^{\G} \map H_*M^{\p_\out \G}\]
which corresponds to $H_0(B\Mod(S_\G))$.

To get the higher operations, we construct one big map of spectra
\[\Sigma^{\infty} \left|\tilfatop \int M^{\p_\win -}\right| \map \Thom(\kappa)\]
Here $\tilfatop$ is a topological category which is a thickened version of $\aFat$. It contains the choices involved in the construction of the Thom-Pontrjagin collapses of \eqref{eq:operation one}. The target virtual bundle  $\kappa$ lies above the space
\[\left|\tilfatop \int M^{-}\right| \simeq E\Diff(S)\underset{\Diff(S)}\times \Map(S;M)\]
which twists the $M^{\G}$ together. This construction in section \ref{sec:TP} will consume the most of our energy. 

In section \ref{sec:orient} we complete the construction of the higher operations. We first study the determinant bundle of  $\kappa$ and relate it to the $\det(\chi)^{\tens d}$. We then use the Thom isomorphism
\[H_*\left(\Thom(\kappa); \det(\kappa)^{\tens d}\right) \map H_* \left|\tilfatop \int M^{-}\right|\]
which we follow with the restriction to the outgoing boundary component. In this section \ref{sec:orient}, we construct the twisting $\det\chi$ and define homological conformal field theory.

Finally in section \ref{sec:gluing operations}, we will show that these operations glue correctly.

\subsection{Acknowledgments}

I am thankful to Antonio Ram\'irez for starting this project with me. I would also like to thank Ralph Cohen, Michael Hopkins, Mohammed Abouzaid, Tyler Lawson and Andrew Stacey for important discussions about this project. I am also thankful to Nathalie Wahl for her comments about an earlier version.

%--------------------------------------------------------- Admissible Fat graphs
\section{Admissible fat graphs}

The goal of this section is to construct a category $\aFat$ whose objects are \emph{admissible} fat graphs and to prove that its geometric realization is a classifying space for the mapping class groups. More precisely,
\[|\aFat| \simeq \du_{[S]} B\Mod(S)\]
where the disjoint union ranges over all diffeomorphism type of cobordism $S$ whose boundary is divided into three parts
\[ \p_{\win} S \qquad \p_{\out} S \qquad \p_{\free} S.\]
We also have an ordering of the connected components of $\p_\win S$ and $\p_\out S$.

In the remainder of this paper, the model $\aFat$ will be used to define the new string operations.

%--------------------------------------------------------- Fat graphs
\subsection{Graphs and fat graphs}

In this section, we briefly introduce fat graphs and construct a category $\Fat$ following the work of  Igusa \cite[Chap. 8]{IgusaReidemeister}.

A \emph{graph} $G=(V,H,s,i)$ is a set of vertices $V$ a set of \emph{half-edges} $H$ a source map
$s:H\smap V$ and a fixed-point free involution $i:H\map H$ which pairs an half-edge $h$ with its other half $\ol{h}=i(h)$. Hence the set $E$ of edges in the graph $G$ is the set of orbits of $i$.

There are two different ways of constructing the geometric realization of $G$ corresponding to thinking of half-edges as an edge with an orientation or as an actual half of an edge. 
\[|G|^{\div} = \frac{ (I\times H) \du V}{\left\{(0,h)\sim s(h),\, (1,h)\sim (1,i(h))\right\}} \qquad |G|=\frac{ (I\times H)\du V}{\left\{(0,h)\sim s(h) \ (t,h)\sim (1-t,i(h))\right\}}.\]
Both these definitions will be useful to us and we will use both.

We can extend the map $s$ and $i$ by the identity on $V$ to get
\[\tilde{i}:V\du H \map V\du H\qquad  \tilde{s} : V\du H \map V.\]
A \emph{morphism} $f:G_0\map G_1$ of graph is a map of sets
\[f : V_0\du H_0 \map V_1\du H_1\]
which commutes with the extended $\tilde{i}$ and $\tilde{s}$. We will add the requirement that $f$ is surjective and that it induces a homotopy equivalence between the geometric realization. At the combinatorial level, we are requiring that any half edge $h_1$ of $G_1$ is the image of exactly one half-edge of $G_0$. We also forces that for any vertex $v_1$ of $G_1$, the subgraph $f^{-1}(v_1)$ of $G_0$ is a tree, ie contains no loops.

A \emph{fat graph} $\G=(G,\sigma)$ is a graph $G$ with a permutation $\sigma$ on the set  of half-edges. We ask that the orbits of $\sigma$ are exactly the sets $s^{-1}(v)$. Hence $\sigma$ gives cyclic orderings of the half-edges starting at each vertex.

From a fat graph $\G$, we can construct a surface $\Sigma_\G$ as follows. For each vertex $v$ of $G$, we take a disk $D_v$ whose oriented boundary we identifies with
\[ \p D_v = \frac{s^{-1}(v)\times I}{\{(h,1)\sim (\sigma(h),0)\}}. \]
For each edge $e=\{h,\ol{h}\}$ of $G$ we take a strip
\[S_e= \frac{\{h,\ol{h}\}\times I}{\{(h,t)\sim (\ol{h},1-t)\}}\times \frac{\{h,\ol{h}\}\times I}{\{ (h,t)\sim (\ol{h},1-t)\}}\]
To construct the surface $\Sigma_\G$, we glue one side of this band to the disk of $s(h)$ 
 and the other side to the disk of $s(\ol{h})$ by identifying
\begin{eqnarray*}
(t,(0,h)) &\sim& (t,h)\in \p D_{s(h)}\\
(t,(0,\ol{h}))&\sim& (t,\ol{h})\in \p D_{s(\ol{h})}
\end{eqnarray*}
Using this last identification we define
\[\Sigma_\G =  \frac{\Du D_v \du \Du S_e}{\sim}.\]
The geometric realization of $\G$ sits inside $\Sigma_\G$ as a spine. The boundary of $\Sigma_\G$ is made up of intervals
\[\p_h = \{[(h,0),(h,t)] t\in I\} \subset S_e\]
which leads from $D_{s(h)}$ to $D_{s(\ol{h})}$. Because of this, we define the \emph{boundary cycles} of $\G$ to be the permutation
\[\omega = \sigma \cdot i.\]
Note that since $\sigma = \omega \cdot i$, the fat structure on $\G$ is completely determined by $(G,\om)$. 

Following Igusa, we let a \emph{morphism}
\[\varphi: \G_0=(G_0,\sigma_0)\map \G_1=(G_1,\sigma_1)\]
of fat graphs be a morphism of graphs $\varphi:G_0\smap G_1$ which preserves the boundary cycles.
More precisely, we ask that for any $h\in H_1$
\[\om_1(h_1) = \varphi\left(\om_1^k (h_0)\right)\]
where $h_0$ is the unique preimage of $h_1$ in $H_0$ and $k$ is the smallest positive integer so that $\om_1^k (h_0)$ is not collapsed. In other words, $\om_1$ is gotten by removing the collapsed elements from $\om_0$ and  by applying $\varphi$ to the remaining half-edges.

\begin{figure} \begin{center}
\mbox{\epsfig{file=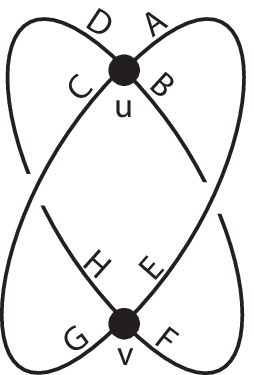, width=100pt}}
\mbox{\epsfig{file=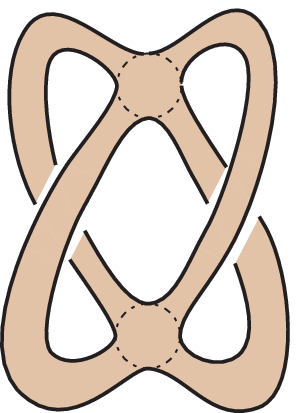, width=90pt}}
\caption{A fat graph $\G$ and its surface $\Sigma_\G$}\label{fig:fatgraph}
\end{center}
\end{figure}
\begin{example}
As in figure \ref{fig:fatgraph}, we consider the graph with two vertices $u$ and $v$ and whose half-edges are 
\[H=\{A, B,C,D,E,F,G,H,\}\]
The source map is 
\[s(A)= s(B)=s(C)=s(D)=u\qquad s(E)=s(F)=s(G)=s(H)=v\]
and the involution associates
\[i(A)= E\quad  i(B)=F \quad i(C)=G\quad i(D)=H.\]
We put the fat structure
\[\sigma=(ABCD)(EFGH)\]
 on $G$ which is given by the clockwise orientation on the plane. The boundary cycles of $\G$ are then
 \[\om= (AFCH)(BGDE)\]
Hence the surface $\Sigma_\G$, which is also pictured in ffigure \ref{fig:fatgraph}, has two boundary components.
Since $\Sigma_\G$ is homotopy equivalent to $\G$, its Euler characteristic is $-2$ and hence $\Sigma_\G$ is homeomorphic to a torus with two boundary components.
\end{example}

Igusa defined $\Fat$ to be the category whose objects are fat graphs $(\G,\om)$. We ask that every vertex of $\G$ has valence at least three, ie there are at least three half-edges with source $v$. The morphisms of $\Fat$ are the morphism of fat graphs. 

For any closed surface $S$ with marked points $p_1,\ldots p_k$, let the mapping class group of $S$ 
\[\Mod(S)= \pi_0\Diff(S; \{p_1,\ldots p_k\})\]
be the group of isotopy classes of diffeomorphisms of $S$ which fix the $p_i$'s.
\begin{theorem}[Harer, Strebel, Penner, Igusa]\label{thm:igusa}
The geometric realization of the category $\Fat$ is 
\[|\Fat| \simeq \Du_{[S]} B\Mod(S)\]
a disjoint union of classifying space of the mapping class groups of marked surfaces with positive Euler characteristic. Here the disjoint union is over all diffeomorphism type of surfaces $S$ of genus $g$ with $k\geq 1$ marked points with negative Euler characteristic.
\end{theorem}

\begin{remark}
We will let $\Fat_{g,s}$ be the subcategory which realizes to the connected component homotopy equivalent to $B\Mod(S)$ where $S$ is a connected surface of genus $g$ with $s$ marked points. 
The  universal cover of $|\Fat_{g,s}|$ is the geometric realization of the $E\Fat_{g,s}$ whose objects 
are pairs
\[\left(\G, [f:|\G|\map S]\right).\]
Here $[f]$ is a isotopy classes of embeddings that avoid the marked points. We also asked that $f$ be a homotopy equivalence and that each marked point be circled by its corresponding boundary cycle. See \cite{IgusaReidemeister, GodinGraphComplex} for more details.
\end{remark}

\subsection{Open-closed fat graphs.}

An \emph{open-closed cobordism} $S$ is a surface with boundary. The boundary of $S$ is divided into an incoming $\p_\win S$, an outgoing $\p_\out S$ and a free $\p_\free S$ parts. Hence, $S$ is a cobordism between the 1-manifolds with boundary $\p_\win S$ and $\p_\out S$. The remainder $\p_\free S$ is a cobordism between the boundary of $\p_\win S$ and $\p_\out S$.
The \emph{mapping class group} of such an $S$ is the group
\[\Mod(S) = \pi_0\Diff(S; \p_\win S\du \p_\out S)\]
of isotopy classes of diffeomorphisms of $S$ which fixes the incoming and the outgoing boundary components pointwise.

In this section we will decorate the fat graphs of $\Fat$ to construct a category $\Fat^\oc$ of open-closed fat graphs. The connected components of the geometric realization of $\Fat^{\oc}$ are homotopy equivalent to the classifying spaces of the different open-closed mapping class groups $\Mod(S)$. 

Let $G$ be a graph. A vertex $v$ of $G$ is called a \emph{leaf} if it has valence one, ie if there is a single edge attach to it. We let $V_l\subset V$ be the set of leaf of $G$.
An \emph{open-closed} fat graph is a quadruple
\[\left(\G,\Incoming,\Outgoing, \Closed\right)\]
where $\G=(G,\om)$ is a fat graph and 
\[\Incoming, \Outgoing, \Closed \subset V_l\]
are all subsets of the set of leaves in $G$. We require that 
$\Incoming$ and $\Outgoing$ be disjoint and ordered. We will call the elements of $\Incoming$ \emph{incoming} leaves  and  the elements of $\Outgoing$, \emph{outgoing} leaves.  We will call a leaf \emph{special} if it is either incoming or outgoing. We also require that the closed leaves all be special
\[\Closed\subset \Incoming \cup \Outgoing=:\Special\]
and that each leaf $v$ in $\Closed$ be the only special leaf in its boundary cycle.
The special leaves that are not closed are said to be \emph{open}
\[\Open=\Special\setminus \Closed.\]

This extra data on $\G$ gives $S= \Sigma_\G$ the structure of an open-closed cobordism as follows. We first divide the boundary into the three parts $\p_\win S$, $\p_\out S$ and $\p_\free S$. 
The incoming boundary is made up of the \emph{entire} boundary associated to a closed incoming leaf and a small part 
\[\Big\{[(0,h),(t,h)] \in \p_{h} \quad t\leq 1/2\quad s(h)=v \Big\}
\cup \Big\{[(0,h),(t,h)] \in \p_{h} \quad t\geq 1/2 \quad s(i(h))=v\Big\}\]
of the boundary around an open incoming leaf.

\begin{figure} \begin{center}
\mbox{\epsfig{file=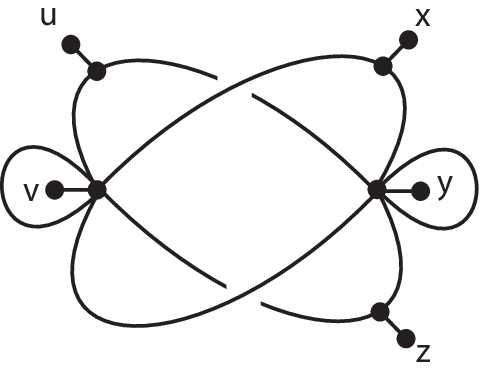, width=150pt}}\qquad
\mbox{\epsfig{file=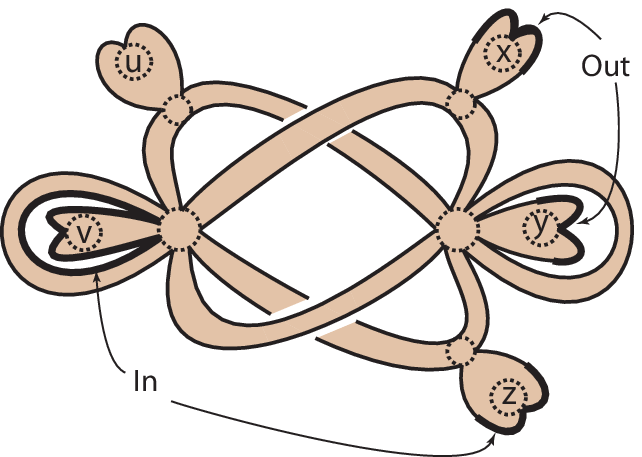, width=200pt}}
\caption{An open-closed fat graph and its associated open-closed cobordism.}\label{fig:openclosed}
\end{center}
\end{figure}
\begin{example}
The fat graph of figure \ref{fig:openclosed} has five leaves
\[V_l=\{u,v,x,y,z\}\]
and four boundary cycles $\om_u$, $\om_v$, $\om_{x,z}$, $\om_y$. 
We let
\[ \Incoming = \{v, z\} \qquad \Outgoing=\{x,y\}\qquad \Closed = \{v\} \qquad \Open=\{x,y,z\}.\]
The open-closed cobordism $\Sigma_\G$ is also illustrated in figure \ref{fig:openclosed}. The incoming and outgoing boundary are thickened while the free boundary is thin. In particular $\Sigma_\G$ is a cobordism between
\[I\du S^1\map I\du I.\]
\end{example}

A morphism of open-closed fat graph 
\[\varphi : \left(\G_0,\Incoming_0,\Outgoing_0,\Closed_0\right) \map \left(\G_1,\Incoming_1,\Outgoing_1,\Closed_1\right) \]
is a morphism of fat graphs $\varphi:\G_0\map \G_1$ which induces order-preserving isomorphism
\[\varphi_\win : \Incoming_0\maplu\cong \Incoming_1\qquad \varphi_\out:\Outgoing_0\maplu\cong \Outgoing_1
\qquad \varphi_\closed:\Closed_0\maplu\cong \Closed_1.\]
Let $\Fatoc$ be the category of open-closed fat graphs with these morphisms.

\begin{theorem} \label{thm:openclosed}
The geometric realization of the category $\Fatoc$ of open-closed
\[|\Fatoc| \simeq \Du_{[S]} B\Mod^{\oc}(S)\]
a disjoint union of classifying spaces of open-closed mapping class group. More precisely, the disjoint union is over all diffeomorphism of open-closed cobordism $[S]$ with ordered incoming and outgoing boundary components.
\end{theorem}
\begin{proof}
First we claim that allowing for univalent and bivalent vertices in Igusa's category does not change the homotopy type of the geometric realization. Consider the categories $\Fat^{1,2}$ and $\Fat^{2}$ whose objects are respectively any fat graph and fat graphs with no univalent vertices. 

Lets first show that removing univalent vertices does not change the homotopy type. There are two functors
\[\Fat^{1,2}\maplu{\xi_1} \Fat^{2} \maplu{\xi_2} \Fat^{1,2}.\]
The functor $\xi_2$ is the natural inclusion.
The functor $\xi_1$ removes all univalent vertices and the unique edge leading to it. And repeats this process until there are no univalent vertices. The composition $\xi_1\cdot \xi_2$ is the identity. There is a natural transformation $\theta$ between the composition the identity on $\Fat^{1,2}$ and $\xi_2\cdot \xi_1$ by collapsing all the edges that $\xi_2\cdot \xi_1$ has removed. 

Lets now shoe that we can remove bivalent vertices as well.There are also two functors
\[\Fat^{2}\maplu{\xi_3}\Fat\maplu{\xi_4} \Fat^2.\]
The functor $\xi_4$ is the natural inclusion. The functor $\xi_3$ takes a fat graph $\G=(G,\sigma=(\sigma_v))$ to the fat graph $\G^{\word{united}}$ obtained from removing all bivalent vertices $v$ and joining the two edges attaching at $v$ into a single edge. Again the composition $\xi_3\cdot \xi_4$ is the identity. It therefore suffices to show that $\zeta= \xi_3$ is a homotopy equivalence.

To show this, we will prove that the geometric realization of the over category $\zeta/\G$ is contractible for each $\G$ in $\Fat$. The objects of $\zeta/\G$ are pairs $(\tilG,\varphi)$ where $\tilG$ is a fat graph with no leaves (but with possibly bivalent vertices) and where
\[\varphi : \zeta(\tilG)\map \G\]
is a morphism of fat graph between two fat graphs with no univalent or bivalent vertices. The morphism 
\[\tilvarphi : (\tilG_0,\varphi_0)\map (\tilG_1,\varphi_1)\]
of $\zeta/\G$ are morphism 
\[\tilvarphi : \tilG_0\map \tilG_1\]
of fat graphs that make the following diagram commute.
\[\xymatrix{
\zeta(\tilG_0)\ar[d]_{\zeta(\tilvarphi)}\ar[r]^{\varphi_0}& \G\\
\zeta(\tilG_1)\ar[ru]_{\varphi_1}.
}\]
Let $\cC_\G$ be the full subcategory of $\zeta/\G$
whose objects are the pairs $(\tilG, \varphi)$ where $\varphi$ is an isomorphism. We claim that $\cC_\G$ is a deformation retract of $\zeta/\G$. Note that the objects of the category $\cC_\G$ give a subdivision of $\G$ by adding bivalent vertices. Let $F$ be the functor 
\[F : \zeta/\G \map \cC_\G\]
which takes an object $(\tilG,\varphi : \zeta(\tilG)\map \G)$ to the fat graph above $\G$ obtained by collapsing every subdivision of every edge collapsed by $\varphi$. This gives both the functor and the natural transformation back to the identity. 

It now suffices to show that for any $\G$, the two functors
\[\xymatrix{
\cC_\G \ar@<-3pt>[r] _{Id}\ar@<3pt>[r]^{cst} & \cC_\G
}\]
where $cst$ is the constant functor which sends anything to $(\G,Id)$ induces homotopy equivalent maps on the realization. To do so, we first pick arbitrary orientations of every edge of $\G$. Following an argument of Igusa \cite[Chapter 8]{IgusaReidemeister}, we then construct an intermediary functor
\[F: \cC_\G\map \cC_\G\]
which sends $(\tilG,\varphi)$ to the fat graph obtained by adding a new edge and a new bivalent vertex at the beginning of every edge of $\G$. (Here we are using the arbitrary choice of an orientation.) We have natural transformation
\[\xymatrix{
Id &\ar[l]_{\nu_1} F \ar[r]^{\nu_2} & cst.
}\]
Here $\nu_1$ collapses the newly added edges while $\nu_2$ collapses every other edge. This shows that adding bivalent and univalent vertices does not change the homotopy type of the geometric realization.

The rest of the proof is a slight variant from \cite{GodinGraphComplex}. Fix an open-closed cobordism $S$. Denote by $\tilde{S}$ the marked surface obtained from $S$ by collapsing each boundary component to a marked point. The maked points are divided into $\marked_{\pecial}$ and $\marked_{\free}$ depending on whether their corresponding boundary contains an incoming or outgoing boundary or whether it does not. (Note that the \emph{entire} boundary must be free for a marked point to be considered free.)

Recall that $\Mod(\tilde{S})$ is the mapping class group of diffeomorphisms which preserve each marked point. Consider the mapping class groups $\Mod^{\free}(\tilde{S})$
associated to the diffeomorphisms which preserves $\free$ as a set and fixes every special marked point. We have a short exact sequence of groups
\[ 0\map \Mod(\tilde{S}) \map \Mod^{\free}(\tilde{S})\map \Sigma(\free) \map 0\]
where $\Sigma(A)$ is the group of permutations on the set $A$. On the classifying spaces, we get a homotopy fibration
\[\xymatrix{
\Sigma(\free) \ar[r]& B\Mod(\tilde S)\ar[r]  &B\Mod^\free(\tilde{S}) 
}\]
Fix an ordering on the marked points of $\tilde{S}$ so that the special ones $p_1,\ldots p_k$ come first. 
Consider the category $\Fat^{\free}_S$ whose objects are fat graphs with a splitting of the boundary cycles into special ones and free ones. We give the special boundary cycles an ordering which is preserved by the morphisms of $\Fat^{\free}_S$. We only take fat graphs so that the surface $\Sigma_\G$ is homeomorphic to S as surfaces and that the number of special boundary cycles in $\G$ is the number of special mark points in $\tilde{S}$. Note that the morphism of $\Fat^{\free}_S$ are allowed to interchange the boundary cycles of $\free$. Hence there is a functor
\[\xi\ :\  \Fat_S \map \Fat^{\free}_S\]
which sends the object $(\G, \om_1,\om_2\ldots \om_n)$ to the same fat graph $\G$ with special cycles 
$\om_1\, \ldots \om_k$ ordered as before. The fiber above each $\G$ is simply a choice of ordering on the remaining boundary cycles and hence is (not naturally)  isomorphic to $\Sigma(\free)$. Any morphism of $\Fat^{\free}_S$ lifts to a bijection between the two corresponding fibers. Hence we get a covering space
\[\xymatrix{
\Sigma(\free)\ar[r]&|\Fat_S|\ar[r]&|\Fat^{\free}_S|.
}\]
Since the two fibrations lift to the universal cover in exactly the same way we get that
\[|\Fat^{\free}_S| \simeq B\Mod^{\free}(\tilde{S}).\]

Now every diffeomorphism of $S$ induces one on $\tilde{S}$ which gives a  group homomorphism
\[\Mod^{\oc}(S)\map \Mod^{\free}(\tilde{S}).\]
Since all these mapping class groups are generated by Dehn twists, this homomorphism is surjective. Its kernel is
generated by Dehn twists around the boundary components. This gives a central extension
\[0\map \bZ^{\pi_0(\p_\pecial S)}\map \Mod^{\oc}(S)\map \Mod(\tilde{S})\map 0.\]
On the classifying spaces, we get a fibration
\[\xymatrix{
 (S^1)^{\pi_0(\p_\pecial S)}\cong B\bZ^{\pi_0(\p_\pecial S)} \ar[r] &B\Mod^\oc(S)\ar[r]&B\Mod(\tilde{S})
}\]
We shall define a functor
\[\Psi : \Fat^{\oc} \map \Fat\]
which will realize this fibration on the categories.

First fix an open-closed cobordism $S$. Pick an ordering on the components $\pi_0\p_\win S$ of the incoming boundary, the components $\pi_0\p_\win S$ and an ordering on the special boundary components $\pi_0\p_\pecial S$ of $S$. 
Let $\Fat^{\oc}_S$ be the category of fat graphs $\G^{\oc}= (\G,\In,\Out,\Closed)$ whose associated open-closed cobordism $\S_{\G^\oc}$ is homeomorphic to $S$ as open-closed cobordism.

Take any open-closed fat graphs $(\G,\In,\Out,\Closed)$ in $\Fat^{\oc}_S$. Since $\In$ is ordered, there is a natural bijection
\[\In\cong \pi_0\p_\win S\]
and similarly, we get
\[\Out\cong \pi_0 \p_\out S.\]
In particular this fixes a bijection between the special boundary components of $S$ and the ones of $\Sigma_\G$. Using the ordering on $\pi_0\p_\pecial S$, we get an ordering of the special boundary cycles of $\G$.

Define $\Psi(\G^{\oc})$ to be the fat graph $\tilG$ obtained from $\G$ by removing all the special leaves in $\In \cup \Out$ and the unique edge leading to them. If that special edge was attached to a trivalent vertex, we remove this (now bivalent) vertex. The special boundary cycles of $\tilG$ are the ones corresponding to special boundary  cycles in $\G$ and these are ordered.

Fix a fat graph $\tilG$ in $\Fat_S$. Lets study the category $\Psi/\tilG$. An object of $\Psi/\tilG$ is a pair 
\[\left(\G^{\oc}, \varphi:\Psi(\G^\oc)\map \tilG \right).\]
As before, we define a functor
\[ J : \Psi/\tilG \map \Psi^{-1}(\tilG)\]
by sending $(\G^{\oc}, \varphi)$ to the open-closed fat graph obtained by collapsing in $\G^{\oc}$ the edges collapsed in $\Psi(\G^{\oc})$ by $\varphi$. It is easy to show that $J$ is adjoint to the inclusion
\[I: \Psi^{-1}(\tilG)\map \Psi/\tilG\]
and in particular, to use Quillen's theorem B, it suffices to show that any morphism
\[\psi : \tilG_1\map \tilG_2\]
of $\Fat_{S}$ induces an homotopy equivalence
\[\psi^* :|\Psi^{-1}(\tilG_1) |\map |\Psi^{-1}(\tilG_2)|.\]
To do this we will show that for any $\tilG$  in $\Fat_S$
\[|\Psi^{-1}(\tilG)| \simeq \left(S^1\right)^{\p_\pecial S}.\]
Since the special boundary cycle of $\tilG$ are ordered, they correspond to special boundary component of $S$. To get an object in $\Psi^{-1}(\tilG)$, we need to attach leaves on each of these boundary cycles one for each incoming or outgoing parts. For example,  if the \ith boundary component of $S$ has four special intervals on it which alternate between incoming and outgoing as someone walks along the boundary. Say that these intervals are number 2nd, 2nd, 3rd, 3rd (recall that incoming and outgoing are ordered separately). Then we need to pick along the \ith boundary of $\tilG$ four attaching points for the leaves 
\[l_2^\win \quad l_2^\out \quad l_3^{\win} \quad l_3^\out.\]
This corresponds to choosing fourf cyclically ordered points on the circle of $\om_i$. 
This choice is therefore, up to homotopy, a circle. 
\end{proof}

\subsection{Admissible fat graphs}

Following the definition of the Sullivan chord diagrams in \cite{CohenGodin},  we now restrict the shape of the graph underlying an open-closed fat graph. These \emph{admissible} fat graphs will be used to define the higher genus string operations in the following sections. Note that this model is quite similar to the model used in \cite{CostelloCalabiYau}. 

\begin{definition}
An open-closed fat graph 
\[\G^{\oc}= (\G, \In, \Out,\Closed) \]
is \textbf{admissible} if its incoming boundary cycles are embedded in $\G$. 
\end{definition}

More precisely each incoming leaf $v$ in $\In\ \cap\ \Closed$, which represent a closed incoming boundary component of $\Sigma_\G$, correspond to some boundary cycle of $\G$ and hence to a map
\[f_v : S^1 \map |\G|.\]
We get a map
\[\xymatrix{
\DU{v\in \In\setminus \Closed} \{v\} \du \DU{v\in \In\cap \Closed} S^1\ar[rrr]^-{\Du Id \du \Du f_v}&&&|\G|.
}\]
The open-closed fat graph $\G^{\oc}$ is admissible if and only if this map is an inclusion.

\begin{figure} \begin{center}
\mbox{\epsfig{file=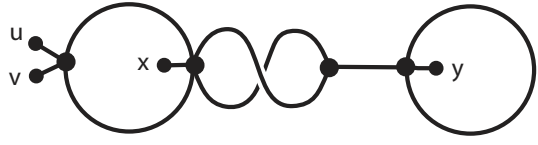, width=200pt}}
\caption{A fat graph}\label{fig:admissible}
\end{center}
\end{figure}

\begin{example}
Let $\G$ be the fat graph of figure \ref{fig:admissible}. An open-closed fat graph $\G^\oc=(\G,\In,\Out,\Closed)$ is admissible if and only if
\[\In\cap \Closed \subset \{x,y\}.\]
\end{example}

Let $\aFat$ be the full subcategory of $\Fatoc$ whose objects are all admissible fat graphs. For any open-closed cobordism $S$, we denote by 
\[\aFat_S = \aFat\cap \Fatoc_S\]
the subcategory of $\aFat$ whose open-closed fat graph $\G^{\oc}$ have a surface $\Sigma_{\G^{\oc}}$ homeomorphic to $S$ as an open-closed cobordism. 

\begin{theorem} 
For any open-closed cobordism $S$ whose boundary is not completely incoming, the inclusion
\[\aFat_S \map \Fatoc_S\]
is a homotopy equivalence.
\end{theorem}
\begin{proof}
It suffices to show this result for connected surfaces. So let $S$ be a connected open-closed cobordism. 
Say the surfaces $S$ has $p$ closed incoming boundary components, $q$ other boundary components and genus $g$. By the assumption, $q$ is at least one.

We will first prove a similar statement for the mapping class group of the surface $S^\prime$ obtained from $S$ by collapsing each boundary component to a marked point. 

As before let $\Fat^2_{q,g}$ be the category of fat graphs $\G$ with $q$ ordered boundary cycles and whose surface $\Sigma_\G$ has genus $g$. As in the previous section, we allow bivalent and univalent vertices. We already know that
\[|\Fat^2_{q,g}| \simeq  B\Mod(S_{q,g})\]
where $S_{q,g}$ is a connected surface with $q$ marked points and with genus $g$.
Lets also consider the category $\cF^p_{q,g}$ of pairs $(\G,\{v_1,\ldots v_p\})$ where $\G$ is a fat graph in $\Fat_{q,g}$ and where 
\[\{v_1,\ldots, v_p\}\subset V\]
is an ordered set of $p$ distinct vertices in $\G$. The morphisms of $\cF^p_{q,g}$ are morphisms of fat graphs which sends each marked vertex to the corresponding marked vertex.

A diffeomorphisms of $S_{p+q,g}$ which fixes $p+q$ marked points does fix $q$ points. We therefore get a short exact sequence
\[\Diff(S_{p+q,g})\map \Diff(S_{q,g}) \maplu{ev_p} C_p(S_{q,g})\]
where $C_p(S_{q,g})$ is the space of $p$ distinct and ordered points on $S$. Since $q>0$, we have that the connected components of $\Diff(S_{p+q,g})$ are trivial and hence
\[\Diff(S_{p+q,g})\simeq \Mod(S_{p+q,g}).\]
We therefore get a fibration
\[\xymatrix{
\C_p(S_{q,g}) \ar[r]& B\Mod(S_{p+q,g}) \ar[r]& B\Mod(S_{q,g})}\]

Consider the functor
\[\xi : \cF^p_{q,g} \map \Fat_{q,g}\]
which forgets the $p$ marked vertices. It also removes any special vertex that was a leaf and the unique edge attached to it. Fix a fat graph $\G$ in $\Fat_{q,g}$. As before we can argue that the homotopy fibre of the geometric realization of $\xi$ is the geometric realization of the category $\xi^{-1}(\G)$. Now an object in this category corresponds to a choice of $p$ different and ordered marked vertices. However this vertices may be on the graph $\G$ itself or on a leaf attached to it. We are therefore picking $p$ points around the graph. Since a neighborhood of the graph is homeomorphic to $\Sigma_\G$, we get that 
\[|\xi^{-1}(\G)| \cong C_p(S_{q,g}).\]
Since both fibration have the same holonomy, we get that
\[|\cF^p_{q,g}| \simeq B\Mod(S_{p+q,g}).\]

Now consider the category $\cF^a_{p+q,g}$ of fat graph $\G$ in $\Fat_{p+q,g}$ so that the first $p$ boundary components are disjoint circles. There is a functor
\[\Psi: \cF^a_{p+q,g} \map \cF^p_{q,g}\]
which transforms the $p$ circles into $p$ marked vertices. We claim that $\Psi$ induces an homotopy equivalence on the geometric realization. Fix a fat graph $\G$ with $p$ marked points. Again, the homotopy fibre of $\Psi$ above $\G$ is simply the geometric realization of the category $\Psi^{-1}(\G)$.

\begin{figure} \begin{center}
\mbox{\epsfig{file=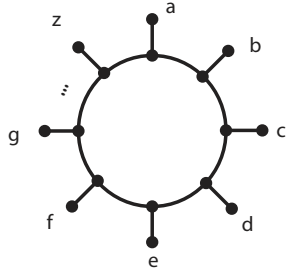, width=105pt}}
\caption{The universal element of $\cC((a,b,c,d,e,f,g \ldots z))$}\label{fig:universal}
\end{center}
\end{figure}
An object in $\Psi^{-1}(\G)$ is an opening of each marked vertex into a circle. Hence for each of these marked vertices $v$, we only need to pick a circle and where each half-edge adjacent to $v$ lands. These choices are independent of each other hence
\[\Psi^{-1}(\G) \cong \prod_{v \in\Marked} \cC(s^{-1}(v)) \]
is a product of $p$ categories. For any cyclically ordered set $A$, we let $\cC(A)$ be the category of graphs $G$ which are oriented circle with leaves cyclically labeled by $A$.
Using the argument of theorem \ref{thm:openclosed}, we can show that the subcategory $\cC(A)_{\geq 3}$ where each non-leaf vertex has at least valence three is a deformation retract. Now this subcategory has an initial object, namely the graph where each vertex has exactly order 3. This graph is shown in figure... In particular 
\[|\cC(A)| \simeq |\cC(A)_{\geq 3}| \simeq *\]
are both contractible which means that so is $\Psi^{-1}(\G)$. Hence $\Psi$ induces a homotopy equivalence and 
\[|\cF^{a}_{p+q,g} | \simeq B\Mod(S_{p+q,g}).\]

Now we can reapply the whole argument of theorem \ref{thm:openclosed} with the category $\cF^a_{p+q,g}$ replacing $\Fat$. This gives the desired result.
\end{proof}

\subsection{Gluing fat graphs}

\label{sub:gluing graphs}
We now will show how to see the gluing of surfaces at the category level. Say we have two open-closed cobordisms $S_1$ and $S_2$ and say we have an identification $\p_\out S_1$ and $\p_\win S_2$. We would like to have a functor
\[ \aFat_{S_1} \times \aFat_{S_2} \map \aFat_{S_1\# S_2}\]
which would model the map
\begin{equation}\label{gluing}
B\Mod^\cB(S_1)\times B\Mod^{\cB}(S_2)\map B\Mod^{\cB}(S_1\# S_2).\end{equation}
However to glue two fat graphs, we would need to subdivide the corresponding boundary cycles to a matching size and then identify them. For general fat graphs, such a subdivision might be infinite. For admissible fat graphs, they are finite and come in a contractible family, but there is no natural one. 

A simple way to remedy this problem is to build a gluing that is only partially defined. More precisely, we will construct a subcategory
\[\cG_{S_1,S_2} \subset \aFat_{S_1} \times \aFat_{S_2}\]
of gluable fat graphs.  
On the geometric realization, we get a homotopy equivalence
\[|\cG_{S_1,S_2}| \subset |\aFat_{S_1}| \times |\aFat_{S_2}|.\]
We also define a gluing functor
\[\cG_{S_1,S_2} \map \aFat_{S_1\# S_2}.\]
which models \eqref{gluing}.

The objects of $\cG$ are pairs $(\G,\tilG)$ of admissible fat graphs in $\aFat_{S_1}\times \aFat_{S_2}$ so that for each closed outgoing boundary of $S_1$, the corresponding boundary cycles of $\G$ and $\tilG$ have  \emph{exactly} the same number of edges. We want to think of these boundary cycles as glued. In particular if
\[\om_i^{\out} = (L\ol{L} A_1\ldots A_k) \qquad \tilom_i^{\in}=(L\ol{L}B_1\ldots B_k)\]
then our morphisms will act as if $B_i$ and $\ol{A}_{k-i}$ are already identified. In particular, the morphisms of $\cG$ are morphisms 
\[(\varphi,\tilvarphi):(\G_1,\tilG_1)\map (\G_2,\tilG_2)\]
which  collapse $B_i$ if only if it also collapse $\ol{A}_{k-i}$

We define a functor
\[\#: \cG_{S_1,S_2} \map \aFat_{S_1\# S_2}\]
by gluing the admissible fat graphs $(\G,\tilG)$ along their identified edge. To make this definition more precise, let $eV_{\tilG}$, $eE_{\tilG}$ and $eH_{\tilG}$ be respectively the set of all vertices, edges and half-edges of $\tilG$ which are not part of an incoming boundary cycles. The admissible fat graph $\G\#\tilG$ will have vertices, edges and half-edges as follows.
\begin{eqnarray*}
V_\# &=& V_\G \du eV_{\tilG}\\
E_\# &=& E_\G \du eE_{\tilG}\\
H_\#&=& H_\G\du eH_{\tilG}.
\end{eqnarray*}
If $V^\win_\tilG$ denotes the set of vertices which are in fact part of an incoming boundary cycles, we have a map
\[V^\win_\tilG \map V\]
which sends an incoming vertex $s(B_i)$ of $\tilG$ to the vertex  $s(\til{A}_{k-i})$.
The admissible fat graph structure on $\G\#\tilG$ is determined by the fact that it's incoming boundary cycles are simply the incoming boundary cycles of $\G$ and it's other boundary cycles are obtained from the other boundary cycles of $\tilG$ by replacing the $\ol{B}_i$ by $A_{k-i}$.

\begin{lemma}
The inclusion 
\[\cG_{S_1,S_2}\map \aFat_{S_1}\times \aFat_{S_2}\]
induces a homotopy equivalence and the functor $\#$ realizes to the gluing of surfaces.
\end{lemma}
\begin{proof}
We can assume that $S_2$ has at least one incoming boundary circle since otherwise there are no restrictions.

We will consider the functor
\[\Psi : \cG_{S_1,S_2}\map (\aFat_{S_1})^{\neq 2} \times (\aFat_{S_2})^{\neq 2} \]
from the category of gluable admissible fat graphs to the category of pairs of fat graphs with no bivalent vertices. The functor $\Psi$ sends a pair $(\G,\tilG)$ to the pair $(\Psi(\G),\Psi(\tilG))$ where $\Psi(\G)$ is obtained from $\G$ by removing all bivalent vertices and joining edges that are separated by bivalent vertices. A morphism $\Psi(\varphi,\tilvarphi)$ will collapse an edge $A$ of $\Psi(\G)$ only when $\varphi$ collapses every edge of $\G$ that is part of $A$. We will show that the functor $\Psi$ realizes to a homotopy equivalence and hence that $\cG_{S_1,S_2}$ has the right homotopy type. 

Fix a pair of fat graphs $(\G_0,\tilG_0)$ in $(\aFat_{S_1})^{\neq 2} \times (\aFat_{S_2})^{\neq 2}$. We will prove that the category $\cC=\Psi/(\G_0,\tilG_0)$ is contractible. An object in $\cC$ is a pair of fat graphs $(\G,\tilG)$ with morphisms
\[ \varphi:\Psi(\G)\map \G_0 \qquad \tilvarphi:\Psi(\tilG)\map \tilG_0.\]
Such an object is determined by the following information. We have a pair $(\Psi(\G), \Psi(\tilG))$ of fat graphs with no bivalent vertex. We also have a way of adding bivalent vertices to get $\G$ and $\tilG$ so that the outgoing boundary of $\G$ is identified to the incoming boundary of $\tilG$. 

Say $\tilom_i$ is an incoming boundary circle of $\tilG$ and let $\Phi(\tilom_i)$ be its corresponding boundary cycle in $\Phi(\tilG)$. Then $\tilom_i$ is a subdivision of $\Phi(\tilom_i)$ and corresponds to a boundary cycle $\om_j$ of $\G$. The extra bivalent vertices in $\tilom_i$ are of two types. We will call these bivalent vertices \emph{superfluous} if they correspond to bivalent vertices in $\G$ as well. The other bivalent vertices will be called \emph{necessary}.

We will first get rid of all superfluous bivalent vertices one boundary cycle at a time. Note that $\tilom_i$ is an oriented circles. Using this we get two functors
\[F, G : \cC\map \cC.\]
The functor $F$ adds an edge right after each necessary vertex in $\tilom_i$. The functor $G$ replaces the edges between two necessary vertex by a single edge. There are  two natural transformations
\[\xymatrix{
Id_{\cC} &\ar@{=>}[l] F \ar@{=>}[r] &G
}.\]
The first one collapses the newly added edge while the second collapses all other edges.

We will then move all the bivalent vertices to the first edge of each $\tilom_i$ present in $\G$. Moving these subdivision is strangely complicated at the category level. However it becomes child's play when we consider metric fat graphs. With this in mind, we will briefly construct the simplicial space of metric fat graphs and relate it to our model. 

A admissible metric  fat graph is a pair $(\G,\lambda)$ where $\G$ is an adimissble fat graph and
\[\lambda : E_\G\map \bR^{\geq 0}\]
assigns a length to each edge. We ask that any cycle of $\G$ has positive length, that 
\[\sum_{e\in E_\G} \lambda(e) =1\]
and that fat graph
\[\frac{\G}{\{e\quad \lambda(e)=0\}}\]
obtained by collapsing all edges of length zero still be an admissible fat graph. For a fixed $\G$, the allowable $\lambda$'s form a subset $D_\G$ of the simplex $\Delta^{E_\G}$. We build a space of admissible metric fat graphs by taking
\[\cM\aFat = \frac{\Du_{\G} D_\G}{\sim}\]
where we identify two $(\G,\lambda)$ and $(\tilG, \tillambda)$ if and only if there is an isomorphism
\[\varphi : \G/\{e\ \lambda(e)=0\}\map \tilG/\{e\ \tillambda(e)=0\},\]
so that
\[\lambda(e) = \tillambda(\varphi(e)).\]
Note that since we have removed part of $\Delta^{E_\G}$, the space $\cM\aFat$ is not a simplicial space. However, if we first take the simplicial subdivision of $\cM\aFat$ and then keep only the subdivided simplices which are complete, we get a homotopy equivalent subspace
\[\widetilde{\cM\aFat} \subset \cM\aFat.\]
See \cite{CV86} for a proof that for any \emph{ideal} simplicial space $\cS$, the inclusion
\[\widetilde{\cS} \map \cS\]
is a deformation retract.

A simplex of $\widetilde{\cM\aFat}$ corresponds exactly to a simplex
 \[\G_0\smap \cdots \smap \G_{n}\]
 in the nerve of the category $\aFat$. In fact, there is an inclusion
\[\mu : |\Fat| \map \cM\Fat\]
hitting exactly $\widetilde{\cM\Fat}$.
This inclusion sends a point
\[\left(\G_0\maplu{\varphi_{0,1}} \cdots \maplu{\varphi_{n-1,n}} \G_n, 0\leq t_1\leq\cdots\leq t_n\leq 1\right)\in |\aFat| \]
to the metric fat graph $(\G_0,\lambda)$ where we take $\lambda$ to be the normalized version of 
\[\lambda^{\prime} (e) =\begin{cases}
1 & \varphi_{0,n}(e):=\varphi_{0,1}\cdots \varphi_{n-1,n}(e) \in E_{\G_n}\\
t_{k+1}&\varphi_{0,k}(e) \in E_{\G_k}\qquad \varphi_{0,{k+1}}(e)\in V_{\G_{k+1}}\\
t_1 &\varphi_{0,1}(e)\in V_{\G_1}
\end{cases}.\]

Now in the metric fat graphs setting, we take
\[\cM\cG_{S_1,S_2} \subset \cM\aFat_{S_1}\times \cM\aFat_{S_2}\]
to contain the pairs $\left((\G,\lambda),(\tilG,\tillambda)\right)$ where the corresponding boundary cycles of $\G$ and $\tilG$ have the same number of edges and each on of these edge has the same length as its corresponding one. Now, by the preceding argument
\[\xymatrix{
|\cG_{S_1,S_2}|\ar@{<=>}[r]\ar[d]& \widetilde{\cM\cG_{S_1, S_2}} \ar[r]^{\simeq}\ar[d]& \cM\cG_{S_1,S_2}\ar[d]\\
|\left(\aFat_{S_1}\right)^{\neq 2}|\times |\left(\aFat_{S_2}\right)^{\neq 2}|\ar@{<=>}[r]& \widetilde{\cM\left(\aFat_{S_1}\right)^{\neq 2}}\times \widetilde{\cM\left(\aFat_{S_2}\right)^{\neq 2}}\ar[r]^{\sim}& \cM\left(\aFat_{S_1}\right)^{\neq 2}\times \cM\left(\aFat_{S_2}\right)^{\neq 2}
}\]
It therefore suffices to prove that the map
\[\cM\Psi: \cM\cG_{S_1,S_2}\map \cM\left(\aFat_{S_1}\right)^{\neq 2}\times \cM\left(\aFat_{S_2}\right)^{\neq 2}\]
is a homotopy equivalence.

Now in this setup, we consider can move the bivalent vertices continuously to the desired edge.
Finally we can now collapse in the subdivided world what is collapsed by $\varphi$ and $\tilvarphi$.
\end{proof}

\subsection{A combinatorial cobordism category}

We can use this partially defined gluing to get a combinatorial version of a cobordism category. This is the category that Tillmann used in \cite{TillmannInfinite} to show that the stable mapping class group has the homology of an infinite loop space.

Consider the following partial 2-category $\mathit{Cobor}^{\Fat}$. The objects of $\mathit{Cobor}^{\Fat}$ are isomorphism classes of open-closed 1-manifolds. The morphism between any two such $P$ and $Q$ is the category 
\[\aFat_{P,Q} = \du_{[S]} \Fat_{S}\]
 of admissible fat graphs $\G$ which represents a cobordism $S$ between $P$ and $Q$.
 
 The composition in $\mathit{Cobor}^\Fat$ is the gluing of admissible fat graphs described in the previous section. In particular, it only partially defined.

  \section{Construction of the Thom-Pontryagin collapse}
\label{sec:TP}

In this section we start the construction of the string operations using the model $\aFat$ of the previous section. For each open-closed cobordism $S$, we will construct a map of spectra
\begin{equation}\label{eqn:thom col}
\Sigma^{\infty} M^{\p_\win S} \times B\Mod^{\oc}(S) \map \Thom(\kappa_S)
\end{equation}
Here $\kappa_S$ is a virtual bundle over a space homotopy equivalent to
\[E\Diff(S;\p_\win S\du \p_\out S) \underset{\Diff(S)}{\times} \Map(S; M).\]

We will start constructing \eqref{eqn:thom col} over the objects of $\aFat$. In section \ref{sub:TP admissible}, we will build homotopically-trivial spaces $\Tub(\G)$ which parameterize maps of spectra 
\[\Sigma^{\infty} M^{\p_\win\G} \map \Thom(\kappa_\G).\]
Here $\kappa_\G$ is a virtual bundle above the space $\Map(\G,M)$. 
We will then construct the map \eqref{eqn:thom col} over the simplices of $\aFat$. More precisely, for each simplex
\[\sigma = \G_n\smap \cdots \smap \G_0\]
of $\aFat$, we will built in section \ref{sub:second thickening}, 
a space $\Tub(\sigma)$ whose points parameterize compatible maps of spectra for each of the $\G_k$'s.
We will then show how to include these choices in our model $\aFat$ without changing its homotopy type in section \ref{sub:spaces}.
Finally we will construct the map of \eqref{eqn:thom col} in section \ref{sub:construction}.

In section \ref{sec:orientation} and \ref{sec:gluing operations} we will discuss orientability, Thom isomorphism and gluing of the operations.

\begin{remark}
This section is heavy on homotopy limits and colimits. For the uninitiated, we recommend \cite{BousfieldKan} as a reference. Also, the construction of the spaces $\cT_{\sigma}(\sigma_0)$ and $\Tub(\sigma)$ are quite technical and should probably be skipped on a first reading of this section. The results of proposition \ref{prop:tub} are sufficient for the construction of the desired generalized Thom collapse.
\end{remark}

%=====================Objects
\subsection{Thom-Pontryagin collapse for one admissible fat graph.}
\label{sub:TP admissible} \label{sub:objects}
Let $\G$ be an admissible fat graph. In this section, we will define a contractible family of generalized Thom collapses
\begin{equation}\label{eqn:gen thom G}
 M^{\p_{\win} \G } \times W_\G \map \Thom(\kappa_\G)
\end{equation}
Here $\kappa_\G$ is a virtual bundle above $M^{\G}$, $W_\G$ is a Euclidean space
and $\p_\win\G$ is the subgraph of $\G$ which consists of only the incoming closed circles and the incoming open leaves. Each of the maps of \eqref{eqn:gen thom G} gives a map of spectra
\[\Sigma^{\infty} M^{\p_\win \G} \map \Thom(\left(\kappa_\G,W_\G\times M^\G\right))\]
between the suspension spectrum of $M^{\p_\win \G}$ and the Thom spectra of the virtual bundle
\[\kappa_\G - \left(W_\G\times M^\G\right)\map M^{\G}.\]

We first need to fix some notation. Let $V_\win$, $H_\win$ and $E_\win$ be the vertices, half-edges and edges that are part of the incoming boundary circles or are the incoming  leaves of $\G$. Let $eV$, $eH$ and $eE$ be the remaining (extra) vertices, half-edges and edges. Let $\G^{div}$ be the fat graph obtained from $\G$ by forgetting the attachment of the extra-edges to their endpoints and adding a vertex at the middle of each extra edge. Let
\[\rho_{\infty} : M^\G\map M^{\p_\win \G}\times PM^{eE}\times M^{eV} \]
be the embedding induced by the map of graphs $\G^{div}\map \G$ which reattach the extra-edges. Fix an embedding $f:M\map W$  of $M$ into an Euclidean space $W=\R^k$ and denote its normal bundle by $\nu$. Note that as virtual bundles $\nu\cong (W\times M,TM)$. 

We will denote by $\kappa_\G$ the pull-back bundle of
\[\xymatrix{
\kappa_\G\ar[r]\ar[d]& TM^{eH}\times \nu^{eV\du eE}\ar[d]\\
M^\G\ar[r]&M^{eH}\times M^{eV\du eE}.
}\]
Here the bottom arrow sends a map $\gamma:\G\map M$ to its value at each vertex in $eV$, at the source of every extra half-edge and at the middle point of each extra edge.

We are ready to construct the space parameterizing the maps \eqref{eqn:gen thom G}. Consider the diagram
\begin{equation}\label{dia:dim0}
\xymatrix{
M^\G \ar[r]^-{\rho_{\infty}}& M^{\p_\win \G}\times PM^{eE} \times M^{eV}&\ar[l]^-{cst} M^{\p_\win \G} \times M^{eE\du eV} \ar[rr]^-{Id\times f^{eE\du eV}}&& M^{\p_\win \G} \times W^{eE\du eV}.
}\end{equation}
where the map $cst$ includes $M^{eE}$ into $PM^{eE}$ as the constant paths. We want to take  the composition of a Thom collapse of $Id\times f^{eE\du eV}$ with $cst$ and then with a Thom collapse for the embedding $\rho_\infty$. Of course there are many choices involved in defining the two Thom collapses. In the next proposition we narrow these choices down to a contractible family $\cT(\G)$.
 To construct $\cT(\G)$, we will need the following concept of a propagating flow to lift tubular neighborhoods of a finite dimensional embedding $\rho_{fin}$ to the infinite embedding $\rho_\infty$.

\begin{definition}
A \emph{propagating flow} for a bundle $\lambda$ is a function
\[ \Omega : \lambda \map \cX_c^{v}(\lambda)\]
from $\lambda$ to the space $\cX_c^v$ of compactly supported vertical vector fields in $\lambda$ so that 
\[\Omega_{\textbf{x}}( \textbf{z}) = \begin{cases}
\textbf{x}& \textbf{z}=t\textbf{x} \quad t\in [0,1]\\
0& \textbf{x}=0. 
\end{cases}\]
Denote the space of such propagating flows by $ \cP(\lambda)$. 
\end{definition}

Note that if we start at the zero vector in the same fiber as $\textbf{x}$ and flow according to $\Omega_{\textbf{x}}$ for time 1, we end up at $\textbf{x}$. 

\begin{proposition}
There is a contractible space $\cT(\G)$ and a map
\[\cT(\G)\times M^{\p_{\win}\G} \times W^{eE\du eV} \map Thom(\kappa_\G).\]
\end{proposition}
\begin{proof}
We define $\cT(\G)$ as the following product of spaces. We first include a choice of a tubular neighborhood 
\[\varphi : \nu^{eE\du eV}\map W^{eE \du eV}\]
for 
\[f^{eE\du eV} : M^{eE\du eV}\map W^{eE\du eV}\]
which gives one for the right-most map of \eqref{dia:dim0}. 
We then pick a tubular neighborhood 
\[\psi_{fin} : \lambda_{fin} \map M^{V_{\win}}\times M^{eH}\times M^{eV}\]
for the restriction $\rho_{fin}$ of $\rho_\infty$ to the vertices of $\G^{div}$.  
We also pick a propagating flow for the normal bundle $\lambda_{fin}$ of $\rho_{fin}$. 
Finally we pick a connection $\nabla$ on the bundle
\[\nu^{eE\du eV}\map  M^{eE\du eV}.\]
Let $\cC(\nu^{eE\du eV})$ be the space of such connections.
Hence we let
\[\cT(\G) = \Tub(f^{eE\du eV})\times \Tub(\rho_{fin}) \times \cP(\lambda_{fin})\times \cC(\nu^{eE\du eV}).\]
Here $\Tub(A\smap B)$ is the space of tubular neighborhood of the embedding $A\smap B$ as defined in the appendix \ref{app:tubular}.  By proposition \ref{prop:tub contractible} of the appendix, the first two spaces involved in the construction of $\cT(\G)$ are contractible. The last two spaces are convex and hence it suffices to show that $\cP(\lambda_{fin})$ is not empty. This will be proved in lemma \ref{lem:propagating}.

Fix an element in $\xi=(\varphi, \psi, \Omega,\nabla)$ in $\cT(\G)$. Lets now construct the Thom collapse associated to $\xi$. The tubular neighborhood $\varphi$ induces a Thom collapse
\[M^{\p_\win \G} \times W^{eE\du eV} \map Thom(M^{\p_\win \G}\times \nu^{eE\du eV}).\]
in a continuous way. 
We then consider the inclusion $cst$ of $M$ into $PM$ and the evaluation $ev_{1/2}$ of a path at its half-way point to get maps of bundles
\[\xymatrix{
\nu \ar[r]^-{cst}\ar[d]& \nu_{PM} =ev_{1/2}^*\nu \ar[r]\ar[d]& \nu\ar[d]\\
M\ar[r]^-{cst}&PM \ar[r]^{ev_{1/2}} &M
}\]
Since $cst$ is a map of bundle between $\nu$ and $\nu_{PM}$, it defines a map of Thom spaces
\[Thom\left(M^{\p_\win \G}\times \nu^{eE\du eV}\right) \map Thom\left(M^{\p_{\win} \G}\times \nu_{PM}^{eE}\times \nu^{eV}\right).\]
We now have 
\[M^{\p_\win \G}\times W^{eE\du eV}\map \Thom\left(M^{\p_\win\G}\times \nu^{eE}_{PM}\times \nu^{eV}\right).\]

We then use $\psi$ and $\Omega$ to construct a diffeomorphism
\[ \psi_{\infty} : \lambda_{\infty} \maplu{\cong} ev_{\G^{\div}}^{-1}(\word{Image}(\psi)) \subset M^{\p_\win \G}\times {PM}^{eE}\times M^{eV}\]
and hence a tubular neighborhood for $\rho_\infty$.
Using $\psi$ we send the vector fields of $\Omega$ to vector fields on $\word{Image}(\psi)\subset M^{V_{\win}}\times M^{eH}\times \nu^{eE}_{PM}\times \nu^{eE}$. Since these vector fields are compactly supported, we can then extend them to 
\[\xymatrix{
Z : \lambda_{fin} \ar[r]&\cX^v_c(\lambda_{fin}) \ar[r]&\cX_c(M^{V_\win}\times M^{eH}\times M^{eV})
}\]
Now for any element $a\in V_{\win}\du eH\du eV$ we have a replacement function
\[\word{rep}_a : \left(M^{V_\win\du eH\du eV}\right) \times M \map M^{V_\win\du eH\du eV}\]
which sends $((m_b),m)$ to the element obtained by replacing $m_a$ by $m$. Mixing this with $Z$, we get
\[ \xymatrix{
\left(M^{V_\win\du eH\du eV}\right)\times M \ar[r]^{\word{rep}_a} \ar@/_20pt/[rrr]_{Z_x^v}& \left(M^{V_\win\du eH\du eV}\right)\ar[r]^{Z_x} 
&TM^{V_\win \du eH\du eV}\ar[r]^-{ev_a}& TM
}\]

Note that $\rho_{fin}$ fits in the following pull-back diagram.
\[\xymatrix{
M^{\G}\ar[r]_-{\rho_{\infty}}\ar[d]^{ev_\G}& M^{\p_\win}\times {PM}^{eE}\times M^{eV}\ar[d]^{ev_{\G^{div}}}\\
M^{V} \ar[r]_-{\rho_{fin}}&M^{V_\win} \times M^{eH}\times M^{eV}
}\]
\begin{comment}
The normal bundle $\lambda_{fin}$ of $\rho_{fin}$ is the pullback of 
\[\xymatrix{
\lambda_{fin} \ar[rr]\ar[d]^{\pi_{fin}} && TM^{eH}\ar[d]\\
M^{V} \ar[rr]^{M^{s(-)}} &&  M^{eH}
}\]
where the bottom map is induced by the source map $eH\smap V$. 
\end{comment}
Hence the normal bundle of $\rho_\infty$ is the pullback
\[\lambda_{\infty} = ev^{-1}_\G (\lambda_{fin}).\]
Lets construct the tubular neighborhood
\[\psi_{\infty}: \lambda_\infty\map M^{\p_\win\G}\times \nu_{PM}^{eE}\times \nu^{eV} = M^{\G^{div}}.\]
Take an element 
\[\left(\gamma :\G\map M,\, \textbf{x}\right) \qquad \textbf{x} \in \lambda_{fin}|_{ev_\G(\gamma)}\]
of $\lambda_\infty$. 
We let the tubular neighborhood $\psi_{fin}$ determine the value of $\psi_{\infty}$ on the vertices of ${\G^{div}}$.
\[\psi_{\infty}(\gamma,x)(v) = \psi_{fin}(ev(\gamma),x)(v)\]

Here is a convoluted way of picturing this. Let $\sigma_t \in M^{V_{\win}}\times M^{eH}\times \nu^{eE}\times \nu^{eV}$ be the flow which starts at $ev(\gamma)$ and follows the $Z_\textbf{x}$. The path $\sigma_t$  is the unique solution to the ODE
\begin{eqnarray*}
\sigma_0 &=& ev (\gamma)\\
\frac{d}{dt} \sigma_t &=& Z_{x} (\sigma_t)
\end{eqnarray*}
Note that by our choice of $\Omega$, this $\sigma_0$ is a path that ends at $\psi_{fin} (ev(\gamma), x)$. We have therefore
\[\psi_\infty(\gamma,x)(v)  =\sigma_1(v)\]
for any vertex $v\in V\du eH$.

We will use a similar ODE to pick the value of $\psi_{\infty}$ on edges. Pick any edge $e$ of $\G^{div}$. Say the vertices of $e$ are $v_0, v_1 \in V\du eH$. We already know that these end points  $v_0$ and $v_1$ will follow the flow $Z_x$. The path between these two values will follow a slight variant of the flow $Z_x$. More precisely, let $\sigma_t(s,e)$ be the unique solution to the following ODE.
\begin{eqnarray*}
\sigma_0(s,e) &=& \gamma(s,e)\\
\frac{d}{dt} \sigma_t(s,e) &=& (1-s) Z_x^{v_0}\left( (\sigma_t(v)), \sigma_t(s,e)\right) + s Z_{x}^{v_1}\left((\sigma_t(v)),\sigma_t(s,e)\right)
\end{eqnarray*}
We now define 
\[\psi_\infty(\gamma,\textbf{x})(e, s)= \sigma_{1}(e,s).\]
Note that when $s=0$, the ODE becomes
\begin{eqnarray*}
\sigma_0(0,e) &=& \gamma(0,e)= \gamma(v_0)\\
\frac{d}{dt} \sigma_t(0,e)&=& Z_x^{v_0}\left((\sigma_t(v)), \sigma_t(0,e)\right)
\end{eqnarray*}
which is satisfied by $\sigma_t(0,e) = \sigma_t(v_0)$. A similar argument shows works for $s=1$ and hence we have indeed defined a map
\[\psi_{\infty} : \lambda_{\infty} \map M^{\G^{div}}.\]

Tthis function gives a homeomorphism
\[\psi_\infty : \lambda_{\infty} \map ev^{-1}(\word{Image}(\psi_{fin})).\]
Fx any $\theta$ in $ev^{-1}(\word{Image}(\psi_{fin}))$. Let $x$ be the unique element of $\lambda_{fin}$ with
\[\psi_{fin}(x) =ev(\theta).\]
Since we have an $x$, we have both the starting value at the vertices and the vector field $Z_x$. We also get the value of $\sigma_t(v)$ by flowing along the $Z_x$. Hence we can set up the ``inverse ODE'' 
\begin{eqnarray*}
\sigma_1(s,e)&=& \theta(s,e)\\
\frac{d}{dt} \sigma_t(s,e) &=& (1-s) Z_x^{v_0}\left( (\sigma_t(v)), \sigma_t(s,e)\right) + s Z_{x}^{v_1}\left((\sigma_t(v)),\sigma_t(s,e)\right)
\end{eqnarray*}
to construct $\sigma$ and $\gamma(s,e)=\sigma_0(s,e)$. By construction, this $(\gamma,x)$ is the unique point in $\lambda_\infty$ which lands on $\theta$.

Finally we use the parallel transport associated to the connection $\nabla$ on $\nu^{eE\du eV}$ to construct
\[\psi : \kappa_\G \map M^{\p_\win \G}\times \nu_{PM}^{eE} \times \nu^{eV}.\]
\end{proof}

\begin{lemma} \label{lem:propagating}
Any bundle has a propagating flow. 
Hence the space $\cP(\nu)$ of propagating flows is contractible.
\end{lemma}

\begin{proof}
These ideas are largely based on \cite{Stacey}. We include them for completeness.
Lets first consider the case of a bundle $\R^n$ over a single point. In this case, we need to construct a smooth map
\[Z: \R^n\map \cX_c(\R^n)\]
which sends $v$ to a compactly supported vector field $Z_v$ so that 
\[Z_v (u) =\begin{cases}
v& u =t v \qquad t\in [0,1]\\
0 & v=0.
\end{cases}\]
Pick a smooth bump function
\[\rho:\R^2\map \R\]
so that
\[\rho(x,y)=\begin{cases}
1&|x|\leq 1 \\
1& |x|\leq |y|\\
0& |x| >>> |y|
\end{cases}\]
and let 
\[Z_v(u) =\rho(||u||^2, ||v||^2)\ v.\]
By construction $Z_v$ is a propagating flow.

Now lets consider a trivial bundle 
\[\R^n\times \R^k\map \R^k\]
over some Euclidean space. Pick an other smooth bump function 
\[ \sigma : \R\map \R \]
with
\[\sigma(x) = \begin{cases}
1 & |x|\leq 1\\
0 &|x|\geq 2
\end{cases}\]
We define a propagating flow
\[Y : \bR^n\times \bR^k \map \cX_c(\bR^n\times \bR^k).\]
by setting
\[Y_{(v_n,v_k)}(u_n,u_k) = \sigma(||u_k-v_k||^2) Z_{v_n}(u_n) \qquad \in T_{u_n}\bR^n\subset  T_{u_n,u_k} \bR^n\times \bR^k\]
where $Z$ was constructed above when we considered the fibre $\R^n\map {u_k}$.

Finally lets attack the general case. Pick a bundle $\nu$ above a smooth manifold $N$. Let $\{U_\alpha\}$ be a covering of $N$ by coordinate patches which trivializes $\nu$. Hence we have
\[\xymatrix{
\nu|_{U_\alpha}\ar@{<->}[r]^{\phi_\alpha}& \R^n \times U_\alpha \ar@{<->}[r]^{\psi_\alpha}& \R^n\times \R^k}\]
Pick a \emph{squared} partition of unity $\{\zeta_\alpha\}$ subordinate to the $U_\alpha$. Define
\[X:\nu \map \cX_c(\nu)\]
by setting
\[X_x(y) = \sum_{\alpha}\zeta_{\alpha}(x)\zeta_{\beta}(y) Y^{\alpha}_x(y)\]
where $Y^\alpha$ is the propagating flow obtained above for
\[\nu|_{U_\alpha} \cong \bR^{n}\times \R^k.\]
\end{proof}

\subsection[Definition of $\aFat_1$]{Thickening of $\aFat$ to $\aFat_1$ which includes information for comparing Euclidean spaces}
\label{sub:first thickening}

We now want to construct these string operations over families of admissible fat graphs. The operations constructed in the previous section use a different Euclidean space $W_\G=W^{eE\du eV}$ for each admissible fat graph. To compare the operations, we need to be able to compare these Euclidean spaces. In this section, we thicken the category $\aFat$ to a homotopy equivalent category $\aFat_1$ which includes comparisons between the Euclidean spaces.

 We let $\aFat_1$ be the category whose objects
 \[\mathit{Object}_{\aFat_1} = \mathit{Object}_{\aFat}\]
 are again admissible open-closed fat graphs. The space of morphisms between two admissible fat graphs $\tilG$ and $\G$ is
 \[\mathit{Morphism}_{\aFat_1}(\tilG, \G) = \Du_{\varphi :\tilG\smap \G} \Split(\varphi)\]
where $\varphi$ is a morphism of $\aFat$ between $\tilG$ and $\G$. The space $\Split(\varphi)$, contains for every extra vertex $v$ of $\G$ a splitting $(\alpha_v,\beta_v)$ of the short exact sequence
\[\xymatrix{
\bR\ar[r]&\ar@/_10pt/[l]_{\alpha_v} \bR^{\til{eE}_v\du \til{eV}_v}\ar[r]&\ar@/_10pt/[l]_{\beta_v} \frac{\bR^{\til{eE}_v\du \til{eV}_v}}{\bR}
}\]
where $\til{eE}_v$ contains the extra edges of $\tilG$ which are collapsed to $v$ by $\varphi$ and $\til{eV}_v$ contains the extra vertices of $\tilG$ which are sent to $v$ by $\varphi$. Note that these splittings give a splitting $(\alpha, \beta)$ of
\[\bR^{eV\du eE} \map \bR^{\til{eV}\du \til{eE}} \map \frac{ \bR^{\til{eV}\du \til{eE}}}{\bR^{eV\du eE}}\]
and an isomorphism
\[W_{\G} \cong W_{\tilG} \oplus \frac{ W_\tilG}{W_\G}.\]

The composition of two composable morphism 
\[(\varphi_1:\G_1\smap \G_0, (\alpha_1,\beta_1)) \cdot (\varphi_2:\G_2\smap \G_1, (\alpha_2,\beta_2))\]
is gotten by composing the morphisms in $\aFat$ and by composing the $\beta_v$'s. Note that  $\alpha_v$ is completely determined by $\beta_v$ and vice versa.

\begin{lemma}\label{lem:fat1 vs fat}
The geometric realization of the forgetful functor
\[\Psi:\aFat_1\map \aFat\]
is a homotopy equivalence.
\end{lemma}
\begin{proof}
Each space $\Split(\varphi)$ is a convex non-empty set and hence it is contractible. Above a simplex
\[\sigma=\left(\G_n\maplu{\varphi_n}\cdots \maplu{\varphi_1} \G_0\right)\]
in the nerve of $\aFat$, we have simplices
\[\Psi^{-1}(\sigma) =  \Split(\varphi_n)\times\cdots \times  \Split(\varphi_1)\]
which is again contractible. Using this contractibility and obstruction theory, we can define  a section
\[\Omega: |\aFat|\map |\aFat_1|\]
and homotopies between the composition $\Omega\cdot \Psi$ and the identity.
\end{proof}

%==========Higher simplices
\subsection{Choices involved in the construction of the operations over simplices}
\label{sub:second thickening}
\label{sub:higher simplices}
In this section we  thicken our model further to include the spaces of choices $\cT(\G)$ which appear in section \ref{sub:TP admissible}. We will first construct for each simplex $\sigma$ of $\Fat_1$ a contractible space $\Tub(\sigma)$.  Each point in that space will include a compatible choice of operations for each admissible graph $\G_i$ in $\sigma$. To be able to twist these choices in our model, we make sure these spaces are functorial with regards to inclusion of simplices.

Let $\cN\Fat_1$ denote the category of simplices of $\Fat_1$ under coinclusion. More prescisely, for any pair of simplices
\[\sigma = \left(\G_n\smap\cdots\smap \G_0\right) \qquad \til\sigma = \left(\tilG_k\smap \cdots \smap \tilG_0\right)\]
the morphisms
\[\mathit{Mor}_{\cN\Fat_1}(\sigma, \til\sigma) = \left\{ 0\leq i_0< \cdots<i_k \leq n \quad \G_{i_j} = \tilG_j\right\}.\]
Let $\Top^{?}$ be the category of functors to $\Top$. The objects of $\Top^?$ are pairs $(\cC,F)$ where $\cC$ is a small category and $F:\cC\smap \Top$ is a functor. A morphism
\[(J, \theta): (\cC,F)\map (\cD, G)\]
is a functor $J:\cC\smap \cD$ with a natural transformation
\[\theta: G\cdot J \Rightarrow F.\]
Finally for any simplex $\sigma =\G_n\smap \cdots \G_0$ of $\Fat$, let $\cN\sigma$ be the category of subsimplices of $\sigma$ under coinclusion. Note that $\cN\sigma$ is naturally isomorphic to the category $\Delta^n$. 

In this section, we construct a functor
\[\cT : \cN\Fat_1 \map \Top^{?}.\]
For any simplex $\sigma$ of $\Fat_1$, the functor 
\[\cT_\sigma : \cN\sigma\map \Top\]
maps from the category of subsimplices of $\sigma$ to the category of topological spaces. It sends a subsimplex 
\[\sigma_0 = (\G_{i_k}\smap \cdots \smap \G_{i_0})\] 
of $\sigma$ to a space $\cT_\sigma(\sigma_0)$ wich
picks compatible generalized Thom-Pontryagin collapses for the diagram
\begin{equation}\label{dia:sigma0}
\xymatrix@C=10pt{
M^{\G_0} \ar@{=>}[r]\ar[dd]\ar@{=>}[rd]& M^{\p_\win\G_0}\times PM^{eE_0}\times M^{eV_0}\ar@{=>}[d] &\ar[l]  M^{\p_\win\G_0}\times M^{eE_0\du eV_0}\ar@{=>}[r] \ar@{=>}[d]& M^{\p_\win\G_0}\times W_{\G_{i_k}}\ar[dd]\\
& M^{\p_\win \G_0}\times PM^{eE_0}\times M^{eE^{\varphi_{01}} \du eV_1}\ar[d] &\ar[d]\ar[l] M^{\p_\win\G_0}\times M^{eE_1\du eV_1}\ar@{=>}[ru]\\
M^{\G_1} \ar[dd]\ar@{=>}[r]\ar@{=>}[rd]&M^{\p_\win\G_1}\times PM^{eE_1}\times M^{eV_1}\ar@{=>}[d]&\ar[l]  M^{\p_\win\G_1}\times M^{eE_1\du eV_1} \ar@{=>}[r]\ar@{=>}[d]& M^{\p_\win\G_1}\times W_{\G_{i_k}}\ar[dd]\\
&M^{\p_\win\G_1}\times PM^{eE_1}\times M^{eE^{\varphi_{12}}\du eV_2}\ar[d] &\ar[l]\ar[d] M^{\p_\win\G_1}\times M^{eE_2\du eV_2} \ar@{=>}[ru]\\
\vdots\ar[d] &\ar[d] \vdots&\vdots \ar[d]&\vdots\ar[d]\\
M^{\G_{i_0}}\ar@{=>}[r] & M^{\G_{i_0}} \times PM^{eE_{i_0}}\times M^{eV_{i_0}} &\ar[l] M^{\p_\win\G_{i_0}} \times M^{eE_{i_0}\du eV_{i_0}} \ar@{=>}[r] & M^{\p_\win\G_{i_0}}\times W_{\G_{i_k}}
}
\end{equation}
The value of $\cT$ and $\cT_{\sigma}$ on morphisms will ensure compatibility when passing from one simplex to a subsimplex.

\begin{proposition} \label{prop:properties} \label{prop:key}
There is a functor
\[\cT : \cN\Fat_1 \map \Top^?\]
which sends a simplex $\sigma$ to a functor 
\[\cT_\sigma : \cN\sigma \map \Top.\]
This functor $\cT$ has the following properties.
\begin{enumerate}
\item For any subsimplex
$\sigma_0 =\left(\G_{i_k}\smap\cdots\smap \G_{i_0}\right)$
of $\sigma$, a point $x_0$ in $\cT_\sigma(\sigma_0)$ determines compatible operations
\begin{equation}\label{dia:comp op}
\xymatrix{
M^{\p_\win\G_0}\times W_{\G_{i_k}}\ar[d]\ar[rr]^{\mu_{x,0}}&& \Thom(\kappa_{\G_{i_k}}|_{M^{\G_0}})\ar[d]\\
\vdots\ar[d]&&\vdots\ar[d]\\
M^{\p_\win\G_{i_0}}\times W_{\G_{i_k}} \ar[rr]^{\mu_{x,i_k}}&&\Thom(\kappa_{\G_{i_k}}|_{M^{\G_ {i_0}}})
}\end{equation}
for the rows of \eqref{dia:sigma0}.
\item \label{subsubsimplices} For any inclusion of subsimplices of $\sigma$
\[  \sigma_1 = \left(\G_{i_{j_m}} \smap\cdots \smap \G_{i_{j_0}} \right)\subset \sigma_0=\left(\G_{i_k}\smap\cdots\smap \G_{i_0}\right) \]
we have a functor
\[\cT_\sigma(g) : \cT_\sigma(\sigma_1) \map \cT_\sigma(\sigma_0)\]
which is compatible with the operations. 
\item \label{inclusion simplices} For any inclusion 
\[h : \til\sigma=\left(\G_{i_k}\smap\cdots\smap \G_{i_0}\right) \subset \sigma=\left(\G_{n}\smap \cdots \smap \G_0\right)\]
of simplices of $\Fat_1$, the functor $\cT$ sends $h$ to a pair $(I, \theta)$ where 
\[I : \cN\til\sigma\map \cN\sigma\]
 is the natural inclusion of subsimplices of $\til\sigma$ as subsimplices of $\sigma$ and where the natural transformation
 \[\theta : \cT_{\sigma} \cdot I\Rightarrow \cT_{\til\sigma}\]
respects the operations of \eqref{dia:comp op}. 
\item The spaces $\cT_{\sigma}(\sigma_0)$ are contractible.
\end{enumerate}
\end{proposition}

Lets make the compatibility with the operations more precise. Property  \ref{subsubsimplices} of the last theorem means that for any point $y$ in $\cT_\sigma(\sigma_0)$ and $x=\cT_\sigma(g)(y)$, the following diagram commutes for $r\leq i_{0}$
\[\xymatrix{
M^{\p_\win \G_{r}} \times W_{\G_{i_{j_m}}}
 \times \frac{W_{\G_{i_k}}}{W_{\G_{i_{j_m}}}}
  \ar@{<->}[d]^{\cong} \ar[rr]^{\mu_{y,r}\times Id}&& \Thom\left(\kappa_{\G_{i_{j_m}}}|_{M^{\G_r}}\times  \frac{W_{\G_{i_k}} }{W_{\G_{i_{j_m}}}}
 \right)\ar@{<->}[d]^{\cong}\\
M^{\p_\win \G_{r}} \times W_{\G_{i_k}}
 \ar[rr]^{\mu_{x,r}}&& \Thom(\kappa_{\G_{i_k}}|_{M^{\G_r}})
}\]
Here the vertical maps are obtained by using the splitting of 
 the composition $\varphi_{i_k i_{j_m}} : \G_{i_k} \smap \G_{i_{j_m}}$ in $\aFat_1$.
In property \ref{inclusion simplices}, compatibility with the operations means that for any subsimplex
\[\sigma_1=\left(\G_{i_{j_m}}\smap \cdots \smap \G_{i_{j_0}}\right)\]
of $\til\sigma$ and for any $x\in\cT_\sigma(\sigma_1)$ and $y=\theta(x) \in \cT_{\til\sigma}$, the two operations
\[\xymatrix{
M^{\p_{\win} \G_{i_r}}\times W^{eE_{i_{j_m}}\du eV_{i_{j_m}}} 
\ar@/^20pt/[rr]^{\mu_{x,i_r}} 
\ar@/_20pt/[rr]_{\mu_{y,i_r}}
&&\Thom\left(\left.\kappa_{\G_{i_{j_m}}}\right|_{M^{\G_{\G_{i_r}}}}\right)
}\]
are equal for any $r\leq j_0$.

%----cONstruction
\begin{proof}[\underline{Choices for $\G_{i_0}$}]
Say we have a simple $\sigma$ and a subsimplex $\sigma_0$
\[\sigma = \left(\G_n\maplu{\varphi_{n-1}} \cdots \maplu{\varphi_0} \G_0\right)\qquad 
\sigma_0 = \left( \G_{i_k}\smap \ldots \smap \G_{i_0}\right)\]
We will construct here the part of $\cT_{\sigma}(\sigma_0)$ that gives the operation for $\G_{i_0}$ associated to the diagram
\begin{equation}\label{eqn:Gi0}
\xymatrix@C=15pt{
M^{\G_{i_0}}\ar@{=>}[r] & M^{\G_{i_0}} \times PM^{eE_{i_0}}\times M^{eV_{i_0}} &\ar[l] M^{\p_\win\G_{i_0}} \times M^{eE_{i_0}\du eV_{i_0}} \ar@{=>}[r] & M^{\p_\win\G_{i_0}}\times W_{\G_{i_k}}.
}\end{equation}
Later on, we will describe the  part of $\cT_{\sigma}(\sigma_0)$ which contains the data necessary to lift these choices to similar choices for $\G_{i_0-1}$, $\ldots$ ,$\G_{0}$.

We include in $\cT_\sigma(\sigma_0)$ a tubular neighborhood 
$\varphi_{V,i_0}$ for the map
\begin{equation}\label{eq:left basic}
M^{V_{i_0}} \map  M^{V_\win^{i_0}}\times W^{\left(eE_{i_k}\du eV_{i_k}\right)\setminus eE_{i_0}}.
\end{equation}
We assume that $\varphi_{V,i_0}$ sits above the identity in $M^{V_\win^{i_0}}$.
We also pick tubular neighborhoods $\varphi_{eE,r}$ and $\varphi_{eE,0}$ for the embeddings
\begin{eqnarray}
M^{eE^{\varphi_r}} &\map& W^{eE_{\varphi_r}}\qquad r\leq i_0\label{eq:left sup r}\\
 M^{eE_0}&\map& W^{eE_0}.\label{eq:left sup 0}
\end{eqnarray}
Recall that $eE_{\varphi_r}$ is the set of extra edges collapsed by $\varphi_r$.
These tubular neighborhoods  determine a tubular neighborhood for the last  embedding of \eqref{eqn:Gi0}.

We also pick a tubular neighborhood $\varphi_{\rho, i_0}$ for the diagonal arrow in  
\begin{equation}\label{eq:right basic}
\xymatrix{
&&\frac{\left.T\left(M^{V_{i_0}^\win}\times W^{\left(eE_{i_k}\du eV_{i_k}\right)\setminus eE_{i_0}}\right)\right|_{M^{V_{i_0}}}}{TM^{V_{i_0}}}\times M^{eH_{i_0}}\ar[d]\\
M^{V_{i_0}}\ar[rru]\ar[rr]&&M^{V_{i_0}}\times M^{eH_0}
}
\end{equation}
which is a twisted finite dimensional version of the first embedding of \eqref{eqn:Gi0}. We again ask that these sit atop the identity ma in $M^{V_{i_0}}$.
We pick a propagating $\Omega_{\rho,i_0}$ flow for the normal bundle of this diagronal arrow. These choices determine a tubular neighborhood for the diagonal arrow in 
\[\xymatrix{
&& \left.\frac{T\left(M^{V_{i_0}}\times W^{\left(eE_{i_k}\du eV_{i_k}\right)\setminus eE_{i_0}}\right)}{TM^{V_{i_0}}}\right|_{M^{\p_\win\G_{i_0}}\times M^{eV_{i_0}}}\times PM^{eE_{i_0}}
\ar[d]\\
M^{\G_{i_0}}\ar[rr]\ar[rru]&& M^{\p_{\win}\G_0}\times M^{eV_{i_0}} \times PM^{eE_{i_0}}. }\] 
We finally pick a connection $\nabla_r$ for $\nu^{eE^{\varphi_r}}$. 

Not all such choices will be permitted as some will not lift to operations for the other admissible fat graphs $\G_0,\ldots \G_{i_0-1}$. We will  add restrictions on these choices at a later point. But we first need to fix the information that we need to set up to potential liftings.

As in section \ref{sub:objects}, these choices define an operation
\[M^{\G_{i_0}}\times W^{eE_k\du eV_k} \map \Thom\left(\left.\kappa_{i_k}\right|_{M^{\G_{i_0}}}\right) \]
where $\kappa_{i_k}$ is the bundle of section \ref{sub:TP admissible} which is the pullback of the following diagram
\[\xymatrix{
\kappa_{i_k}\ar[r]\ar[d]&TM^{eH_{i_k}} \times \nu^{eV_{i_k}\du eE_{i_k}} \ar[d]\\
M^{\G_{i_k}}\ar[r]&M^{eH_{i_k}}\times M^{eV_{i_k}\du eE_{i_k}}
}\]
which we restrict along
\[M^{\G_{i_0}} \map M^{\G_{i_k}}.\]
\end{proof}

\begin{proof}[\underline{Choices to lift}]
For each $r\leq i_0$,  consider the diagram
\begin{equation}\label{lifting}
\xymatrix{
M^{V_{r}}\ar@{=>}[r]\ar@{=>}[d]&M^{V_r^\win}\times W^{\left(eE_{i_k}\du eV_{i_k}\right)\setminus eE_{r}} \\
M^{V_r^\win}\times M^{eV_{r+1}}\times M^{eE^{\varphi_r}}\ar@{=>}[ru].
}\end{equation}
%Note that we remove the edges $eE_r$ to be able to lift the propagating flow on the left hand side. 
To be able to lift a tubular neighborhood from the diagonal arrow to the top arrow, we pick tubular neighborhoods $\psi_r$ for the embedding 
\begin{equation}\label{eq:goingup}
M^{V_r} \map \left.\left(\frac{T\left(M^{V^\win_r}\times W^{\left(eE_{i_k}\du eV_{i_k}\right)\setminus eE_{r}}\right)}{T\left(M^{V_{r}^{\win}}\times M^{ eV_{r+1}}\times M^{eE^{\varphi_r}}\right)}\right)\right|_{M^{V_r^{\win}}\times  M^{eV_{r+1}}\times M^{eE^{\varphi_r}}}
\end{equation}
of $M^{V_r}$ into the normal bundle of the diagonal map. We again ask that it lies above  the identity on $M^{V^\win_r}$. 
Finally we chose a propagating flow $\Omega_{\varphi_r}$ for each $\nu^{eE^{\varphi_r}}$. Together these will allow us to carry the choices for $\G_{i_0}$ up the simplex under two types of restrictions.
\end{proof}

\begin{proof}[\underline{Lifting of the choices and restrictions.}]
Lets now lift the choices for $\G_{i_0}$ to similar choices for each $\G_r$ for $r\leq i_0$. Pick an $x$ in $\cT_{\sigma}(\sigma_0)$. Lets assume by induction that we have already constructed a tubular neighborhood $\varphi_{V,r+1}$ for 
\[M^{V_{r+1}}\map M^{V_{r+1}^\win} \times W^{(eE_{i_k}\du eV_{i_k})\setminus eE_{r+1}},\]
 which lives above the identity on $M^{V^{\win}_{r+1}}$. Lets built a tubular neighborhood $\varphi_{V,r}$ for 
\[M^{V_r}\map M^{V_r^\win} \times W^{(eE_{i_k}\du eV_{i_k})\setminus eE_{r}}.\]
We  consider the diagram
\begin{equation}\label{dia:bottom left}
\xymatrix{
M^{V_r^\win}\times M^{eV_{r+1}} \ar@{=>}[r]\ar[d]& M^{V_r^\win} \times W^{\left(eE_{i_k}\du eV_{i_k}\right)\setminus eE_{r+1}} \ar[d]\\
M^{V_{r+1}}\ar@{=>}[r]& M^{V_{r+1}^\win}\times W^{\left( eE_{i_k}\du eV_{i_k}\right)\setminus eE_{r+1}}
}\end{equation}
Since the tubular neighborhood $\varphi_{V,r+1}$  lives above the identity on $M^{V_{r+1}^\win}$, it lifts to a tubular neighborhood for the top map. We can then use the tubular neighborhood $\psi_{eE,r}$ for 
\[M^{eE^{\varphi_r}}\map W^{eE^{\varphi_r}}\]
 to get a tubular neighborhood for the diagonal arrow in the diagram of \eqref{lifting}. Since the map of \eqref{eq:goingup} is the inclusion of the top entry into the normal bundle of the bottom embedding in diagram \eqref{lifting}, we can compose the a tubular neighborhood $\psi_r$ for \eqref{eq:goingup} with the one we have just constructed for the bottom map to get a tubular neighborhood for 
\[M^{V_r} \map M^{V_r^\win}\times W^{(eE_{i_k}\du eV_{i_k})\setminus eE_{r}}\]
as desired. So far there are no restrictions.

We will apply a similar reasoning to construct a tubular neighborhood $\varphi_{\rho,r}$ for 
\begin{equation}\label{eq:leftr}
M^{V_r}\map \frac{T\left.\left( M^{V_r^\win}\times  W^{\left(eE_{i_k}\du eV_{i_k}\right)\setminus eE_r}\right)\right|_{M^{V_{r}}}}{ TM^{V_r}}\times M^{eH_{r}}.
\end{equation}
Since the diagram
\[\xymatrix{M^{V_r}\ar@{=>}[r]\ar[d] & M^{V_{r}^\win}\times M^{eH_r}\times M^{eE^{\varphi_r}\du eV_{r+1}}\ar[d]\\
M^{V_{r+1}}\ar@{=>}[r]& M^{V_{r+1}}\times M^{eH_{r+1}}
}\]
is a pull-back then so is
\begin{equation}\label{pullback r r+1}
\xymatrix{
M^{V_r}\ar[d] \ar[r] &
\left.\frac{T\left.\left( M^{V^\win_r}\times  W^{\left(eE_{i_k}\du eV_{i_k}\right)\setminus eE_{r+1}}\right)\right|_{M^{V_{r}}}}{ TM^{V_r}}\right|_{M^{V_r^{\win}\du eV_{r+1}}}\times M^{eH_{r}}\times M^{eE^{\varphi_r}} \ar[d]\\
M^{V_{r+1}}\ar[r]& \frac{T\left.\left( M^{V^\win_{r+1}}\times  W^{\left(eE_{i_k}\du eV_{i_k}\right)\setminus eE_{r+1}}\right)\right|_{M^{V_{r+1}}}}{ TM^{V_{r+1}}}\times M^{eH_{r+1}}.
}\end{equation}
Lets assume that we have, by induction, a tubular neighborhood $\varphi_{\rho,r+1}$ for the bottom map. We need assume that this tubular neighborhood restricts to a tubular neighborhood of the top map. This is the \emph{first restriction} that we put on the points of $\cT_{\sigma}(\sigma_0)$. We extend this neighborhood to one for
\[\xymatrix{
M^{V_r} \ar[r]&\left.\frac{T\left.\left( M^{V_r^\win}\times  W^{\left(eE_{i_k}\du eV_{i_k}\right)\setminus eE_{r+1}}\right)\right|_{M^{V_{r}}}}{ TM^{V_r}}\right|_{M^{V_r^{\win}\du eV_{r+1}}}\times M^{eH_{r}}}\times \nu^{eE^{\varphi_r}}\]
by using the connection on $\nu^{eE^{\varphi_r}}$.
We now consider the diagram
\begin{equation}\label{restriction 2}
\xymatrix{
M^{V_r}\ar[r]\ar[dr]& M^{V_r}\times M^{eH_r}\ar[d]\\
& \left.\frac{T\left.\left( M^{V_r^\win}\times  W^{\left(eE_{i_k}\du eV_{i_k}\right)\setminus eE_{r+1}}\right)\right|_{M^{V_{r}}}}{ TM^{V_r}}\right|_{M^{V_r^{\win}\du eV_{r+1}}} \times M^{eH_{r}} \times \nu^{eE^{\varphi_r}}
}\end{equation}
We have just constructed a tubular neighborhood for the bottom map. By crossing the tubular neighborhood $\psi_r$ of \eqref{eq:goingup} with the identity of $M^{eH}$, we get a tubular neighborhood for the vertical map. We restrict our choices so that these tubular neighborhood induce one for the inclusion of $M^{V_r}$ into the normal bundle of the vertical map. This is a \emph{second and last restriction} on the points of $\cT_\sigma(\sigma_0)$. Note that this restriction only asks that the tubular neighborhood of the diagonal map be small enough to sit inside the tubular neighborhood of the vertical one.

We now construct for each $r$ a propagating flow $\Omega_{\rho,r}$ for the normal bundle of \eqref{eq:leftr}. Lets assume by induction that we have chosen one $\Omega_{\rho,r+1}$ for $r+1$. Then we consider again the pullback square of \eqref{pullback r r+1}.
The propagating flow $\Omega_{\rho,r+1}$ for the normal bundle of its bottom map restricts to one for the normal bundle of its top map by 
\[\xymatrix{
\Omega_{top} : \lambda_{\mathit{top}}\ar[r]&
\lambda_{\mathit{bottom}}\ar[rr]^{\Omega_{\rho,r+1}}&& \chi^{v}_c(\lambda_{\mathit{bottom}})\ar[r]
& \chi^{v}_c(\lambda_{\mathit{top}})
}\]
where the last map restricts the vector fields to the appropriate fibres. Hence we get a propagating flow for the normal bundle of 
\[\xymatrix{
M^{V_r} \ar[r]&\left.\frac{T\left.\left( M^{V^\win_r}\times  W^{\left(eE_{i_k}\du eV_{i_k}\right)\setminus eE_{r+1}}\right)\right|_{M^{V_{r}}}}{ TM^{V_r}}\right|_{M^{V_r^{\win}\du eV_{r+1}}}\times M^{eE^{\varphi_r}}\times M^{eH_{r}}
}\]
We extend this propagating flow to one for  the normal bundle of
\[\xymatrix{
M^{V_r} \ar[r]&\left.\frac{T\left.\left( M^{V^\win_r}\times  W^{\left(eE_{i_k}\du eV_{i_k}\right)\setminus eE_{r+1}}\right)\right|_{M^{V_{r}}}}{ TM^{V_r}}\right|_{M^{V_r^{\win}\du eV_{r+1}}}\times \nu^{eE^{\varphi_r}}\times M^{eH_{r}}
}\]
by adding the propagating flow $\Omega_{\varphi_r}$ for $\nu^{eE^{\varphi_r}}$ that is part of $x$. Finally, we consider the pullback diagram
\[\xymatrix{
M^{V_r}\ar[r]\ar[d]& \left.\frac{T\left.\left( M^{V^\win_r}\times  W^{\left(eE_{i_k}\du eV_{i_k}\right)\setminus eE_{r+1}}\right)\right|_{M^{V_{r}}}}{ TM^{V_r}}\right|_{M^{V_r^{\win}\du eV_{r+1}}}\times \nu^{eE^{\varphi_r}}\times M^{eH_{r}}\ar[d]\\
M^{V_r}\ar[r] & \left.\frac{T\left.\left( M^{V^\win_r}\times  W^{\left(eE_{i_k}\du eV_{i_k}\right)\setminus eE_{r+1}}\right)\right|_{M^{V_{r}}}}{ TM^{V_r}}\right|_{M^{V_r^{\win}\du eV_{r+1}}}\times W^{eE^{\varphi_r}}\times M^{eH_{r}}
}\]
where the vertical map uses the tubular neighborhood for
\[M^{eE^{\varphi_r}}\map W^{eE^{\varphi_r}}.\]
We have constructed a propagating flow for the bottom map
and  we lift this propagating to get a propagating flow $\Omega_{\rho,r}$ for the normal bundle of the top map.
\end{proof}

%---------1st property
\begin{proof}[\underline{Proof of the first property.}]
For any point $x$ in $\cT_{\sigma}(\sigma_0)$, we have shown how to lift the choices for $\G_{i_0}$ to choices for any $\G_r$ with $r\leq i_0$. We therefore get operations
\[\xymatrix{
 M^{\p_\win \G_r} \times W^{eE_{i_k}\du eV_{i_k}} \ar[rr]&&\Thom\left(\left.\kappa_{\G_{i_k}}\right|_{M^{\G_r}}\right)
 }\]
for each such $r$. We will prove that the each diagram
\begin{equation}\label{eq:comp r r+1}
\xymatrix{
 M^{\p_\win \G_r} \times W^{eE_{i_k}\du eV_{i_k}} \ar[rr]\ar[d]&&\Thom\left(\left.\kappa_{\G_{i_k}}\right|_{M^{\G_r}}\right)\ar[d]\\
 M^{\p_\win \G_{r+1}} \times W^{eE_{i_k}\du eV_{i_k}} \ar[rr]&&\Thom\left(\left.\kappa_{\G_{i_k}}\right|_{M^{\G_{r+1}}}\right)
}\end{equation}
commutes.

For this purpose, we will consider the diagram
\begin{equation}\label{eq:dia r r+1}
\xymatrix{
M^{\G_r}\ar@{=>}[r]\ar@{=>}[rd]\ar[dd]& M^{\p_\win\G_r}\times PM^{eE_r}\times M^{eV_r}\ar@{=>}[d]&\ar[l] M^{\p_\win\G_r}\times M^{eE_r\du eV_r} \ar@{=>}[r]\ar@{=>}[d]& M^{\p_\win \G_r}\times W_{\G_{i_k}}\ar[dd]\\
& M^{\p_\win\G_r}\times PM^{eE_r}\times M^{eE^{\varphi_r}\du eV_{r+1}}\ar[d]&\ar[l] M^{\p_\win\G_r}\times M^{eE_{r+1}\du eV_{r+1}}\ar@{=>}[ru]\ar[d]\\
M^{\G_{r+1}}\ar@{=>}[r]& M^{\p_\win\G_{r+1}} \times PM^{eE_{r+1}}\times M^{eV_{r+1}}&\ar[l] M^{\p_\win\G_{r+1}}\times M^{eE_{r+1}\du eV_{r+1}} \ar@{=>}[r]& M^{\p_\win \G_{r+1}}\times W_{\G_{i_k}}
}\end{equation}
We already have tubular neighborhoods for the double arrows in the top and the bottom rows. From the choices already made, we will construct tubular neighborhood for all other double arrows in this diagram. These already appeared secretely in our construction of the liftings. 

%-----Right
Lets first consider the right-most part of this diagram and its finite dimensional equivalent
\begin{equation}\label{dia:left two downstairs}
\xymatrix{
M^{V^{\win}_r}\times M^{eE_r\du eV_r} \ar@{=>}[r]^{(C)}\ar@{=>}[d]_{(B)}& 
M^{V^{\win}_r}\times W_{\G_{i_k}}\ar[dd]\\
M^{V^{\win}_r}\times M^{eE_{r+1}\du eV_{r+1}}\ar@{=>}[ru]_{(A)}\ar[d]\\
M^{V^{\win}_{r+1}}\times M^{eE_{r+1}\du eV_{r+1}} \ar@{=>}[r]&
M^{V^{\win}_{r+1}}\times W_{\G_{i_k}}.
}\end{equation}
Lets construct the tubular neighborhood for (A) explicitly. The
diagram
\[\xymatrix{
M^{V_r^\win}\times M^{eV_{r+1}} \ar@{=>}[r]\ar[d]& M^{V_r^\win} \times W^{\left(eE_{i_k}\du eV_{i_k}\right)\setminus eE_{r+1}} \ar[d]\\
M^{V_{r+1}}\ar@{=>}[r]& M^{V_{r+1}^\win}\times W^{\left( eE_{i_k}\du eV_{i_k}\right)\setminus eE_{r+1}}
}\]
already appeared in diagram \eqref{dia:bottom left} of page \pageref{dia:bottom left}. We have already shown that our choice of a tubular neighborhood $\varphi_{V,r+1}$ for the bottom row gives a tubular neighborhood $\tilde{\varphi}_{V,r+1}$ for the top row. We get a tubular neighborhood $\psi_A = \tilde{\varphi}_{V,r+1}\times \psi_{eE,r}$ for arrow (A) by taking the product of this tubular neighborhood and of the chosen tubular neighborhood for the 
\[M^{eE^{\varphi_s}} \map W^{eE^{\varphi_s}}.\]
This choice for $(A)$ gives a commutative diagram
\[\xymatrix{
M^{V^{\win}_{r}}\times W^{eE_{i_k}\du eV_{i_k}}\ar[d]\ar[r]&
\Thom\left(\frac{M^{V^{\win}_r}\times \left.TW^{(eV_{i_k}\du eE_{i_k})\setminus eE_{r+1}}\right|_{M^{eV_{r+1}}}}{M^{V^{\win}_r}\times TM^{eV_{r+1}}} \times \nu^{eE_{r+1}}\right)\ar[d]\\
M^{V^{\win}_{r+1}} \times W^{ eE_{i_k}\du eV_{i_k}}\ar[r]&\Thom\left(\frac{M^{\p_\win\G_{r+1}}\times \left.TW^{(eV_{i_k}\du eE_{i_k})\setminus eE_{r+1}}\right|_{M^{eV_{r+1}}}}{M^{V^{\win}_{r+1}}\times TM^{eV_{r+1}}} \times \nu^{eE_{r+1}}\right).
}\]

We now consider the triangle of diagram \eqref{dia:left two downstairs}. As part of $\cT_{\sigma}(\sigma_0)$, we have already chosen a tubular neighborhood $\psi_r$ for the embedding
\[M^{V_r} \map \left.\left(\frac{T\left(M^{V^\win_r}\times W^{\left(eE_{i_k}\du eV_{i_k}\right)\setminus eE_{r}}\right)}{T\left(M^{V_{r}^{\win}}\times M^{ eV_{r+1}}\times M^{eE^{\varphi_r}}\right)}\right)\right|_{M^{V_r^{\win}}\times  M^{eV_{r+1}}\times M^{eE^{\varphi_r}}}\]
of \eqref{eq:goingup}. We define a tubular neighborhood for the twisted version
\[\tilde{B} : M^{V_r}\times M^{eE_r}  \map \left.\left(\frac{T\left(M^{V^\win_r}\times W^{\left(eE_{i_k}\du eV_{i_k}\right)\setminus eE_{r}}\right)}{T\left(M^{V_{r}^{\win}}\times M^{ eV_{r+1}}\times M^{eE^{\varphi_r}}\right)}\right)\right|_{M^{V_r^{\win}}\times  M^{eV_{r+1}}\times M^{eE^{\varphi_r}}} \times \nu^{eE_r}.
\]
of $(B)$ by $\psi_r\times Id_{\nu^{eE_r}}$. Since we have used $\psi_r$ to construct $\varphi_{V,r}$, we  
get compatible choices for the diagram
\[\xymatrix{
M^{V^{\win}_r}\times M^{eE_r\du eV_r} \ar@{=>}[r]^{(C)}\ar@{=>}[d]_{(B)}& 
M^{V^{\win}_r}\times W^{eE_{i_k}\du eV_{i_k}}\\
M^{V^{\win}_r}\times M^{eE_{r+1}\du eV_{r+1}}\ar@{=>}[ru]_{(A)}
.}\]
We can lift the choices for $(A)$ and $(\tilde{B})$ by using the fact that these tubular neighborhoods live above the identity on $M^{V_r^\win}$. Since this is what we have done to get the original operations, our  tubular neighborhoods are compatible for this part of \eqref{eq:dia r r+1}.

%-Middle!
Lets now consider the middle part of the diagram
\[\xymatrix{
 M^{\p_\win\G_r}\times PM^{eE_r}\times M^{eV_r}\ar@{=>}[d]_{(D)}&\ar[l] M^{\p_\win\G_r}\times M^{eE_r\du eV_r}\ar@{=>}[d]^{(B)}\\
 M^{\p_\win\G_r}\times PM^{eE_r}\times M^{eE^{\varphi_r}\du eV_{r+1}}\ar[d]&\ar[l] M^{\p_\win\G_r}\times M^{eE_{r+1}\du eV_{r+1}}\ar[d]\\
 M^{\p_\win\G_{r+1}} \times PM^{eE_{r+1}}\times M^{eV_{r+1}}&\ar[l] M^{\p_\win\G_{r+1}}\times M^{eE_{r+1}\du eV_{r+1}}.
 }\]
 We have already picked a tubular neighborhood for the twisted version $\tilde{B}$ of $B$. We construct a tubular neighborhood 
$\psi_r \times Id_{\nu_{PM}^{eE_{r}}}$
 for the twisted version of $(D)$ similarly.
These choices give a commutative diagram
\[\xymatrix{
\Thom\left(M^{\p_\win\G_r}\times \nu_{PM}^{eE_r} \times \frac{T\left(W^{(eE_{i_k}\du eV_{i_k})\setminus eE_r}\right)}{TM^{eV_r}}\right)&
\ar[l]
\Thom\left(M^{\p_\win\G_r} \times 
\frac{TW_{\G_{i_k}}}{TM^{eV_r}\times TM^{eE_{r}}}\right)\\
\Thom\left(M^{\p_\win\G_r}\times \nu_{PM}^{eE_r}\times \frac{T\left(W^{(eE_{i_k}\du eV_{i_k})\setminus eE_{r}}\right)}{T\left(M^{eV_{r+1}\du eE^{\varphi_r}}\right)}\right)\ar[d]\ar[u]&
\Thom\left(M^{\p_\win\G_r}\times \frac{TW_{\G_{i_k}}}{T\left(M^{eV_{r+1}\du  eE_{r+1}}\right)}\right)
\ar[u]\ar[d]\ar[l]\\
\Thom\left(M^{\p_\win\G_{r+1}}\times \nu_{PM}^{eE_{r+1}} \times \frac{T\left(W^{(eE_{i_k}\du eV_{i_k})\setminus eE_{r+1}}\right)}{TM^{eV_{r+1}}}\right)&
\ar[l]
\Thom\left(M^{V_{r+1}^\win} \times \frac{T\left(W^{(eE_{i_k}\du eV_{i_k})}\right)}{TM^{eV_{r+1}\du eE_{r+1}}}\right).
}\]

%-----Left stuff
We are now left to study the left-most part of the diagram. 
\[\xymatrix{
M^{\G_r}\ar@{=>}[r]\ar@{=>}[rd]_{(E)}\ar[dd]& M^{\p_\win\G_r}\times PM^{eE_r}\times M^{eV_r}\ar@{=>}[d]^{(D)}\\
& M^{\p_\win\G_r}\times PM^{eE_r}\times M^{eE^{\varphi_r}\du eV_{r+1}}\ar[d]\\
M^{\G_{r+1}}\ar@{=>}[r]& M^{\p_\win\G_{r+1}} \times PM^{eE_{r+1}}\times M^{eV_{r+1}}
}\]
We already have chosen a tubular neighborhood for every double arrow except $(E)$.
We construct a tubular neighborhood for $(E)$ by first considering the pull-back diagram \[\xymatrix{
M^{V_r}\ar[d] \ar[r] &
\left.\frac{T\left.\left( M^{V_r^{\win}}\times  W^{\left(eE_{i_k}\du eV_{i_k}\right)\setminus eE_{r+1}}\right)\right|_{M^{V^{\win}_{r}}}}{ TM^{V_r}}\right|_{M^{V_r^{\win}\du eV_{r+1}}}\times M^{eH_{r}}\times M^{eE^{\varphi_r}} \ar[d]\\
M^{V_{r+1}}\ar[r]& \frac{T\left.\left( M^{V^{\win}_{r+1}}\times  W^{\left(eE_{i_k}\du eV_{i_k}\right)\setminus eE_{r+1}}\right)\right|_{M^{V^{\win}_{r+1}}}}{ TM^{V_{r+1}}}\times M^{eH_{r+1}}.
}\]
which has already appeared in \eqref{pullback r r+1}. The first restriction ensured that the choice of a tubular neighborhood for the bottom arrow restricts to one for the top arrow. We have also constructed a propagating flow for the tubular neighborhood of the top map. Just as before, we extend this tubular neighborhood to one for 
\[\xymatrix{
M^{V_r} \ar[r]
&\left.\frac{T\left.\left( M^{V^\win_r}\times  W^{\left(eE_{i_k}\du eV_{i_k}\right)\setminus eE_{r+1}}\right)\right|_{M^{V^\win_{r}}}}{ TM^{V_r}}\right|_{M^{V_r^{\win}\du eV_{r+1}}}\times M^{eH_{r}}\times \nu^{eE^{\varphi_r}}}\]
by using the connection on $\nu^{eE^\varphi}$. We also extend the propagating flow for 
\[M^{V_r}\map
\left.\frac{T\left.\left( M^{V_r^{\win}}\times  W^{\left(eE_{i_k}\du eV_{i_k}\right)\setminus eE_{r+1}}\right)\right|_{M^{V^{\win}_{r}}}}{ TM^{V_r}}\right|_{M^{V_r^{\win}\du eV_{r+1}}}\times M^{eH_{r}}\times M^{eE^{\varphi_r}} 
\]
to one for
\[M^{V_r}\map
\left.\frac{T\left.\left( M^{V_r^{\win}}\times  W^{\left(eE_{i_k}\du eV_{i_k}\right)\setminus eE_{r+1}}\right)\right|_{M^{V^{\win}_{r}}}}{ TM^{V_r}}\right|_{M^{V_r^{\win}\du eV_{r+1}}}\times M^{eH_{r}}\times \nu^{eE^{\varphi_r}} 
\]
by using the propagating flow for $\nu^{eE^\varphi_r}$. Using this tubular neighborhood and this propagating flow, we can define a tubular neighborhood for $(E)$ in the usual way. 

By construction the tubular neighborhood on the finite diagram are compatible. 
Note that the propagating flows move constant paths to constant paths following the vertices. In particular, we get a commutative diagram 
\[\xymatrix{
\Thom\left(\left.\kappa_{\G_{i_k}}\right|_{\G_{r}}\right)\ar[dd]
\\
&\ar[ul]\Thom\left(M^{\p_\win\G_r}\times \nu_{PM}^{eE_r}\times \frac{T\left(W^{(eE_{i_k}\du eV_{i_k})\setminus eE_{r}}\right)}{T\left(M^{eV_{r+1}\du eE^{\varphi_r}}\right)}\right)\ar[d]\\
\Thom\left(\left.\kappa_{\G_{i_k}}\right|_{\G_{r+1}}\right)&\ar[l]\Thom\left(M^{\p_\win\G_{r+1}}\times \nu_{PM}^{eE_{r+1}} \times \frac{T\left(W^{(eE_{i_k}\du eV_{i_k})\setminus eE_{r+1}}\right)}{TM^{eV_{r+1}}}\right)
}\]
Finally, by our choice of tubular neighborhood for the twisted version of
\[M^{\G_r}\map M^{\p_\win\G_r}\times PM^{eE_r}\times M^{eV_r}\]
we also get a commutative diagram
\[\xymatrix{
\Thom\left(\left.\kappa_{\G_{i_k}}\right|_{\G_{r}}\right)&\ar[l]\Thom\left(M^{\p_\win\G_r}\times \nu_{PM}^{eE_r} \times \frac{T\left(W^{(eE_{i_k}\du eV_{i_k})\setminus eE_r}\right)}{TM^{eV_r}}\right)\\
&\ar[ul]\ar[u]\Thom\left(M^{\p_\win\G_r}\times \nu_{PM}^{eE_r}\times \frac{T\left(W^{(eE_{i_k}\du eV_{i_k})\setminus eE_{r}}\right)}{T\left(M^{eV_{r+1}\du eE^{\varphi_r}}\right)}\right)
}\]
This completes the proof of property 1.

\end{proof}

%-----Morphisms----Property 2?
\begin{proof}[\underline{The value of $\cT_\sigma$ on morphisms and property 2}]
Say we have a morphism 
\[\sigma_1=\left(\G_{i_{j_m}}\smap \cdots \smap \G_{i_{j_0}}\right)\map
\sigma_0= \left(\G_{i_k}\smap  \ldots \smap \G_{i_0}\right)\]
of $\cN\sigma$. We will define a continuous map
\[\cT_\sigma(\sigma_1)\map \cT_{\sigma}(\sigma_0)\]
which respects the operations. 
Let $x$ be an element of $\cT_\sigma(\sigma_1)$. This element picks choices for 
\begin{equation}\label{eq:sigma1}
\xymatrix@C=15pt{
M^{\G_{i_{j_0}}}\ar[r]& M^{\p_\win\G_{i_{j_0}}} \times PM^{eE_{i_{j_0}}}\times M^{eV_{i_{j_0}}} &\ar[l] M^{\p_\win\G_{i_{j_0}}} \times M^{eE_{i_{j_0}}\du eV_{i_{j_0}}} \ar[r]&
M^{\p_\win\G_{i_{j_0}}}\times W_{\G_{i_{j_m}}}
}\end{equation}
and the information necessary to lift these choices up to choices for $\G_r$ with $r\leq i_{j_0}$. 

In the construction of $\cT_\sigma(\sigma_1)$ we have shown how to get expand the choices for the bottom row to similar choices for the other rows. Since $i_0\leq i_{j_0}$, $x$ determines the necessary tubular neighborhood and propagating flows for
\[\xymatrix@C=15pt{
M^{\G_{i_0}}\ar[r]& M^{\p_\win\G_{i_0}} \times PM^{eE_{i_0}}\times M^{eV_{i_0}} &\ar[l] M^{\p_\win\G_{i_{0}}} \times M^{eE_{i_{0}}\du eV_{i_{0}}} \ar[r]&
M^{\p_\win\G_{i_{0}}}\times W_{\G_{i_{j_m}}}.}\]
We also have the necessary ``going up" information since any $r$ less than $i_0$ is also less than $i_{j_0}$. In particular, we have a continuous map
\[\cT_{\sigma}(\sigma_1)\map \cT_{\sigma}(\G_{i_0} \smap \cdots \smap \G_{i_{j_m}}).\]

To get a point in $\cT_{\sigma}(\sigma_0)$, we simply have to change our Euclidean space from $W_{\G_{i_{j_m}}}$ to $W_{\G_{i_k}}$. As part of the composition
\[\varphi_{i_k i_{j_m}} : \G_{i_k} \smap \G_{i_k-1}\smap\cdots \smap \G_{i_{j_m}},\]
we get a splitting
\[W_{\G_{i_k}} \cong W_{\G_{i_{j_m}}} \oplus \frac{W^{eE_{i_k}\du eV_{i_k}}}{W^{eE_{i_{j_m}}\du eV_{i_{j_m}}}}.\]
Using this splitting, we can transform 
\[\xymatrix@C=8pt{
M^{\G_{i_0}} \ar[r]&M^{\p_{\win}\G_{i_0}}\times PM^{eE_{i_0}} \times M^{eV_{i_0}} &\ar[l] M^{\p_\win\G_{i_0}}\times M^{eE_{i_0}\du eV_{i_0}} \ar[r]\ar[rd]_{-\times 0}& M^{\p_\win\G_{i_0}} \times W^{eE_{i_{k}}\du eV_{i_{k}}}\ar[d]\\
&&& M^{\p_\win\G_{i_0}}\times W^{eE_{i_{j_m}}\du eV_{i_{j_m}}} \times \frac{W^{eE_{i_{k}}\du eV_{i_{k}}}}{W^{eE_{i_{j_m}}\du eV_{i_{j_m}}}}
}\]
and similarly for all other lines in the diagram \eqref{eq:sigma1}. We can extend all of the choices by using the identity on the extra
\[\frac{W^{eE_{i_{k}}\du eV_{i_{k}}}}{W^{eE_{i_{j_m}}\du eV_{i_{j_m}}}}\]
term. This gives the map
\[\cT_{\sigma}(\sigma_1)\map \cT_{\sigma}(\sigma_0)\]
and  for any $r\leq i_0$ a commutative diagram
\[\xymatrix{
M^{\p_\win \G_{r}} \times W^{eE_{i_{j_m}}\du eV_{i_{j_m}}} \times \frac{W^{eE_{i_k}\du eV_{i_k}} }{W^{eE_{i_{j_m}}\du eV_{i_{j_m}}}} \ar[d] \ar[rr]^{\mu_{x,r}\times Id}&& \Thom\left(\kappa_{\G_{i_{j_m}}}|_{M^{\G_r}}\times  \frac{W^{eE_{i_k}\du eV_{i_k}} }{W^{eE_{i_{j_m}}\du eV_{i_{j_m}}}} \right)\ar[d]\\
M^{\p_\win \G_{r}} \times W^{eE_{i_k}\du eV_{i_k}} \ar[rr]^{\mu_{y,r}}&& \Thom(\kappa_{\G_{i_k}}|_{M^{\G_r}}).
}\]
\end{proof}

%--------------Nat Trans.
\begin{proof}[\underline{Construction of the natural transformations}]
Say we have a morphism $g$ in $\cN\Fat$
\[g: \tilsigma=\left(\G_{i_k}\smap \cdots \G_{i_0}\right) \subset \sigma=\left(\G_n\smap \cdots \smap \G_0\right).\]
The morphism
\[\cT(g)=(\cN g, \theta) : (\cN\tilsigma, \cT_{\tilsigma}) \map (\cN\sigma, \cT_{\sigma})\]
of $\Top^?$ is composed of the natural inclusion
\[\cN g : \cN\tilsigma\map \cN\sigma\] 
and of a natural transformation between the composition 
\[\theta : \cT_\sigma\cdot \cN g\Leftarrow \cT_\tilsigma.\]
\begin{comment} %Complicated diagram
\[\xymatrix{
\cN\tilsigma \ar[rr]_{\cT_{\tilsigma}}="1"\ar@/_10pt/[dd]_{\cN g}\ar`d[rd]`^ur[rr]^{\cT_\sigma \cdot \cN g}="2"[rr]&& Top\\
&\\
\cN\sigma\ar@/_15pt/[rruu]_{\cT_{\sigma}}&
\ar@{=>}"2";"1"
}\]
\end{comment}
In particular, for each subsimplex 
\[\sigma_0 = \left(\G_{i_{j_m}}\smap\cdots \smap \G_{i_{j_0}}\right)\]
of $\tilsigma$, we will construct a map
\[\theta_{\sigma_0}:\cT_{\sigma}(\sigma_0)\map \cT_{\tilsigma}(\sigma_0).\]

A point $x$ in $\cT_\sigma(\sigma_0)$ contains tubular neighborhoods for 
\begin{equation}\label{eqn:basic ij0}
\xymatrix{
M^{\G_{i_{j_0}}}\ar@{=>}[r] & M^{\G_{i_{j_0}}} \times PM^{eE_{i_{j_0}}}\times M^{eV_{i_{j_0}}} &\ar[l] M^{\p_\win\G_{i_{j_0}}} \times M^{eE_{i_{j_0}}\du eV_{i_{j_0}}} \ar@{=>}[r] & M^{\p_\win\G_{i_{j_0}}}\times W_{\G_{i_{j_m}}}.
}\end{equation}
It also contains the information necessary to lift these choices through the simplex
\[\G_0\smap\G_1\smap \cdots \smap \G_{i_{j_0}}.\]
A point $y$ in $\cT_{\tilsigma}(\sigma_0)$ considers the same diagram \eqref{eqn:basic ij0} but contains the information to lift through the simplex
\[\G_{i_0}\smap \G_{i_1}\smap \cdots \smap \G_{i_{j_0}}.\]
Hence to go from $\cT_{\sigma}(\sigma_0)$ to $\cT_{\tilsigma}(\sigma_0)$ we need to compose some of the information used to lift the operations. 

From $x$, we will construct a tubular neighborhood for
\[M^{V_{i_r}} \map \left.\left(\frac{M^{V^\win_{i_r}} \times W^{(eE_{i_k}\du eV_{i_k})\setminus eE_{i_r}}}{M^{V^\win_{i_r}}\times M^{eV_{i_{r+1}}}\times M^{eE^{\varphi_{i_{r+1}i_r}}}}\right)\right|_{M^{V^{\win}_{i_r}}\times M^{eV_{i_{r+1}}} \times M^{eE^{\varphi_{i_{r+1}i_r}}}}
\]
for $r=0,\ldots, k-1$. Here $\varphi_{i_ri_{r-1}}$ denotes the composition
\[\varphi_{i_{r}i_{r-1}} : \G_{i_r}\smap \G_{i_r-1}\smap \G_{i_r-2}  \cdots\smap \G_{i_{r-1}}.\]
The point $x$ gives a tubular neighborhood $\psi_s$ for each
\[
M^{V_s}\map \left.\frac{M^{V_s^\win}\times W^{(eE_{i_k}\du eV_{i_k})\setminus eE_s}}{M^{V_s^\win}\times TM^{eV_{s+1}}\times TM^{eE^\varphi_s}}\right|_{M^{V_s^\win}\times M^{eV_{s+1}}\times M^{eE^{\varphi_s}}}.
\]
By first restricting to $M^{V_{i_r}^\win}$ and then adding the identity on
\[\nu^{eE^{\varphi_r}\du\cdots \du eE^{\varphi_{s-1}}}\]
we get a tubular neighborhood for
\[\xymatrix{
M^{V_r^\win} \times M^{eV_s}\times M^{eE^{\varphi_r}\du \cdots \du eE^{\varphi_{s-1}}}
\ar[d]\\
\frac{M^{V_r^\win}\times TW^{(eE_{i_k}\du eV_{i_k})\setminus eE_s}}{M^{V_r^\win}\times TM^{eV_{s+1}}\times TM^{eE^{\varphi_s}}}\times  \nu^{eE^\varphi_r}\times \cdots \times \nu^{eE^{\varphi_{s-1}}}
\ar@{=}[r]&
\frac{M^{V_r^\win}\times TW^{(eE_{i_k}\du eV_{i_k})\setminus eE_r}}{M^{V_r^\win}\times TM^{eV_{s+1}}\times TM^{eE^{\varphi_r}\du \cdots \du eE^{\varphi_s}}}.
}\]
These gives tubular neighborhood for the vertical inclusion of diagram
\[\xymatrix{
M^{V_{i_r}}\ar[rr]\ar[d]\ar@/_100pt/[ddd]&& M^{V_{i_r}^\win}\times W^{\left(eE_{i_k}\du eV_{i_k}\right)\setminus eE_r} \\
M^{V_{i_r}^\win} \times M^{eV_{(i_r)+1}}\times M^{eE^{\varphi_r}}\ar[d]\ar[rur]\\
\vdots\ar[d]\\
M^{V_{i_r}^\win} \times M^{eV_{i_{(r+1)}}} \times M^{eE^{\varphi_{i_{(r+1)}}}\du \cdots \du eE^{\varphi_{(i_r)+1}}}\ar[rruuu]
}\]
twisted by the normal bundles of the diagonal embeddings. In particular, we can compose these embedding and get the required one. Note that since composition and restriction are associative, this operation is also associative. 

For the rest of the data, we join the connections on each $\nu^{eE^{\varphi_s}}$ into connections for the 
\[\nu^{eE^{\varphi_{i_{r+1}i_r}}} =\nu^{eE^{\varphi_{i_{r+1}}}}\times \cdots \times \nu^{eE^{\varphi_{(i_r)+1}}}.\]
We do the same on the propagating flow using that if propagating flows on two bundle $\lambda_A\smap A$ and $\lambda_B\smap B$, give one on $\lambda_A\times \lambda_B$ by 
\[\xymatrix{
\lambda_A\times \lambda_B \ar[rr]^-{\Omega_A\times \Omega_B}
&& \chi^v_c(\lambda_A)\times \chi^v_c(\lambda_B) \ar[r]& \chi^{v(\lambda_A\times \lambda_B)}.
}\]
\end{proof}

%-------------------Homotopy triviality
\begin{proof}[\underline{$\cT_\sigma(\sigma_0)$ is contractible}]
Say
\[\sigma_0 =(\G_{i_k}\smap \cdots\smap \G_{i_0}) \qquad \sigma=(\G_n\smap \cdots \smap \G_0).\]
The space $\cT_\sigma(\sigma_0)$ is a subspace of the following product of contractible space.
\begin{equation}\label{eq:big contractible}
\begin{split}
&\hspace{-30pt}\Tub\left[\eqref{eq:left basic}\right]\times \prod_{r\leq i_0} \Tub\left[\eqref{eq:left sup r}_r\right] \times \prod_r \cP\left[\eqref{eq:left sup r}_r\right]
 \times \prod_{r\leq i_0} \cC\left[\eqref{eq:left sup r}_r\right] \times \Tub\left[\eqref{eq:left sup 0}\right]\\
 & \times \Tub\left[\eqref{eq:right basic}\right] \times \prod_{r} \Tub\left[\eqref{eq:goingup}_r\right]\times \cP\left[\eqref{eq:right basic}\right] 
\end{split}\end{equation}
where the equation number represents its normal bundle, $\cC$ is the space of connection and $\cP$ the space of propagating flows. We claim that $\cT_\sigma(\sigma_0)$ is carved into this space as a homotopy pullback over a contractible diagram. 

We will prove this by induction on the difference between $i_0$ and $0$. If $i_0=0$ then we are only considering the diagram
\[\xymatrix{
M^{\G_{i_0}}\ar@{=>}[r]& M^{\p_\win \G_{i_0}}\times PM^{eE_{i_0}}\times M^{eV_{i_0}}&\ar[l] M^{\p_\win \G_{i_0}}\times M^{eE_{i_0}\du eV_{i_0}}\ar@{=>}[r]& M^{\p_\win\G_{i_0}}\times W_{\G_{i_k}}.
}\]
There are no restrictions and 
\[\cT_{\G_n\smap\cdots\smap \G_{i_0}}(\G_{i_k}\smap \cdots\smap \G_{i_0})\]
 is equal to the contractible space of \eqref{eq:big contractible}.

Assume that we have shown that 
\[\cT_{\G_n\smap\cdots\smap \G_1}(\G_{i_k}\smap \cdots \smap\G_{i_0})\]
is contractible. Assume also that we have that the map
\[\cT_{\G_n\smap\cdots\smap \G_1}(\G_{i_k}\smap \cdots \smap\G_{i_0})\map \Tub\left(
M^{V_1}\smap \frac{TM^{V_1}\times TW^{(eE_{i_k}\du eV_{i_k})\setminus eE_1}}{TM^{V_1}}\times M^{eH_1}\right). \]
is a fibration (which is clearly true when $i_0=0$).
When constructing $\cT_\sigma(\sigma_0)$, we are considering one extra line in our diagram. In particular, we need to add the information necessary to lift from the line of $\G_1$ to the line of $\G_0$. This also gives two extra restrictions. 

Let
\[\Tub^{\word{square}}(\varphi_0: \G_{1}\smap \G_0) \subset \Tub\left(M^{V_1}\smap \frac{M^{V^\win_{1}}\times TW^{(eE_{i_k}\du eV_{i_k})\setminus eE_{1}}}{TM^{V_{1}}} \times M^{eH_{1}}\right)
\]
be the subspace which contains all tubular neighborhood which lift up 
\[\xymatrix{
M^{V_0}\ar[d]\ar[r]&\left.\frac{T\left(M^{V_0^\win}\times W^{(eE_{i_k}\du eV_{i_k})\setminus eE_1}\right)}{TM^{V_0}}\right|_{M^{V_0^\win\du eV_1}}\times M^{eH_0}\times M^{eE^{\varphi_0}}\ar[d]\\
M^{V_1}\ar[r]& \frac{M^{V^\win_{1}}\times TW^{(eE_{i_k}\du eV_{i_k})\setminus eE_{1}}}{TM^{V_{1}}} \times M^{eH_{1}}
}\]
as in \eqref{pullback r r+1}. Note that the restriction to the top map is a fibration
\[\Tub^{\word{square}}(\varphi_0)\map \Tub\left(M^{V_0}\smap\left.\frac{T\left(M^{V_0}\times M^{(eE_{i_k}\du eV_{i_k})\setminus eE_1}\right)}{TM^{V_0}}\right|_{M^{V_0^\win\du eV_1}}\times M^{eH_0}\times M^{eE^{\varphi_0}}\right).\]
Similarly let $\Tub^{\word{triangle}}(\varphi_0)$ be the subset of
\[\begin{split}
 \cC(\nu^{eE^{\varphi_0}})&\times \Tub\left(TM^{V_0}\smap \frac{M^{V_0}\times TW^{(eE_{i_k}\du eV_{i_k})\setminus eE_{1}}}{TM^{V_0}}\times M^{eH_0}\times M^{eE^{\varphi_0}}\right)\\
 &\times \Tub\left(M^{V_0}\smap \frac{M^{V^\win_0}\times TW^{(eE_{i_k}\du eV_{i_k})\setminus eE_{1}}}{M^{V^{\win}_0}\times T\left(M^{eV_1}\times M^{eE^{\varphi_0}}\right)}\right)
\end{split}\]
which contains the triples which induce a tubular neighborhood on the top map of \eqref{restriction 2}. By the appendix, we know that both these spaces are contractible.

The first restriction can be expressed by a pullback diagram
\[\xymatrix{
\word{Strict}^1_{\sigma}(\sigma_0)\ar[r]\ar[d]&\Tub^{\word{square}}(\varphi_0)\ar[d]\\
\cT_{\G_n\smap\cdots \smap\G_1}(\sigma_0)\ar[r]&\Tub\left(
M^{V_1}\smap \frac{TM^{V^\win_1}\times TW^{(eE_{i_k}\du eV_{i_k})\setminus eE_1}}{TM^{V_1}}\times M^{eH_1}\right).
}\]
Since the bottom arrow is a fibration, $\word{Strict}^1_\sigma(\sigma_0)$ is also the homotopy pullback and hence is  contractible.
The second restriction can then be expressed as a second pullback diagram.
\[\xymatrix@C=10pt{
\cT_{\sigma}(\sigma_0) \ar[r]\ar[d]&\Tub^{\word{triangle}}(\varphi_0)\ar[d]\\
\word{Strict}^1_{\sigma}(\sigma_0) \ar[r]&\Tub\left(M^{V_0} \smap \frac{TM^{V^\win_0}\times TW^{(eE_{i_k}\du eV_{i_k})\setminus eE_1}}{TM^{V_0}} \times M^{eH_0}\times M^{eE^{\varphi_0}}\right)
}\]
Again the bottom arrow is a fibration since it is a composition of two fibrations. In particular, $\cT_{\sigma}(\sigma_0)$ is contractible.
\end{proof}

\subsection{The definition of the space $\Tub(\sigma)$.}
\label{sub:spaces}\label{sec:thickening}
In section \ref{sub:higher simplices}, we have defined a functor
\[\cT : (\cN\aFat)^{op} \map \Top^?\]
from the category of simplices of $\Fat$ to the category of diagrams in spaces. 
The functor $\cT$ takes a $k$-simplex $\sigma$ to the pairs $(\cN \sigma, \cT_\sigma)$
where $\cN\sigma$ is the category of subsimplices of $\sigma$ and $\cT_\sigma$ includes the necessary choices to determine the operations for all subsimplices of $\sigma$. We will now put all these choices in a single space.

Consider the functor
\[\Tub : \cN\aFat_1 \map \Top\]
which sends $\sigma$ to the homotopy limit
\[\Tub(\sigma) = \holim{\cN\sigma} \cT_c.\]
For a description of homotopy limits and homotopy colimits see \cite{BousfieldKan}.

\begin{lemma}
The space $\Tub(\sigma)$ is contractible.
\end{lemma}
\begin{proof}
This space $\Tub(\sigma)$ is a homotopy limit over a contractible diagram of contractible spaces.
\end{proof}

By construction, a point $x$ of $\Tub(\sigma)$ determines for each subsimplex $\sigma_0$ a map
\[ x_{\sigma_0} : |\cN(\sigma_0)| \map \cT_\sigma(\sigma_0)\]
between the geometric realization of $\cN(\sigma_0)$ and the space $\cT_{\sigma}(\sigma_0)$. In particular, it picks a continuous family of operations parameterized by $|\cN(\sigma_0)|\cong \Delta^k$.
In fact, a point in the contractible space $\Tub(\sigma)$ fixes the operations for every fat graph in $\sigma$ as well as the higher homotopies between them. 

\begin{example}
If $\sigma=\G$ then 
\[\Tub(\sigma)= \holim{\cN\sigma} {\cT_\sigma} = \cT_\sigma(\sigma)\]
which determines an operation for $\G$.
\end{example}
\begin{example}
If $\sigma = \G_1\smap \G_0$ then $\cN\sigma$ is a category
with three elements $\G_0, \G_1,\sigma$ and the following morphisms
\[\G_0 \map \sigma \imap \G_1.\]
The space
\[\holim{\cN\sigma}(\cT_\sigma)=\word{homotopy\ pullback}\left(\cT_{\sigma}(\G_0)\smap\cT_{\sigma}(\sigma) \imap \cT_{\sigma}(\G_1)\right)\]
is the space of triples $(a,\gamma, b)$. Here $a \in \cT_{\sigma}(\G_0)$, $b\in \cT_{\sigma}(\G_1)$ and $\gamma$ is a path in $\cT_{\sigma}(\sigma)$ which starts at the image of $a$ and ends at the image of $b$ in $\cT_{\sigma}(\sigma)$.
By our definition of $\cT_{\sigma}$, $a$ fixes an operation 
\[M^{\p_\win \G_0}\times W^{eE_0\du eV_0} \maplu{\mu_a} \Thom(\kappa_{\G_0}).\]
The point $b$ picks a pair of compatible operations
\[\xymatrix{
M^{\p_\win \G_0}\times W^{eE_1\du eV_1} \ar[d]\ar[r]^{\mu_b|_{\G_0}}& \Thom(\kappa_{\G_1}|_{M^{\G_0}})\ar[d]\\
M^{\p_\win \G_1}\times W^{eE_1\du eV_1} \ar[r]^{\mu_b}& \Thom(\kappa_{\G_1}).
}\]
Finally $\gamma$ is a homotopy between
\begin{equation}\label{eq:10}
\xymatrix{
M^{\p_\win \G_0}\times W_{\G_1} \ar[r]^-{\cong}&
M^{\p_\win \G_0}\times W^{\G_0}\times 
\frac{W_{\G_1}}{W_{\G_0}} \ar[r]^-{\mu_a}& \Thom\left(\kappa_{\G_0} \times \frac{W_{\G_1}}{W_{\G_0}}\right)\ar[r]^-{\cong}
& \Thom\left(\kappa_{\G_1}|_{M^{\G_0}}\right).
}\end{equation}
and $\mu_b|_{\G_0}$.
\end{example}

%Third example a three simplex
\begin{example}
Say $\sigma=\left(\G_2\smap \G_1\smap \G_0\right)$ is a 2-simplex. We then get the following diagram of spaces
\[\xymatrix@=10pt{
\cT_\sigma(\G_2) \ar[rrr]\ar[rrrrdd]\ar[rrrddd]&&& \cT_\sigma(\G_2\smap \G_1)\ar[rdd] &&&\ar[lll] \cT(\G_1)\ar[lldd]\ar[ddd]\\\\
&&&&\cT_{\sigma}(\sigma)\\
&&&\cT_\sigma(\G_2\smap \G_0)\ar[ur]&&& \ar[ull]\cT_\sigma (\G_1\smap \G_0) \\\\\\
&&&&&&\ar[llluuu]\ar[lluuuu]\cT(\G_0)\ar[uuu].
}\]
A point $x$ of  $\Tub(\sigma)$ picks
\begin{itemize}
\item $(\mu_0,\mu_1,\mu_2)$ where $\mu_i$ is an element of $\cT_\sigma(\G_i)$.
\item $(h_{01},h_{02},h_{12})$ where $h_{ij}$ is a path in $\cT(\G_j\smap \G_i)$.
\item $\kappa_{012}$ a two-cell in $\cT_{\sigma}(\sigma)$.
\end{itemize}
Hence $x$  determines an operation 
\[\mu_0 : M^{\p_\win\G_0}\times W_{\G_0}\map \Thom\left(\kappa_{\G_0}\right).\]
It also fixes an operation $\mu_1$for $\G_1$ which restricts to an operation on $\G_0$.
\[\xymatrix{
M^{\p_\win \G_0}\times W_{\G_1}\ar[d]\ar[r]^{(\mu_1)|_{\G_0}}& \Thom(\kappa_{\G_1}|_{M^{\G_0}})\ar[d]\\
M^{\p_\win \G_1}\times W_{\G_1}\ar[r]^{\mu_1}& \Thom(\kappa_{\G_1}).
}\]
 It also determines an operation $\mu_2$ for $\G_2$ which restricts to one for both $\G_1$ and $\G_0$. 
\[\xymatrix{
M^{\p_\win \G_0}\times W_{\G_2} \ar[d]\ar[rr]^{(\mu_2)|_{\G_0}}&& \Thom(\kappa_{\G_2}|_{M^{\G_0}})\ar[d]\\
M^{\p_\win \G_1}\times W_{\G_2}  \ar[d]\ar[rr]^{(\mu_2)|_{\G_1}}&& \Thom(\kappa_{\G_2}|_{M^{\G_1}})\ar[d]\\
M^{\p_\win \G_2}\times W_{\G_2}  \ar[rr]^{\mu_2}&& \Thom(\kappa_{\G_2})
}\]
It then chooses a homotopy $h_{10}$ between the operation of diagram \eqref{eq:10} and $\mu_{1}$. It also picks a homotopy $h_{20}$ between the two maps of 
\[\xymatrix{
M^{\p_\win \G_0}\times W_{\G_2} \ar[r]\ar@/^20pt/[rrr]^{\mu_2|_{\G_0}}&M^{\p_\win \G_0}\times W_{\G_0} \times \frac{W_{\G_2}}{W_{\G_0}} \ar[r]^{\mu_0}& \Thom\left(\kappa_{\G_0}\times \frac{W_{\G_2}}{W_{\G_0}} \right)\ar[r]& \Thom(\kappa_{\G_2}|_{M^{\G_0}}).
}\]
The point $x$ also gives a homotopy $h_{12}$ between the two operations of
\[\xymatrix{
M^{\p_\win \G_1}\times W_{\G_2} \ar[r]\ar@/^20pt/[rrr]^{\mu_2|_{\G_1}}&M^{\p_\win \G_1}\times W_{\G_1} \times \frac{W_{\G_2}}{W_{\G_1}} \ar[r]^{\mu_1}& \Thom\left(\kappa_{\G_1}\times \frac{W_{\G_2}}{W_{\G_1}} \right)\ar[r]& \Thom(\kappa_{\G_2}|_{M^{\G_1}})
}\]
which restricts to a homotopy $h_{12}|_{\G_0}$ between
\[\xymatrix{
M^{\p_\win \G_0}\times W_{\G_2} \ar[r]\ar@/^20pt/[rrr]^{\mu_2|_{\G_0}}&M^{\p_\win \G_0}\times W_{\G_1} \times \frac{W_{\G_2}}{W_{\G_1}} \ar[r]^{(\mu_1)|_{\G_0}}& \Thom\left(\kappa_{\G_1}|_{M^{\G_0}}\times \frac{W_{\G_2}}{W_{\G_1}} \right)\ar[r]& \Thom(\kappa_{\G_2}|_{M^{\G_0}})
}\]
Finally $x$ contains a triangle
\[|\cN\sigma| \cong \Delta^2 \map \cT_\sigma(\sigma)\]
in $\cT_\sigma(\sigma)$ which gives a 2-homotopy of operations on $\G_0$ between the composition of $h_{01}$ with the restirction $h_{12}|_{\G_0}$ and $h_{02}$.
\end{example}

In general, a point $x$ in $\Tub(\sigma)$ determines for each $\sigma_0\subset \sigma$ a map
\[x_{\sigma_0} : |\cN\sigma_0| \map \cT_{\sigma}(\sigma_0).\]
These maps are compatible which means that for any inclusion $g: \sigma_1\subset \sigma_0$, we have a commutative diagram
\[\xymatrix{
|\cN\sigma_1| \ar[r]\ar[d]_{x_{\sigma_1}} &|\cN\sigma_0| \ar[d]^{x_{\sigma_0}}\\
\cT_{\sigma}(\sigma_1) \ar[r]^{\cT_{\sigma}(g)}& \cT_{\sigma}(\sigma_0)
 }\]

Lets translate proposition \ref{prop:key} to the spaces $\Tub(\sigma)$.

\begin{proposition}\label{prop:tub}
There is a functor
\[\Tub : \cN\Fat_1\map \Top\]
which has the following properties.
\begin{enumerate}
\item For any simplices
\[\sigma_0=\left(\G_{i_k}\smap \cdots \smap \G_{i_0}\right) \subset \sigma=\left(\G_n\smap \cdots \smap\G_0\right)\]
there are maps
\[\xymatrix@R=15pt{
\Tub(\sigma)\times |\cN\sigma_0|\times W_{\G_{i_k}}\times M^{\p_\win \G_{0}} \ar[d]\ar[rr]^-{\mu_{\sigma_0\subset\sigma}|_{\G_0}}&&\Thom\left(\left.\kappa_{\G_{i_k}}\right|_{M^{\G_0}}\right)\ar[d]\\
\vdots\ar[d]&&\vdots \ar[d]\\
\Tub(\sigma)\times |\cN\sigma_0|\times W_{\G_{i_k}}\times M^{\p_\win\G_{i_0}} \ar[rr]^-{\mu_{\sigma_0\subset\sigma}}&&\Thom\left(\left.\kappa_{\G_{i_k}}\right|_{M^{\G_{i_0}}}\right).
}\]
\item For a pair 
\begin{eqnarray*}
\sigma_1 &=& \left(\G_{i_{j_l}}\smap \cdots \smap \G_{i_{j_0}}\right)\\
\sigma_0&=&\left(\G_{i_k}\smap\cdots \smap \G_{i_0}\right)
\end{eqnarray*}
of subsimplices of $\sigma$, we get commutative diagrams
\[\xymatrix{
\Tub(\sigma)\times |\cN\sigma_1|\times W_{\G_{i_{j_1}}} \times\frac{W_{\G_{i_k}}}{W_{\G_{i_{j_1}}}} \times M^{\p_\win\G_{r}} \ar[rr]^-{\mu_{\sigma_1\subset \sigma}|_{\G_{r}}}\ar[d]&&\Thom\left(\left.\kappa_{\G_{i_{j_l}}}\right|_{M^{\G_{r}}} \times \frac{W_{\G_{i_k}}}{W_{\G_{i_{j_1}}}}\right) \ar[d]\\
\Tub(\sigma)\times |\cN\sigma_0|\times W_{\G_{i_k}} \times M^{\p_\win\G_{r}} \ar[rr]^-{\mu_{\sigma_0\subset\sigma}}&& \Thom\left(\left.\kappa_{\G_{i_{k}}}\right|_{M^{\G_r}}\right)
}\]
for $r\leq i_0$. Here the vertical maps are given by the usual splitting of the Euclidean spaces and by the natural inclusion of subsimplices of $\sigma_1$ as simplices of $\sigma_0$.
\item For a subsimplex $\sigma_0$ of both $\tilsigma\subset\sigma$, we get  commutative diagrams
\[\xymatrix{
\Tub(\sigma) \times |\cN\sigma_0|\times M^{\p_\win\G_{r}}\times  W_{\G_{i_k}}  \ar[rr]^{\mu_{\sigma_0\subset\sigma}|_{M^{\G_r}}} \ar[d]&&\Thom\left(\left. \kappa_{\G_{i_k}}\right|_{M^{\G_{r}}}\right)\ar[d]\\
\Tub(\tilsigma)\times |\cN\sigma_0|\times W_{\G_{i_k}} \ar[rr]^{\mu_{\sigma_0\subset\tilsigma}|_{M^{\G_r}}} &&\Thom\left(\left. \kappa_{\G_{i_k}}\right|_{M^{\G_{r}}}\right)
}\]
for $r\leq i_0$.
\item $\Tub(\sigma)$ is contractible.
\end{enumerate}
\end{proposition}

\subsection{Second thickening of $\Fat$ to include the choices of $\Tub(\sigma)$.}

We are now ready to twist these choices of tubular neighborhoods into our model for the classifying spaces of the mapping class groups. We will construct a category $\til{\aFat}$ which lives above the category of simplices in $\aFat_1$. The fiber over a simplex $\sigma$ will be exactly $\Tub(\sigma)$. We then use the fact that $\Tub$ is a functor to move from one fiber to the next.

For this, we use Thomason's categorical construction for the homotopy colimits \cite{Thomason}. We let
\[\til{\aFat} = (\cN\aFat_1)^{op} \int  \Tub\]
be a category enriched over topological spaces.   An object of $\til{\aFat}$ is a pair $(\sigma, x)$ where $\sigma$ is a simplex of $\aFat$ and where $x$ is an element of $\Tub(\sigma)$. A morphism 
\[ (\sigma_0, x) \map (\sigma_1, y)\]
is simply an inclusion of morphism 
\[g : \sigma_1 \map \sigma_0\]
so that
\[y= \cT(g)(x).\]
And so $y$ is the restriction of the choices that $x$ contain for $\sigma$ to choices for $\sigma_1$. 

\begin{lemma}
The functor
\[\til{\aFat} \map \aFat\]
which sends a pair $(\sigma, x)$ to the first vertex $\G_k$ of $\sigma=(\G_k\smap\ldots \smap \G_0)$ induces a homotopy equivalence.
\end{lemma}
\begin{proof}
By Thomason \cite{Thomason}, the geometric realization of $\til{\aFat}$ is the homotopy colimit of the functor $\Tub$. Since the $\Tub(\sigma)$ are contractible, we get that the forgetful map
\[\hocolim{(\cN\aFat)^{op}} F \map |\cN\aFat|\]
is a homotopy equivalence.

Consider the functor
\[\psi : (\cN\aFat)^{op} \map \aFat\]
which sends a simplex $\sigma$ to its first vertex $\G_k$ and a morphism
\[\left(\G_n \map \ldots \map \G_0\right) \map \left(\G_{i_k}\map \ldots \map \G_{i_0}\right)\]
to the composition 
\[\G_n\map \G_{n-1}\map\ldots \map \G_{i_k}\]
whenever $i_k\neq n$ and to the identity $\G_n\map \G_n$ whenever $i_k=n$. It suffices to show that the geometric realization of $\psi$ is a homotopy equivalence. We shall show that for any $\G \in \aFat$ the geometric realization of the category $\G\setminus \psi$ is contractible.

We will define three functors with two natural transformations
\[\xymatrix{
\G\setminus \psi\ar@/_40pt/[rrr]^{\mathit{cst}}="cst" \ar@/^40pt/[rrr]_{Id}="Id" \ar[rrr]^G="GId"_G="Gcst"&& & \G\setminus \psi \ar@{=>}"GId";"Id"^{\theta_1} \ar@{=>}"Gcst";"cst"^{\theta_2}
}\]
The functors send an object
\[(\sigma,\varphi) =(\G_k\smap \cdots \smap \G_0, \varphi:\G\smap \G_l)\]
to
\begin{eqnarray*}
Id\left(\sigma,\varphi\right) &=& \left(\sigma,\varphi \right)\\
G\left( \sigma,\varphi \right) &=& \left(\G\maplu{\varphi}\G_k\smap \G_{k-1}\smap\ldots \smap \G_0, \G\maplu{Id_\G} \G \right)\\
\mathit{cst}\left(\sigma,\varphi \right) &=& \left(\G\maplu{Id_\G} \G, \G\right)
\end{eqnarray*}
We extend these in the natural way to functors. The first natural transformation takes 
\[
\theta_1\left(\sigma,\varphi \right)=\left[ d_{k(k-1)\cdots 0} : \left(\G\maplu{\varphi}\G_k\smap \ldots \smap \G_0, \G\maplu{Id}\G \right)\map \left(\sigma,\varphi\right)\right]
\]
where $d_{k(k-1)\cdots 0}$ is the inclusion of $\sigma$ into 
\[\G\maplu{\varphi}\G_k\smap \ldots \smap \G_0.\]
Similarly the second natural transformation takes
\[
\theta_2\left(\sigma,\varphi\right)=\left[d_{k+1}\left( \G\maplu{\varphi}\G_k\smap \ldots \smap \G_0,\G\maplu{Id}\G\right)\map \left( \G, \G\maplu{Id}\G\right)\right]
\]
Hence the identity on $\G\setminus \psi$ and a constant map are linked by natural transformations. This shows that $|\G\setminus \psi|$ is contractible and hence $\psi$ is a homotopy equivalence as claimed.
\end{proof}

\begin{remark}
Similarly the functors
\[\cN\aFat \map \Fat\]
which sends a simplex $(\G_k\smap\ldots \smap \G_0)$ to its last vertex $\G_0$ also induces a homotopy equivalence on the geometric realization.
\end{remark}

We now have a thicker category $\til{\aFat}$ which realizes to a space with the correct homotopy typer. This new model contains all the information we need to construct the Thom collapses and hence the operations in a compatible way.

\subsection{Construction of the Thom collapse.}
\label{sub:construction}

Using homotopy colimits, we will construct categories which twists the spaces $M^{\p_\win\G}$, $M^{\G}$ and $\kappa_\G$ above the all admissible fat graphs. 

Consider the functor
\[M^{-} :\tilfatop \map \Top\]
which sends an object 
\[\sigma=\left(\G_n\smap \cdots \smap \G_0\right)\] 
to the space $M^{\G_0}$ of piecewise smooth maps from the geometric realization of $\G_0$ to $M$. We again use Thomason's construction \cite{Thomason} and consider the category 
\[ \tilfatop\int M^{-}\]
whose objects are pairs $(\sigma,f:\G_0\map M)$ and whose morphisms 
\[ \varphi: (\sigma,f:\G_0 \smap M) \map (\til{\sigma}, \til{f}:\tilG_0\smap M)\]
are the morphisms 
\[\til{\sigma}=\left(\G_{i_k}\subset \cdots \subset \G_{i_0}=\tilG_0\right)\subset \sigma =\left(\G_n\smap\cdots \subset \G_0 \right)\]
of $\tilfat$ with $g=f\cdot \varphi_{i_00}$
where $\varphi_{i_00}$ is the composition
\[\varphi_{i_00}:\G_{i_0}\smap \cdots\smap \G_0\]
The homotopy type of the geometric realization of this category is the homotopy colimit of the functor $M^{-}$ which is homotopy equivalent to the Borel construction
\[\Mtop = \left|\left(\aFat\right)^{op}\int M^{-}\right| \simeq \DU{S} \left(E\Diff(S:\p) \underset{\Diff(S;\p)}\times \Map(S,M)\right)\]
of pairs $(\Sigma, f:\Sigma\map M)$ where $\Sigma$ is a Riemann surfaces. 

The space  
\[\left| \tilfat \int M^-\right|\]
will be the basis for the virtual bundle $\kappa$. The Thom spectrum of $\kappa$ will be the target of the generalized Thom collapse whose construction is the goal of this entire section.

For each simplex 
\[\sigma = \left(\G_n\smap \cdots \smap \G_0\right),\]
we let $\kappa^+_\sigma=\left.\kappa_{\G_n} \right|_{\G_0}$ which is the target bundle of the operation associated to the diagram
\[\xymatrix{
M^{\G_0}\ar[r]& M^{\p_\win\G_0}\times PM^{eE_0}\times M^{eV_0}&\ar[l] M^{\p_\win\G_0}\times M^{eE_0\du eV_0} \ar[r]& M^{\p_\win\G_0}\times W_{\G_n}.
}\]
We will denote by $\kappa_\sigma^-$ the trivial bundle
\[W_{\G_n}\times M^{\G_0}\map  M^{\G_0}.\]

We claim that these bundles glue to form a bundle over the space 
\[\left| \tilfat \int M^-\right|\]
To keep track of how we patch these bundles together, we consider the following category $\VB$ of virtual bundles. The objects of $\VB$ are pairs $(X,\kappa_+,\kappa_-)$ where $X$ is a topological space and the $\kappa_\pm$'s are bundles above $X$. A morphism 
\[ [f,\theta,\psi^+,\psi^-] :(X,\kappa_+,\kappa_-)\map (Y, \lambda_+,\lambda_-)\]
is an equivalence class of quadruples where $f:X\smap Y$ is a map of spaces, $\theta$ is a bundle above $X$ and the $\psi^{\pm}$ are isomorphisms
\[\psi^\pm : f^*\lambda_\pm \oplus \theta \maplu{\cong}\kappa_\pm\]
of bundles. Two such quadruples $(f_i,\theta_i, \psi_i^+,\psi_i^-)$ represent the same morphism exactly when $f_1=f_2$ and when there is an isomorphism
\[\zeta_\theta :\theta_1\map \theta_2\]
of bundles which sends the $\psi_1^\pm$ to the $\psi_2^\pm$. The objects are discreet but the morphisms are topologized as the quotient
\[\left(\DU{(f, \theta)} \mathit{Isom}(\theta\oplus \kappa_+ , \lambda_+)\times \mathit{Isom}(\theta\oplus \kappa_- , \lambda_-)\right)/\sim.\]
Composition sends
\[\xymatrix{
(X,\kappa_+,\kappa_-) \ar[rrr]^{[f_X,\theta_X,\psi_X^+,\psi_X^-]}&&&
(Y,\mu_+,\mu_-)\ar[rrr]^{[f_Y,\theta_Y,\psi_Y^+,\si_Y^-]}&&&
(Z,\lambda_+,\lambda_-)
}\]
to the morphism
\[\xymatrix{(X,\kappa_+,\kappa_-)\ar[rrrrr]^{[f_Y\circ f_X, \left(f_X^*\theta_Y\right)\oplus \theta_X, \psi^+,\psi^-]}&&&&& (Z, \lambda_+,\lambda_-)}\]
where $\psi^\pm$ is the composition
\[\xymatrix{
 (f_Y\circ f_X)^*\lambda_{\pm} \oplus \left(f_X^*\theta_y \oplus \theta_X\right)\ar[r]& f_X^*\left(f_Y^*(\lambda_\pm)\oplus \theta_Y\right)\oplus \theta_X \ar[rr]^{f_X^*(\psi_Y^\pm)\oplus Id} &&f_X^*(\mu_\pm) \oplus \theta_X \maplu{\psi_X^\pm} \kappa_{\pm}
}\]

\begin{lemma}
There is a functor
\[L\ :\ \cN \Fat_1^{op} \map \VB\]
which is defined on objects as
\[L \left(\G_k\smap\ldots \smap \G_0,x\right)= (M^{\G_0}, \kappa_{\sigma}^+,\kappa_{\sigma}^-).\]
\end{lemma}

\begin{proof}
It suffices to define $L$ on morphisms and to prove associativity. Take any morphism
\[g: \sigma=(\G_n\smap \cdots \smap \G_0) \map \sigma_0=(\G_{i_k}\smap \dots \smap \G_{i_0}).\]
of $\cN\Fat_1^{op}$. We will define a continuous map
\[\Split(\varphi_n)\times \cdots  \Split(\varphi_1) \map \mathit{Isom}\left(\kappa_{\sigma_0}^+\oplus \theta, f^*\kappa_{\sigma}^+\right)\times \mathit{Isom}\left(\kappa_{\sigma_0}^-\oplus \theta, f^*\kappa_{\sigma}^-\right)\]
which will define $L$ on the morphisms above $g$.

Let 
\begin{eqnarray*}
{\varphi_{i_00} : \G_{i_0}\maplu{\varphi_{i_0}}\cdots\maplu{\varphi_{1}} \G_{0}}\qquad 
{\varphi_{ni_k} : \G_{n}\maplu{\varphi_{i_0}}\cdots\maplu{\varphi_{1}} \G_{i_k}}
\end{eqnarray*}
be the compositions.  Let 
\[\theta =\left(M^{\G_0}\times \frac{W_{\G_n}}{W_{\G_{i_k}}}
\right)\]
Using the composition of $\Fat_1$, we have a map
\[\Split(\varphi_n)\times \cdots \times \Split(\varphi_{(i_k)+1})\times \cdots\times \Split(\varphi_1)  \map \Split(\varphi_{ni_k}).\]
An element of  $\Split(\varphi_{ni_k})$  gives an isomorphisms
\[M^{\G_0}\times W_{\G_{i_k}} \times\frac{W_{\G_n}}{W_{\G_{i_k}}} \map M^{\G_0}\times W_{\G_n}\]
on the $\kappa^-$'s.

To get the isomorphisms on the positive part of the bundle, we consider  the diagram
\[\xymatrix@C=10pt{
M^{\G_0}\ar[r]\ar[dd]\ar[rd]
&M^{\p_\win\G_0}\times PM^{eE_0}\times M^{eV_0}\ar[d]
&\ar[l] M^{\p_\win\G_0}\times M^{eE_0\du eV_0}\ar[r] \ar[d]\ar[rd]
&M^{\p_\win\G_0}\times W_{\G_n}\ar@{<->}[d]^{\cong}\\
& M^{\p_\win\G_{0}}\times PM^{eE_0}\du M^{eE^{\varphi_{i_00}} \du eV_{i_0}} \ar[d]
&\ar[l] M^{\p_\win\G_0}\times M^{eE_{i_0}\du eV_{i_0}} \ar[r]\ar[d]
&M^{\p_\win\G_0}\times W_{\G_{i_k}} \times \frac{W_{\G_n}}{W_{\G_{i_k}}}\ar[d]\\
M^{\G_{i_0}}\ar[r]
&M^{\p_\win\G_{i_0}}\times PM^{eE_{i_0}}\times M^{eV_{i_0}}
&\ar[l] M^{\p_\win\G_{i_0}}\times M^{eE_{i_0}\du eV_{i_0}}\ar[r]
&M^{\p_\win\G_{i_0}}\times W^{\G_{i_k}} \times \frac{W_{\G_n}}{W_{\G_{i_k}}}.
}\]
The target of the operation associated to the bottom row is 
\[\kappa^+_{\sigma_0}\oplus \left(\frac{W_{\G_n}}{W_{\G_{i_k}}}\times M^{\G_{i_0}} \right)\]
All the squares between the second and third row are pullbacks and hence the target of the middle row is
\[\kappa^+_{\sigma_0}|_{M^{\G_0}} \oplus \left(\frac{W_{\G_n}}{W_{\G_{i_k}}}\times M^{\G_{0}} \right).\]
By definition the target of the first row is $\kappa_\sigma$. However, the first and second rows have isomorphic targets and an isomorphisms is specified by the diagram. This gives the required isomorphism
\[\kappa^+_{\sigma_0}|_{M^{\G_0}} \oplus\theta \maplu{\cong} \kappa^+_\sigma. \]

The associativitiy of the composition of morphisms in $\Fat_1$ makes this assignment associative.
\end{proof}

The functor 
\[\tilfatop\map \cN\Fat_1^{op} \maplu{L} \VB \]
gives us a virtual bundle $\lambda$ above the space
\[\left|\tilfatop\int M^{-}\right|.\]
This bundle is the twisting of the virtual bundles $(\kappa_\G^+,\kappa_\G^-).$
The Thom spectrum of this virtual bundle $(\kappa_\G^+,\kappa_\G^-)$ is
\[\Thom(\kappa_\G^+,\kappa_\G^-) =  \Sigma^{-W_{\G}} \Thom(\kappa_\G^+)\]
the desuspension by the Euclidean space $W_{\G}$ of the Thom space of the bundle $\kappa_\G^+$. The Thom spectrum
\[\Thom(\kappa) = \hocolim{\sigma\in\tilfatop} \Thom(\kappa_\sigma^+,\kappa_\sigma^-)\]
of $\kappa$ is the homotopy colimit of the Thom spectrum of the various pieces.

Consider the functor 
\[M^{\p_\win -} : \tilfatop\map Top\]
which sends an object
\[\left(\G_n\smap\G_{n-1}\smap \cdots \smap \G_0, x\right)\]
to the space $M^{\p_\win\G_0}$ and a morphism
\[\left(\G_n\smap \cdots \smap \G_0, x\right)\map \left(\G_{i_k}\smap \cdots \smap \G_{i_0}, y\right)\]
to the map
\[M^{\p_\win\G_0}\map M^{\p_\win\G_{i_0}}\]
corresponding to the composition $\varphi_{i_00}:\G_{i_0}\smap \G_0$.
Consider the category
\[\tilfatop\int M^{\p_\win-}\]
whose objects are triples $(\sigma, x, f)$ where 
\[\sigma= (\G_n\smap \cdots \smap \G_0)\]
is a simplex of $\aFat$, x is in $\Tub(\sigma)$ and where $f:\p_\win\G_0\smap M$ is a piecewise smooth map. The morphisms are simply ``coinclusion" of simplices with $x$ and $f$ following.

\begin{lemma}\label{lem: pwin}
For  each component $\aFat_{S}$, the geometric realization of 
\[\left|\tilfatop_{S}\int M^{\p\win-}\right|\simeq \left|\aFat\right| \times \Map(\p_\win S,M).\]
\end{lemma}
\begin{proof}
By Thomason \cite{Thomason}, the category $\tilfatop\int M^{\p_\win - }$ is a categorical model for homotopy colimit of the functor $M^{\p_\win - }$. We will show that the geometric realization of this functor is homotopy equivalent to the constant functor $M^{\p_{\win S}}$ on the component $|\aFat_S|$.

Fix an open-closed cobordism $S$. Consider the category $\cG_\win^S$ of graph models for $\p_\win S$. More precisely, the objects  of $\cG_\win^S$ are fat graphs $G$ with an ordering of its connected components $\pi_0G$. We ask that the ith component of $G$ be a single vertex whenever the  ith incoming boundary of $S$ is an interval. Similary the ith component of $G$ is a circle with a single leaf as a starting point whenever the ith incoming boundary of $S$ is a circle.

We first see that the functor $M^{\p_\win -}$ factors through $\cG_\win^S$
\[\xymatrix{
\tilfatop\ar[r]^{\p_\win - }& \cG_\win^S \ar[r]^{M^-}& \Top
.}\]
However, we claim that the category $\cG_\win^S$ realizes to a contractible category. Hence since all the $M^{\p_\win\G}$ are homotopy equivalent, we get that
\[\hocolim{\tilfatop_S} M^{\p_\win -} \simeq |\tilfatop_S| \times M^{\p_\win S}.\]

Lets now show that the geometric realization of  $\cG_\win ^S$ is contractible. Notice that since the morphisms preserve the ordering of the connected components of the graphs $G$, it suffices to prove that the category $\cG_\win^S$ is contractible for surfaces $S$ with a single boundary. Because the category depends only on the incoming boundary of $S$, we only need to cconsider two categories $\cG^{S^1}$ which is the category of ``oriented circle graphs" and $\cG^{I}$ which is the category of graphs with a single vertex and no edges. The second one is clearly contractible as all objects are uniquely isomorphic.

It therefore suffices to conisder $\cG^{S^1}$. We have three functors and two natural transformations
\[\xymatrix{
\cG^{S^1}\ar[rrr]^{F}="Fup"_{F}="Fdown"\ar@/^40pt/[rrr]_{Id}="Id" \ar@/_40pt/[rrr]^{cst}="cst" &&&\cG^{S^1} \ar@{=>}"Fup";"Id"^{\theta_1} \ar@{=>}"Fdown";"cst"_{\theta_2}
}\]
The first functor is the identity $Id$. The second one $F$ adds a single edge at the beginning of the circle. The third one, $cst$, sends everything to the circle with a single edge. The natural transformation $\theta_1$ collapse the edge added by $F$ while $\theta_2$ collapses everything but that added edge. We therefore have natural transformations linking the identity and a constant map which proves the claim.
\end{proof}

\begin{theorem}
For each open-closed cobordism $S$, we get a map of spectrum
\[B\Mod^\oc(S)\times M^{\p_\win S} \map \Thom(\kappa_S)\]
where $\kappa_S$ is the restriction of $\kappa$ to the connected component corresponding to $S$.
\end{theorem}
\begin{proof}
We will construct a map of spectra
\[\mu_{TP} :  \hocolim{\tilfatop}\Sigma^{\infty} M^{\p_\win -} \map \Thom(\kappa) =\hocolim{\tilfatop}\Thom(\kappa_\sigma^+,\kappa_\sigma^-)\]
directly. Morally, we are constructing a natural transformation  
\[\theta : \Sigma^{\infty} M^{\p_\win -} \Longrightarrow \Thom(\kappa_{-}^+,\kappa_{-}^-)\]
up to specified higher homotopies. 

Fix a simplex $\sigma=\left(\G_n\smap \cdots \smap \G_0\right)$.
By proposition \ref{prop:tub}, for any subsimplex 
\[\sigma_0=\left(\G_{i_k}\smap \cdots \smap \G_{i_0}\right)\]
 we get a map of spaces
\[\mu_{\sigma_0\subset\sigma} : \Tub(\sigma) \times |\cN\sigma_0| \times W_{\G_{i_k}}\times  M^{\p_\win\G_{i_0}}\map \Thom\left(\kappa_{\sigma_0}^+\right).\]
By desuspending by $W_{\G_{i_k}}$, we get a map of spectra
\begin{equation}\label{eq:map hocolim}
\lambda_{\sigma_0\subset \sigma}:\Tub(\sigma)\times \left|\cN\sigma_0\right| \times \Sigma^{\infty} M^{\p_\win \G_{i_0}} \map \Thom(\kappa_{\sigma_0}).
\end{equation}

Take an $m$-simplex
\[\sigma_0\subset \cdots \subset \sigma_m\]
of $\cN\Fat_1$. Above this $m$-simplex, we have the space
\[\Delta^{\sigma_0\subset\cdots \subset \sigma_m} \times \Tub(\sigma_m)\]
in the geometric realization of $\tilfatop$. The boundary of this space is identified with the spaces above the $\sigma_0\smap\cdots \widetilde{\sigma_i}\cdots \smap \sigma_m$. This $m$-simplex $\sigma_0\cdots \sigma_m$ also corresponds naturally to an $m$-simplex in $\cN\sigma_m$. Using this, we get
\[\theta_{\sigma_0\cdots \sigma_m}: \Delta^{\sigma_0\subset\cdots\subset\sigma_m}\times \Tub(\sigma_m)\times  \Sigma^{\infty} M^{\p_\win \G_{0}} \map \Delta^{\sigma_0\subset\cdots\subset\sigma_m} \times \Tub(\sigma_m)\times  \Thom(\kappa_{\sigma_m})\] 
where $\G_0$ is the first vertex of $\sigma_m$. This defines $\theta$ above $\sigma_0\subset \cdots \subset \sigma_m$.

We claim that all of  these maps together define a map $\mu_{TP}$ of homotopy colimits. For this to be the case, we need  that the operation above an $m$-simplex coincides with the operation for its subsimplices on the boundary of $\Delta^{\sigma_0\subset \cdots \subset \sigma_m}$.

Take any $m-1$ subsimplex
\[\sigma_0\subset \cdots \subset \widehat{\sigma_i}\subset\cdots \subset \sigma_m.\]
Lets first consider the case where $i<m$ since it is mostly trivial. In this case the identification along this boundary is simply the identity
\[
\p_i\left(\Delta^{\sigma_0\subset\cdots \subset \sigma_m}\right) \times  \Tub(\sigma_m) \map
\Delta^{\sigma_0\subset\cdots \widehat{\sigma_i}\cdots \subset \sigma_m}\times  \Tub(\sigma_m).
\]
Since the operations  we are comparing are built as
\[\xymatrix@R=10pt{
\Delta^{\sigma_0\subset\cdots \widehat{\sigma_i}\cdots \subset \sigma_m} \times \Tub(\sigma_m)\times \Sigma^{\infty }M^{\p_\win\G_0}\ar[r]&|\cN\sigma_m| \times \Tub(\sigma_m)\times \Sigma^{\infty }M^{\p_\win\G_0}\ar[r]^-{\lambda_{\sigma_m\subset\sigma_m}}&
\Thom(\kappa_{\sigma_m})\\
\Delta^{\sigma_0\subset\cdots \subset \sigma_m} \times \Tub(\sigma_m)\times \Sigma^{\infty }M^{\p_\win\G_0}\ar[r]&|\cN\sigma_m| \times \Tub(\sigma_m)\times \Sigma^{\infty }M^{\p_\win\G_0}\ar[r]^-{\lambda_{\sigma_m\subset\sigma_m}}&
\Thom(\kappa_{\sigma_m}).
}\]
we have a commutative diagram
\[\xymatrix{
\Delta^{\sigma_0\subset\cdots \widehat{\sigma_i}\cdots \subset \sigma_m} \times \Sigma^{\infty }M^{\p_\win\G_0}\ar[rr]^-{\theta_{\sigma_0\cdots \widehat{\sigma_i}\cdots \sigma_m}}\ar[d]&&
\Thom(\kappa_{\sigma_m})\ar[d]\\
\Delta^{\sigma_0\subset\cdots \subset\sigma_m} \times \Sigma^{\infty} M^{\p_\win\G_0}\ar[rr]^-{\theta_{\sigma_0\cdots\sigma_m}}&& \Thom(\kappa_{\sigma_m}).
}\]

For $i=m$, we need to work a bit harder. Say 
\[\sigma_{m-1}=\G_{i_k}\smap \cdots \smap \G_{i_0}.\]
In $\hocolim{\tilfatop}\Sigma^{\infty} M^{\p_\win -}$, we have
\[ \Delta^{\sigma_0\subset\cdots \subset \sigma_m}\times \Tub(\sigma_m)\times \Sigma^{\infty} M^{\p_\win \G_0}\]
above $\sigma_0\subset\cdots \subset\sigma_m$. While above $\sigma_0\subset\cdots \subset \sigma_{m-1}$, we have
\[ \Delta^{\sigma_0\subset\cdots\subset \sigma_{m-1}}\times \Tub(\sigma_{m-1}) \times \Sigma^{\infty} M^{\p_\win \G_{i_0}}.\]
These are glued using
\[\p_{m-1}\left(\Delta^{\sigma_0\subset\cdots \subset \sigma_m}\right)\times \Tub(\sigma_m)\times \Sigma^{\infty} M^{\p_\win \G_0} \map \Delta^{\sigma_0\subset\cdots\subset \sigma_{m-1}}\times \Tub(\sigma_{m-1}) \times \Sigma^{\infty} M^{\p_\win \G_{i_0}}.\]
We therefore need to show that the two operations agree along this identification. We will use the properties of \eqref{eq:map hocolim} as in proposition \ref{prop:tub} to show that 
\[\xymatrix{
\p_{m-1}\left(\Delta^{\sigma_0\subset\cdots \subset \sigma_m}\right)\times \Tub(\sigma_m)\times \Sigma^{\infty} M^{\p_\win \G_0} \ar[rr]^-{\lambda_{\sigma_0\cdots\sigma_m}|_{\p_{m-1}}}\ar[d]
&&\Thom(\kappa_{\sigma_m})\ar[d]\\
\Delta^{\sigma_0\subset\cdots\subset \sigma_{m-1}}\times \Tub(\sigma_{m-1}) \times \Sigma^{\infty} M^{\p_\win \G_{i_0}}\ar[rr]^-{\lambda_{\sigma_0\cdots\sigma_{m-1}}}&&
\Thom(\kappa_{\sigma_{m-1}})
}\]
commutes.

The map $\lambda_{\sigma_0\cdots\sigma_m}$ is defined after suspension by $W_{\G_{n}}$ as
\begin{equation}\label{eq:G0}
 |\cN\sigma(m)| \times \Tub(\sigma_m) \times M^{\p_\win\G_{0}}\times W_{\G_{n}} \map \Thom(\kappa^+_{\sigma_{m}}).
 \end{equation}
The map $\lambda_{\sigma_0\cdots \sigma_{m-1}}$ is defined after suspension by $W_{\G_{i_k}}$ as 
\begin{equation}\label{eq:Gi0}
|\cN\sigma_{m-1}|\times \Tub(\sigma_{m-1})\times M^{\p_\win\G_{i_0}}\times W_{\G_{i_k}} \map \Thom(\kappa^+_{\sigma_{m-1}}).
 \end{equation}
However we have an identification 
\[W_{\G_{i_k}} \times \frac{W_{\G_{n}}}{W_{\G_{i_k}}} \cong W^{eE_n\du eV_n}\]
as part of $\sigma_m$. In particular by suspending \eqref{eq:Gi0} by $\frac{W_{\G_{n}}}{W_{\G_{i_k}}}$, we get a map of spaces
 \[
 |\cN\sigma_{m-1}|\times \Tub(\sigma_{m-1})\times M^{\p_\win\G_{i_0}}\times W_{\G_{n}} \map \Thom\left(\kappa^+_{\sigma_{m-1}}\times  \frac{W_{\G_{n}}}{W_{\G_{i_k}}}  \right).
  \]
 which represents the same map of spectra as \eqref{eq:Gi0}. 
This new map is at the same suspension level as \eqref{eq:G0}. 

We know by property 3 of proposition \ref{prop:tub} on page \pageref{prop:tub} applied to $\sigma_{m-1}$ as a subsimplex of both $\sigma_{m-1}$ and $\sigma_m$, we get a commutative diagram
\[\xymatrix{
\Tub(\sigma_{m}) \times |\cN\sigma_{m-1}|\times  M^{\p_\win\G_{i_0}}\times W_{\G_{i_k}} \ar[d] \ar[rr]^-{\mu_{\sigma_{m-1}\subset\sigma_{m-1}}}&& \Thom(\kappa_{\G_{i_k}}|_{M^{\G_{i_0}}})\ar[d]\\
\Tub(\sigma_{m-1})\times |\cN\sigma_{m-1}| \times M^{\p_\win\G_{i_0}}\times W_{\G_{i_k}} \ar[rr]^-{\mu_{\sigma_m\subset\sigma_m}}&& \Thom(\kappa_{\G_{i_k}}|_{M^{\G_{i_0}}})
}\] 
which induce a commutative diagram
\begin{equation}\label{eq:prop3 trans}
\xymatrix{
\Tub(\sigma_{m}) \times |\cN\sigma_{m-1}|\times  M^{\p_\win\G_{0}}\times W_{\G_{i_k}} \ar[d] \ar[rrr]^-{\mu_{\sigma_{m-1}\subset\sigma_{m-1}}|_{M^{\G_0}}}&&& \Thom(\kappa^+_{\sigma_{m-1}}|_{M^{\G_{0}}})\ar[d]\\
\Tub(\sigma_{m-1})\times |\cN\sigma_{m-1}| \times M^{\p_\win\G_{0}}\times W_{\G_{i_k}} \ar[rrr]^-{\mu_{\sigma_m\subset\sigma_m}}&&& \Thom(\kappa^+_{\sigma_{m-1}}|_{M^{\G_{0}}}).
}\end{equation}
By property 2 of the same proposition applied to $\sigma_{m-1}\subset \sigma_m$ both as subsimplices of $\sigma_m$, we get that the diagram
\[\xymatrix{
\Tub(\sigma_m)\times |\cN\sigma_{m-1}|\times W_{\G_{i_k}}\times \frac{W_{\G_n}}{W_{\G_{i_k}}}\times M^{\p_{\win}\G_0} \ar[rr]^-{\mu_{\sigma_{m-1}\sigma_m}|_{M^{\G_0}}}\ar[d]&&\Thom(\kappa^+_{\sigma_{m-1}}|_{M^{\G_0}}\times \frac{W_{\G_{n}}}{W_{\G_{i_k}}})\ar[d]\\
\Tub(\sigma_m)\times |\cN\sigma_m|\times W_{\G_{n}} \times M^{\p_\win\G_0} \ar[rr]^-{\mu_{\sigma_{m}\sigma_m}}&&\Thom(\kappa_{\sigma_m})
}\]
commutes. In particular, we get that
\begin{equation}\label{eq:prop2 trans}
\xymatrix{
\Tub(\sigma_m)\times \p_{m}|\cN\sigma_{m}|\times W_{\G_{n}} \times M^{\p_\win\G_0} \ar[rr]^-{\mu_{\sigma_{m}\sigma_m}}\ar[d]^{\cong}&&\Thom(\kappa_{\sigma_m})\ar[d]^{\cong}\\
\Tub(\sigma_m)\times |\cN\sigma_{m-1}|\times 
 W_{\G_{i_k}}\times \frac{W_{\G_n}}{W_{\G_{i_k}}}\times M^{\p_{\win}\G_0} \ar[rr]^-{\mu_{\sigma_{m-1}\sigma_m}|_{M^{\G_0}}}&&\Thom(\kappa^+_{\sigma_{m-1}}|_{M^{\G_0}}\times \frac{W_{\G_{n}}}{W_{\G_{i_k}}})
 }\end{equation}

 Together \eqref{eq:prop3 trans} and \eqref{eq:prop2 trans} give that the diagram of spectra
\[\xymatrix{
\p_m\left(\Delta^{\sigma_0\subset \cdots \subset \sigma_m}\right) \times \Tub(\sigma_m) \times \Sigma^{\infty} M^{\p_\win\G_{0}}\ar[d] \ar[r]& \Thom(\kappa_{\sigma_{m-1}})\ar[d]\\
\Delta^{\sigma_0\subset\cdots\subset\sigma_{m-1}} \times \Tub(\sigma_{m-1})\times \Sigma^{\infty} M^{\p_\win\G_{0}} \ar[r]& \Thom(\kappa_{\sigma_{m-1}})
}\]
commutes as well which completes the proof. 
\end{proof}

 %================== Orientability!!!!!!!!!!!
\section{Thom isomorphisms and orientability}
\label{sec:orient}\label{sec:orientation}

Say $S$ is an open-closed cobordism.
In the previous section, we have defined a generalized Thom collapse
\[B\Mod(S)\times \Sigma^{\infty}M^{\p_\win S} \map \Thom(\kappa_S).\]
To complete the definition of the string operations, we need to consider Thom isomorphisms for the virtual bundle $\kappa_S$. 
Since the bundle $\kappa_S$ is not orientable, we will need to consider its determinant bundle. 

To get operations parameterized by some twisted homology of moduli space, we will construct a bundle $\chi_S$ over $B\Mod(S)$ and we will relate the determinant bundle of $\kappa_S$ with a tensor of the determinant bundle of $\chi_S$. This will give string operations parameterized by the homology of the moduli space with coefficients in $\det\chi^{\tens d}$ where $d$ is the dimension of the manifold. 

We will also show that the bundle $\chi_S$ is oriented whenever $S$ has at most one boundary component which is completely free. However, there is no way of picking these orientations compatibly with the composition of surfaces. In particular, we will show that our version of the Chas and Sullivan product is not strictly associative.

\subsection{Twisted moduli space and homological quantum field theory of degree $d$.}

For an open-closed cobordism $S$, consider the pair of vector spaces
\begin{equation}\label{eqn:chi vector space}
\left(H_1(S,\p_\win S ;\bR), H_0(S,\p_\win S;\bR)\right).
\end{equation}
We get a virtual bundle $\chi_S$ over $B\Mod(S)$ by setting the holonomy of this virtual bundle to be the usual action of 
\[ \pi_1(B\Mod_S) \cong \pi_0\Diff_{\oc} (S;\p_\win S \cup \p_\out S) \]
on the relative homology groups.

The determinant bundle of the virtual vector space $(V_1,V_2)$ is 
\[\det(V_1,V_2) \cong \left(V_1\right)^{\wedge \dim V_1} \tens \left(V_2^{\word{dual}}\right)^{\wedge \dim V_2}.\]
A morphism of virtual vector spaces
\[ (A,B)\map (C,D)\]
induces an isomorphisms
\[\det(A,B)\map \det(C,D).\]
Consider the bundle $\det(\chi_S)$ on $B\Mod(S)$.

\begin{lemma}
The bundles $\{\det(\chi_S)\}$ gives a bundle over the $B\Mod(S)$-ProP. 
\end{lemma}

More precisely, for any two two composable open-closed cobordism $S_1$ and $S_2$, we have an isomorphism
\[\det(S_1)\tens \det(S_2)\map \det(S_1\#S_2)\] 
which lifts the map
\[B\Mod(S_1)\times B\Mod(S_2)\map B\Mod(S_1\#S_2).\]

\begin{proof}
For any such $S_1$ and $S_2$, the long exact sequence 
of homology groups associated to the triple $(S_1\#S_2, S_1,\p_{\win }S_1)$ gives
\begin{equation}\label{eqn:LES}
\xymatrix{
0 \ar[r]& H_1(S_1;\p_{\win} S_1)\ar[r]& H_1(S_1\# S_2;\p_{\win} S_1)\ar[r]& H_1(S_1\#S_2; S_1) \ar`r[d]`d[l]`[llld]`d[lld][lld]&\\
& H_0(S_1;\p_{\win}S_1) \ar[r]&H_0(S_1\# S_2;\p_\win S_1) \ar[r]& H_0(S_1\# S_2; S_1)\ar[r]&0.
}\end{equation}
Using excision, we get
\[H_*(S_1\#S_2; S_1) \cong H_*(S_2;\p_\win S_2).\]
In particular, \eqref{eqn:LES} gives the require isomorphism.
\end{proof}

We have a symmetric monoidal category $H\word{Bord}^{d}$ enriched over graded abelian groups. The objects of $H\word{Bord}^{d}$ are isomorphism clasees of 1-manifolds with boundary. The morphisms in $H\word{Bord}^d$ between $[P]$ and $[Q]$ is
\[H\word{Bord}^d([P],[Q]) = \Bigoplus{[S]}\ H_*\left(B\Mod(S);(\det\chi_S)^{\tens d}\right).\]
Composition comes from the gluing of the preceding lemma.

\begin{definition}
A \emph{homological conformal field theory} of degree $d$ is a symmetric monoidal functor
\[F: H\word{Bord}^d\map \cG  \cG roups\]
from $H\word{Bord}^d$ to the category of graded abelian groups.
\end{definition}

\subsection{A combinatorial model for $\det(\chi)$.}
Over the fat graph model $\tilfatop$, we build a combinatorial model $\det\chi^{\aFat}$ for $\det(\chi)$ as follows. To any simplex 
\[\sigma=(\G_n\smap\cdots \smap \G_0),\]
we associate the pair of vector spaces
\[\left(\bR^{eH_n}, \bR^{eE_n\du eV_n}\right).\]
We think of this virtual vector spaces as the cellular chain complex
\[\bR^{eH_n}\map \bR^{eE_n\du eV_n}\]
computing the real homology of the pair $(\G_n,\p_\win\G_n)$ subdivided once.
To any morphism 
\[\sigma_0 =(\G_{i_k}\smap \cdots \smap \G_{i_0}) \subset \sigma,\]
we associate a morphism
\[\det\left(\bR^{eH_{i_k}}, \bR^{eE_{i_k}\du eV_{i_k}}\right) \map \det\left(\bR^{{eH_n}}, \bR^{{eE_n}\du {eV_n}}\right)\]
of virtual bundles constructed as follows.
The morphism $\varphi_{ni_k}$ induces a chain map
\begin{equation} \label{eq:twisting fat}
\xymatrix{
H_1(\G_n;\p_\win\G_n) \ar[r]\ar[d]^{\cong}_{\varphi_{ni_k}}& \bR^{eH_n} \ar[r]\ar[d]_{\varphi_{ni_k}}& \bR^{eE_n\du eV_n}\ar[d]_{\varphi_{ni_k}}\ar[r] & H_0(\G_n)\ar[d]^{\cong}_{\varphi_{ni_k}}\\
H_1(\G_{i_k},\p_\win\G_{i_k})\ar[r] &\bR^{eH_{i_k}} \ar[r] & \bR^{eE_{i_k}\du eV_{i_k}} \ar[r] & H_0(\G_{i_k}).
}\end{equation}
which gives a morphism of virtual bundles
\[ (\bR^{eH_{i_k}}, \bR^{eE_{i_k}\du eV_{i_k}}) \map (\bR^{eH_n},\bR^{eE_n\du eV_n}).\]
We therefore get an isomorphism
\[ \det(\bR^{eH_{i_k}}, \bR^{eE_{i_k}\du eV_{i_k}}) \map \det(\bR^{eH_n},\bR^{eE_n\du eV_n}).\]
We get a virtual bundle $\det\chi^{\aFat}$ over $\tilfatop$.

\begin{lemma}
The virtual bundle $\det\chi^{\aFat}$ forms a bundle over the partial-PROP $\aFat$. It models the virtual bundle $\det\chi$ over moduli space. 
\end{lemma}
\begin{proof}
Lets first show that the virtual bundle $\det\chi^{\aFat}$ is a combinatorial version of the virtual bundle $\det\chi$. For any fat graph $\G$, we consider again the cellular chain complex
\[\bR^{eH}\maplu{d} \bR^{eE\du eV}\]
for $H_*(\G,\p_\win \G)$ obtained after a simplicial subdivision of $\G$. This gives a contractible choice of morphisms
\begin{eqnarray*}
H_1(\G,\p_\win\G) \oplus \Image(d) &\maplu{\cong}& \bR^{eH}\\
H_0(\G,\p_\win\G)\oplus \Image(d)&\maplu{\cong}& \bR^{eE\du eV}.
\end{eqnarray*}
These morphisms respect the twisting since we have used
 \eqref{eq:twisting fat} to twist $\chi^{\aFat}$.
In particular these give an isomorphism
\[\det(\chi) \maplu{\cong} \det(\chi^{\aFat}).\]

Recall that $\til{\cG}\subset \tilfatop\times \tilfatop$ is the subcategory of simplices of glue-able fat graphs. We have a diagram
\[\xymatrix{
\chi^{\aFat}\times \chi^{\aFat}\ar[d]&\ar[l] \Psi^{*}(\chi^{\aFat}\times \chi^{\aFat}) \ar[d]\ar@{-->}[r]&\chi^{\aFat}\ar[d]\\
\tilfatop\times \tilfatop&\ar[l]_-{\simeq} \til{\cG} \ar[r]^{\#} &\tilfatop
}\] 
and we want to lift $\#$ to a morphism of vitual bundles
\[\Psi^{*}(\chi^{\aFat}\times \chi^{\aFat}) \map \chi^{\aFat}.\]
For any simplex $\sigma^{\cG}$
\[(\G_n,\tilG_n)\smap \cdots \smap (\G_{0},\tilG_0)\]
of $\til{\cG}$, we have
\[eE_{\G_n\#\tilG_n} \cong eE_{\G_n}\du eE_{\til\G_n}\]
and similar identities for both $eH$ and $eV$. In particular, we get an isomorphism
\[\left(\bR^{eH_n}, \bR^{eE_n\du eV_n}\right)\oplus \left(\bR^{\til{eH_n}}, \bR^{\til{eE_n}\du \til{eV_n}}\right)\map \left(\bR^{eH_{\G_n\#\tilG_n}} , \bR^{eE_{\G_n\#\tilG_n}\du eV_{\G_n\#\tilG_n}}\right)\]
which gives one on the determinant bundle 
\[\det(\chi_{S_1}^{\aFat}) \tens \det(\chi_{S_2}^{\aFat})\cong \det(\chi_{S_1\#S_2}^{\aFat}).\]
\end{proof}

\subsection{Relation between $\chi$ and $\kappa$}

Say we have fixed an orientation on $M$.
The goal of this section is to relate $\det(\kappa)$ and $\det(\chi)^{\tens d}$.
Recall that if
\[\sigma= \G_n\smap \cdots \smap \G_0\]
then
\[\kappa_\sigma^+ = \left.TM^{eH_n} \times \nu^{eE_n\du eV_n}\right|_{M^{\G_0}} \qquad \kappa_\sigma^- = M^{\G_0}\times  W_{\G_n}\]
where $W_{\G_n}\cong W^{eE_n\du eV_n}$. Here $\nu$ is the normal bundle of the fixed embedding $f:M\smap W$. The bundle above a morphism $\tilsigma\subset\sigma$ is determined according to the maps of diagram \ref{dia:sigma0} on page \pageref{dia:sigma0}.

Using the idea that $\nu$ is really $W$ minus $TM$, we can define a second bundle $\det\tilde\kappa$ above 
\[\hocolim{\tilfatop} \Sigma^{\infty} M^{-}\]
which has
\begin{eqnarray*}
\tilde\kappa_{\sigma}^+ &=& \left.TM^{eH_n}\right|_{M^{\G_0}}\\
\tilde\kappa_{\sigma}^- &=& \left.TM^{eE_n\du eV_n}\right|_{M^{\G_0}}.
\end{eqnarray*}
Pick a morphism
\[\tilsigma = (\G_{i_k}\smap\cdots \smap \G_{i_0})\subset \sigma.\]
We have
\[TM^{eH_{i_k}}\map TM^{eH_n}\]
which hits  the half-edges of $eH_n$ that are not collapsed by $\varphi_{ni_k}$. We also have
\[TM^{eV_{i_k}\du eE_{i_k}} \map TM^{eV_n\du eE_{n}}.\]
As we have argued before the cokernel of these maps are identified by the pullback diagram
\[\xymatrix{
M^{V_{i_k}}\ar[r]\ar[d]& M^{V_{i_k}^\win}\times M^{eH_{i_k}}\times M^{eV_{i_k}\du eE_{i_k}}\ar[d]\\
M^{V_{n}}\ar[r]& M^{V_n^\win}\times M^{eH_n}.
}\]
This determines an isomorphism
\[\det\tilde\kappa_{\tilsigma} \cong \det\tilde\kappa_{\sigma}\]
and hence we have a bundle $\det\tilde\kappa$.

\begin{lemma}
There is a natural isomorphism
\[\det(\tilde\kappa)\cong \det(\kappa).\]
\end{lemma}
\begin{proof}
We  have
\[\tilde\kappa_\sigma^+ \oplus \left(\nu^{eE_n\du eV_n}\right)\cong  \kappa_{\sigma}^+.\]
Now by definition of $\nu$, we have a short exact sequence
\[TM\map W\times M \map \nu\]
and in particular, we have a contractible choice of splittings
\[\nu\oplus TM \cong W.\]
This means that we have a contractible choice of identification
\[\tilde\kappa_{\sigma}^+\oplus \left(\nu^{eE_n\du eV_n}\right) \cong \tilde\kappa_{\sigma}^+.\]
We therefore have an isomorphism
\[\det(\tilde\kappa_\sigma)\cong \det(\kappa_\tilsigma).\]
Our choice of gluing in $\tilde\kappa_\sigma$ was made so that the isomorphism respects the twisting above morphisms.
\end{proof}

We can define a similar bundle above $\tilfatop$. 
For any simplex $\sigma$, we let
\[\lambda_{\sigma}=\det\left(\bR^{eH_n},\bR^{eV_n\du eH_n}\right).\]
As in $\tilde\kappa$, for a morphism of simplicies $\tilsigma\subset\sigma$, we consider
\begin{eqnarray*}
\bR^{eH_{i_k}}&\map& \bR^{eH_n}\\
\bR^{eE_{i_k}\du eV_{i_k}}&\map&\bR^{eE_n\du eV_m}
\end{eqnarray*}
induced by the morphism $\varphi_{ni_n}$ of graphs. 

\begin{lemma}
The forgetful functor
\[\Psi: \hocolim{\tilfatop} \Sigma^{\infty} M^{-} \map \tilfatop\]
pulls back 
\[\Psi^*(\lambda^{\tens d}) = \det\tilde\kappa\]
\end{lemma}
\begin{proof}
Since $TM$ is oriented of dimension $d$, the statement is clear.
\end{proof}

\begin{lemma}\label{lem:orientable}
There is an isomorphism
\[\det(\lambda) \cong \det(\chi)\]
\end{lemma}
\begin{proof}
It suffices to compare $\lambda$ and $\det(\chi^{\aFat})$. Over each $\sigma$,
\[\lambda_{\sigma}\cong \det(\chi^{\aFat})\]
and hence it suffices to compare what happens over the morphisms. Fix an inclusion 
\[\tilde{\sigma}=\left(\G_{i_k}\smap \cdots\smap \G_{i_0}\right)
\subset 
\sigma=\left(\G_n\smap \cdots \smap \G_0\right). \] 
The map $\varphi_{i_00}$ is a homotopy equivalence relative to the incoming circles and hence we get a chain map
\[\xymatrix{
H_1(\G_n;\p_\win\G_n)\ar[r]\ar[d]^{\cong}&\bR^{eH_n} \ar[d]^{f_{eH}}\ar[r]&\bR^{eE_n\du eV_n}\ar[r]\ar[d]^{f_{eE\du eV}} &H_0(\G_{n};\p_\win\G_n)\ar[d]^{\cong}\\
H_1(\G_{i_k};\p_\win\G_{i_k})\ar[r]&\bR^{eH_{i_k}} \ar[r] & \bR^{eE_{i_k}\du eV_{i_k}} \ar[r]& H_0(\G_{i_k};\p_\win\G_{i_k})
}\]
which we have used in \eqref{eq:twisting fat} to construct $\chi^{\Fat}$.
Here the map
\[\bR^{eH_{n}}\map \bR^{eH_{i_k}}\]
sends an element $1_h$ to either $1_{\varphi_{ni_k}(h)}$ if $h$ is not  collapsed by $\varphi_{ni_k}$ or to zero if it is. The second map
\[\bR^{eE_{n}\du eV_{n}}\map \bR^{eE_{i_k}\du eV_{i_k}}\]
sends any $1_a$ to $1_{\varphi_{ni_k}(a)}$. We get the isomorphism of pairs of vector spaces
\[\det\left(\bR^{eH_n}, \bR^{eE_n\du eV_n}\right) \cong\det \left(\bR^{eH_{i_k}}, \bR^{eE_{i_k}\du eV_{i_k}}\right).
\]

When constructed the bundle $\det\tilde\kappa$, we considered
\begin{eqnarray*}
g_{eH}:\bR^{eH_{i_k}}&\maps &\bR^{eH_n}\\
g_{eE\du eV}:\bR^{eE_{i_k}\du eV_{i_k}}&\maps &\bR^{eE_n\du eV_n}.
\end{eqnarray*}
Here 
\begin{eqnarray*}
g_{eH}(1_h)&=& 1_{\varphi^{-1}(h)}\\
g_{eE\du eV}(1_e)&=& 1_{\varphi^{-1}(e)}\\
g_{eE\du eV}(1_v)&=& \sum_{a \in \varphi^{-1}(v) \cap\left(eE_{i_0}\du eV_{i_0}\right)} 1_a.
\end{eqnarray*}
Note that $g_{eH}$ is a section of $f_{eH}$. We also have
\begin{eqnarray*}
f_{eH} \cdot g_{eH}(1_e)&=& 1_e\\
f_{eH} \cdot g_{eH}(1_v)&=& \frac{1_v}{\#\varphi^{-1}(v) \cap\left(eE_{i_0}\du eV_{i_0}\right)} 
\end{eqnarray*}
which is a section up to multiplication by a \emph{positive} constant. 

In particular the $f$'s and the $g$'s induce the same map 
\[\det(\bR^{eH_{i_k}},\bR^{eE_{i_k}\du eV_{i_k}})\cong \det(\bR^{eH_{n}},\bR^{eE_n\du eV_n}).\]
\end{proof}

\subsection{The string operations}

In section \ref{sec:TP}, we have constructed a generalized Thom collapse
\begin{equation}\label{eq:TP collapse}
B\Mod(S;\p S) \times M^{\p_\win S} \map \Thom(\kappa_S).
\end{equation}
We will use the Thom isomorphism 
\[H_*\left(\Thom(\kappa_S);\det(\kappa_S)\right) \maplu{\cong} H_*\left(\hocolim{\tilfatop} M^{-} \right) \]
for the non-oriented bundle $\kappa_S$ and our knowledge of $\det\kappa$ to define our string operations.

\begin{lemma}
The homotopy colimit of $M^{\p_\out -}$ is homotopy equivalent to the product
\[\hocolim{\tilfatop} M^{\p_\out -} \simeq \DU{S}\ \left(B\Mod(S) \times M^{\p_\out S}\right).\]
\end{lemma}
\begin{proof}
We have already proven this for the functor $M^{\p_\win -}$ in lemma \ref{lem: pwin}. The proof is exactly the same here.
\end{proof}

This means that the restriction maps
\[M^{\G}\map M^{\p_\out \G}\]
give a map
\[H_*\left(\hocolim{\tilfatop} M^{-} \right)\map
H_*\left(\hocolim{\tilfatop} M^{\p_\out -} \right) \map\bigoplus_{S} H_*\left(M^{\p_\out S}\right) \]

We will now translate the result of the previous section to get a twisted Thom collapse.

\begin{lemma}
The generalized Thom collapse of \eqref{eq:TP collapse} gives a map
\[H_*\left(B\Mod(S);\det\chi_S^{\tens d}\right) \tens H_*M^{\p_{\win}S} \map H_*(\Thom(\kappa), \det(\kappa)).\]
\end{lemma}
\begin{proof}
By the preceding section, the bundle $\det(\kappa)$
pulls back to $\det\chi_S^{\tens d}$ on 
\[B\Mod(S)\times \Sigma^{\infty} M^{\p_{\win}} S.\]
This proves the statement.
\end{proof}

\begin{definition}
The string operation associated to $S$ is the composition 
\[\begin{split}
H_*(B\Mod(S);\det\chi_S^{\tens d}) \tens H_*M^{\p_{\win}S}&\cong H_*(\widetilde{B\Mod(S)};\det\chi_S^{\tens d}) \tens H_*M^{\p_{\win}S}\\ 
&\map H_*(\Thom(\kappa_S);\det(\kappa)) \map H_*(M^{\p_\out S}).
\end{split}\]
\end{definition}

Note that if the manifold is even dimensional then the twisting above $B\Mod(S)$ is trivial and we get operations parameterized by the homology $H_*(B\Mod(S))$.

\subsection{Triviality over certain components}

Let $S$ be an open-closed cobordism with at most one boundary component that is completely free. We claim that the bundle $\chi_S$ is orientable .

\begin{lemma}
The bundle $\chi_S$ over $B\Mod(S)$ is oriented.
\end{lemma}
\begin{proof}
The action of the diffeomorphism group on $H_0(S;\p_\win S)$ is the identity since each connected component of $S$ must have a boundary boundary component which is not incoming and hence all but one must have a bit of outgoing boundary. Since these are preserved by $\Mod(S)$, all but one connected component must be  preserved and hence they must all be preserved.

Consider the long exact sequence
\[\xymatrix{
H_1(\p_\win S) \ar[r]& H_1(S)\ar[r] & H_1(S;\p_\win S)\ar[r]& H_0(\p_\win S) \ar[r]& H_0(S)\ar[r] &H_0(S;\p_\win S)
}\]
The action of the diffeomorphism group is trivial on all terms but 
\[H_1(S)\qquad H_1(S;\p_\win S).\]
If $S$ is a surface of genus $g$ with $n$ boundary components,
\[H_1(S) \cong \bZ^{\oplus 2g} \oplus \bZ^{\oplus n}.\]
The mapping class group acts as the identity on $\bZ^{\oplus n}$ since our diffeomorphisms must preserve all boundary components but one (and hence preserve all). The action on $\bZ^{\oplus 2g}$ is not trivial, however, since diffeomorphisms preserve the intersection pairing of the surface, the action factors through the symplectic group. It therefore is orientation-preserving.
In particular the action on $H_1(S;\p_{\win} S)$ is also orientation preserving.
\end{proof}

In particular, we get string operations parameterized by $H_*(B\Mod(S))$. More precisely, once we have picked an orientation for $\chi_S$, we get an operation
\[\mu_S : H_*B\Mod(S) \tens H_*(M^{\p_{\win} S}) \map H_*(M^{\p_{\out}S}).\]
which increases degree by the relative Euler characteristic
\[\dim(H_0(S;\p S)) - \dim(H_1(S;\p S)\]
times $d$. This degree shift was hidden in the grading of $\chi_S$.

\begin{remark}
Say we are interested in the mapping class group
\[\Mod^\prime(S) = \pi_0\Diff(S;\p_{\win} S\du \p_\out S;\pi_0 \p_\free S)\]
which preserves $\p_\win$ and $\p_\out$ pointwise and fixes each
connected component of $\p_\free$ as a set. The proof of the previous theorem shows that the bundle $\chi_S$ is orientable over
\[B\Mod^{\prime}(S) \subset B\Mod(S)\]
and hence we get operations parameterized by $H_*(B\Mod^{\prime}(S)).$
\end{remark}

\subsection{Our Chas-Sullivan product is skew associative when $d$ is odd.}
\label{sub:ChasSullivan}

This section shows that although $\kappa$ is oriented over certain connected components, it is not possible to pick  orientations which are compatible with the gluing. 

Lets consider the pair of pants $S$. We have
\[(H_1(S;\p_{\win} S), H_0(S;\p_\win S) \cong \left(\bZ,0\right).\]
Say we pick a generator for $H_1(S;\p_\win S)$ which goes from the first boundary to the second. This gives a trivialization of
$\det\chi^{\tens d}$
and hence an operation
\[H_*(B\Mod(S))\tens H_*LM^{\tens 2}\map H_*LM.\]
We will call 
\[\star_{CS}: H_*LM^{\tens 2}\map H_*LM\]
our Chas and Sullivan product.

\begin{figure} \begin{center}
\mbox{\epsfig{file=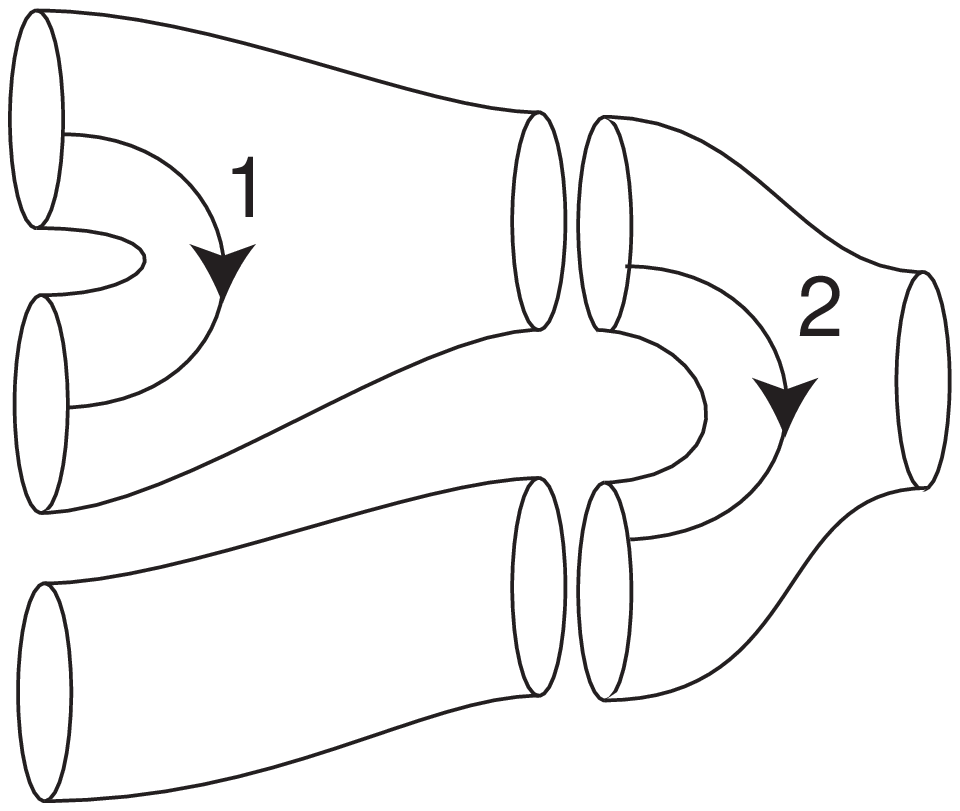, width=150pt}}
\qquad
\mbox{\epsfig{file=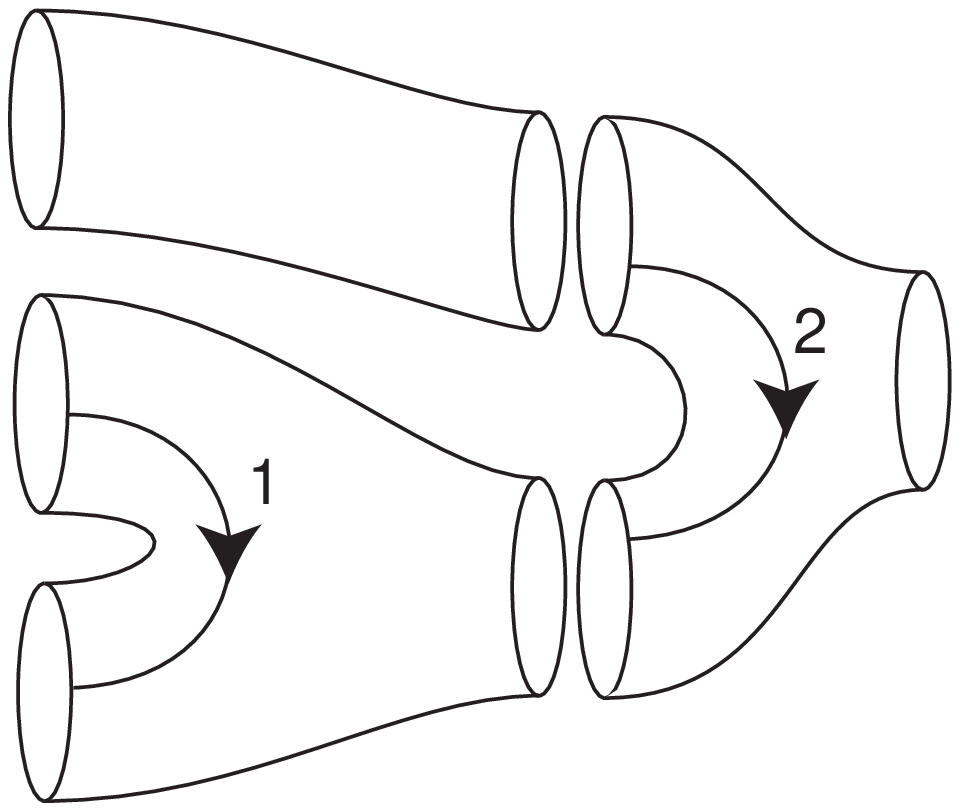, width=150pt}}
\caption{Two compositions}\label{fig:comp}
\end{center}
\end{figure}

\begin{proposition}
The product $\star_{CS}$ is  associative if and only if $M$ is even dimensional. If $M$ is odd-dimensional, $\star_{CS}$ is skew associative.
\end{proposition}
\begin{proof}
We will assume that our operations glue as does $\chi$ (which will be proven in the next section). We therefore want to compare the two isomorphisms
\[\xymatrix{
\det(\chi_S) \times \det(\chi_S)\ar@/^15pt/[rr]\ar@/_15pt/[rr]&& \det(\chi_{S_3})}\] 
given by the composition illustrated in figure \ref{fig:comp}. Here $S_3$ is the pair of pants with three legs. 

We have that
\[(H_1(S_3;\p_\win S_3), H_0(S_3;\p_\win S_3)) \cong(\bZ\oplus \bZ; 0)\]
where we can choose the generators of $H_1$ to be $\gamma_{12}$ a path from the first to the second boundary and $\gamma_{23}$ a path from the second to the third boundary. 
The first composition will send our chosen generators to
\[\gamma_{12}\tens \gamma_{23}\]
while the second composition will send then to
\[\gamma_{23}\tens \gamma_{12}.\] 
In particular, the two composition of $\star_{CS}$ will agree up to a sign of $(-1)^d$
\end{proof}

The original Chas and Sullivan product as defined in \cite{ChasSullivan} was associative. However, Chas and Sullivan shifted  the homology by $d$. This change in grading hides the signs that we are seeing. Note that this change of grading would make the coproduct skew coassociative from odd-dimensinal manifolds and hence would not  rid us of our bundles.

This sign problem is the same as the sign issue that appears in the definition of the intersection pairing in $H_*M$. If we use the Thom-Pontrjagin construction to define the pairing, it not be associative. However, one could use the cup product directly and get an associative product. But in this last case, we would be thinking of a homology class in its cohomological degree.

 %--------------------------------------------------------- Gluing operations
\section{Gluing the operations}

\label{sec:gluing operations}
We are now ready to show that the operations that we have constructed glue  to give a degree $d$-TCFT.

\subsection{Thickening the category $\cG$ of glue-able fat graphs.}

In section \ref{sub:gluing graphs},  we have constructed a category $\cG$ of glue-able admissible fat graphs with functors
\[\aFat \times \aFat \imap \cG \maplu{\#} \aFat.\]
The first functor realizes to a homotopy equivalence on the geometric realization. The functore $\#$ glues the glue-able fat graphs.

Since we have thickened the category $\aFat$ to construct the operations, we will now thicken $\cG$ to match. We will construct two categories $\cG_1$ and $\tilde{\cG}$ as well as a diagram
\[\xymatrix{
\tilfatop\times \tilfatop \ar[d] &\ar[l]^-{\sim} \tilde{\cG}\ar[d] \ar[r]^-{\tilde\#}&  \tilfatop\ar[d]\\
\left( \cN\aFat_1\right)^{op} \times\left( \cN\aFat_1\right)^{op}\ar[d]&\ar[l]^-{\sim} \cN\cG_1^{op}\ar[d]\ar[r]^-{\cN\#_1}&\left( \cN\aFat_1\right)^{op}\ar[d]\\
\aFat_1\times \aFat_1 \ar[d]&\ar[l]^-{\sim} \cG_1 \ar[r]^-{\#_1}\ar[d] & \aFat_1\ar[d]\\
\aFat \times \aFat &\ar[l]^-{\sim} \cG \ar[r]^-{\#}& \aFat
}\]
where the $\sim$ means that the functor induces a homotopy equivalence on the geometric realizations.
First, let $\cG_1$ be the pullback of the diagram
\[\xymatrix{
\cG_1\ar[d]\ar[r]&\aFat_1\times \aFat_1\ar[d]\\
\cG\ar[r]&\aFat\times \aFat.
}\]
The objects of $\cG_1$ are pairs $(\G,\tilG)$ of glueable fat graphs and the morphisms of $\cG_1$ are morphisms $(\varphi,\til\varphi)$ of $\cG$ with a choice of splittings for both $\varphi$ and $\til\varphi$.

\begin{lemma}
The functor
\[\cG_1\map \aFat_1\times \aFat_1\]
realizes to a geometric realization.
There is also  a functor
\[\#_1 : \cG_1\map \aFat_1\]
which lifts $\#$.
\end{lemma}
\begin{proof}
The argument of proposition \ref{lem:fat1 vs fat} that showed that 
\[|\aFat_1|\map |\aFat|\]
is a homotopy equivalence gives that
\[|\cG_1|\map |\cG|\]
is also one. Since  we have a commutative diagram
\[\xymatrix{
|\cG_1| \ar[r]\ar[d]^{\sim} &| \aFat_1 \times \aFat_1|\ar[d]^{\sim}\\
|\cG|\ar[r]^-{\sim}&|\aFat\times \aFat|
}\]
we get a homotopy equivalence on the top arrow.

Lets now construct the gluing $\#_1$. Fix an object $(\G_1,\G_2)$ in the category $\cG$ of glue-able admissible fat graphs. We let $\G=\G_1\#\G_2$ denote the fat graph obtained by gluing $\G_1$ and $\G_2$. Note that the 
\[eE_\G = eE_{\G_1}\du eE_{\G_2} \qquad eV_\G= eV_{\G_1}\du eV_{\G_2}\qquad eH_\G= eH_{\G_1}\du eH_{\G_2}\]
For any morphism 
\[(\varphi,\til\varphi):(\G_0,\tilG_0)\map (\G_1,\tilG_1)\]
of $\cG$, we construct a continuous map
\[\Split(\varphi)\times \Split(\til\varphi)\map \Split(\varphi\#\til\varphi).\]

Say we have splittings $(\alpha,\beta)$ and $(\til\alpha,\til\beta)$.
We then get a splitting for $\varphi\#\til\varphi$ by considering
\[\xymatrix{
\bR^{eE_{\G_0\#\tilG_0}\du eV_{\G_0\#\tilG_0}} \cong 
\bR^{eE_{\G_0}\du eV_{\G_0}} \times \bR^{eE_{\tilG_0}\du eV_{\tilG_0}} \ar[d]^{\alpha\times\til\alpha}\\
 \bR^{eE_{\G_1}\du eV_{\G_1}}\times \bR^{eE_{\tilG_1}\du eV_{\tilG_1}} \cong
\bR^{eE_{\G_1\#\tilG_1}\du eV_{\G_1\#\tilG_1}} 
}\]
By construction, for an extra vertex $v$ of $\G_1\#\tilG_1$ this splittings gives, as required, a splitting for 
\[
\bR\map \bR^{eE^{\G_0\#\tilG_0}_v\du eV^{\G_0\#\tilG_0}_v}\map
\frac{\bR^{eE^{\G_0\#\tilG_0}_v\du eV^{\G_0\#\tilG_0}_v}}{\bR}
\]
\end{proof}

On the categories of simplices, we get naturally
\begin{equation}\label{eqn:gluing cn}
\cN\left(\aFat_1\times\aFat_1\right) \imap \cN\cG_1\maplu{\cN\#_1} \cN\aFat_1.
\end{equation}
We can compose the first functor with the functor
\[\cN\left(\aFat_1\times\aFat_1\right)  \map \cN\aFat_1\times\cN\aFat_1\]
which sends
\[\left[(\G_n,\tilG_n)\smap \cdots \smap (\G_0,\tilG_0)\right] \mapsto
\left(\left[\G_n\smap \cdots \smap \G_0\right], \left[\G_n\smap \cdots \smap \G_0\right]\right).\]
These also give functors
\[ \left(\cN\aFat_1\right)^{op}\times\left(\cN\aFat_1\right)^{op} \imap \cN\cG_1^{op}\maplu{\cN\#^{op}} \left(\cN\aFat_1\right)^{op}.\]

\begin{lemma}
The functor 
\[\cN\cG_1 \map \cN\left(\aFat_1\times\aFat_1\right)\]
realizes to a homotopy equivalence.
\end{lemma}
\begin{proof}
This is a direct consequence of the fact that the functors
\[\cG_1 \map \aFat_1\times\aFat_1\]
and 
\[\cN\cC\map \cC\qquad \forall\ \cC\]
gives a homotopy equivalences.
\end{proof}

To complete the extensions, we now need to define a category $\til{\cG}$ and functors
\[\tilfatop\times \tilfatop \imap \til{\cG}\map \tilfatop.\]
Recall that we have defined $\tilfatop$ by first defining a functor 
\[\xymatrix{
\Tub: \left(\cN\aFat_1\right)^{op} \maplu{\cT} \Top^? \ar[rr]^-{\holim{}}&& \Top
}\]
and then using Thomason's construction
\[\tilfatop = \left(\cN\aFat_1\right)^{op} \int \Tub.\]
Consider the functor
\[\xymatrix{
(\Tub,\Tub) : \cN\cG_1\ar[r]& \cN\Fat_1\times \cN\Fat_1 \ar[rr]^-{\Tub\times \Tub}&& \Top
}\]
and define
\[\til{\cG} = \cN\cG_1^{op} \int (\Tub,\Tub).\]

\begin{lemma}
There is a functor
\[\Psi : \til{\cG}\map \tilfatop\times \tilfatop\]
which is a homotopy equivalence on the geometric realization.
\end{lemma}
\begin{proof}
The functor takes an object 
\[\left((\G,\tilG), (x,\til{x})\right)\]
of $\til{\cG}$ and sends it to the object
\[\left((\G,x),(\tilG,\til{x})\right).\]
The morphisms follow a similar pattern. 

We have a functor
\[F : \cN\cG_1 \map \cN\Fat_1\times \cN\Fat_1\]
which realizes to a homotopy equivalence. Hence
\[\hocolim{\cN\cG_1} (\Tub,\Tub) = \hocolim{\cN\cG_1} \left(\Tub\times \Tub\right)\cdot F \maplu{\simeq} \hocolim{\cN\Fat_1\times \cN\Fat_1} \Tub\times \Tub = \hocolim{\cN\Fat_1} \Tub\times \hocolim{\cN\Fat_1} \Tub\]
\end{proof}

We now need to construct a functor
\[\til{\cG}\maplu{\#} \tilfatop\]
which respects the operations. We will first construct a natural transformation 
\[\zeta:(\Tub,\Tub)\Longrightarrow \Tub \circ\#\]
as in
\[\xymatrix{
\til\cG \ar@/_20pt/[rr]^{\Tub\circ \#}="T" \ar[rd]_{\#}\ar@/^20pt/[rr]_{(\Tub,\Tub)}="TT" &&\Top\\
& \tilfatop \times\tilfatop \ar[ru]_{\Tub}\ar@{=>}"TT";"T"_{\zeta}.
}\]

\subsection{How to glue tubular neighborhoods.}

Before building this $\zeta$, we first start by looking at how one might glue tubular neighborhoods.

\begin{lemma}\label{lem:constructing composition}\label{lem:composition}
Say we have embeddings
\[\xymatrix{
A\times W\ar[r]\ar[d]& A\times X\ar[d] & B\times Y\ar[r]\ar[d]&B\times Z\ar[d]\\
A\ar[r]^{Id}&A& B\ar[r]^{Id}&B
}\]
Say we have tubular neighborhoods
\[\xymatrix{
\left.\frac{A\times TX}{A\times TW}\right|_{A\times W} \ar[r]^{\phi_A}\ar[d]& A\times X\ar[d]& \left.\frac{B\times TZ}{B\times TY}\right|_{B\times Y} \ar[r]^{\phi_B}\ar[d]& B\times Z\ar[d]\\
A\ar[r]^{Id}& A&B\ar[r]^{Id}&B
}\]
which also lie above the identity. Say finally we have a map
\[A\times W\map B.\]
We get a tubular neighborhood for
\[A\times W\times Y\map A\times X\times Z\]
which lie above the identity by taking the composition
\[\xymatrix{
\left.\frac{A\times TX\times TZ}{A\times TW\times TY}\right|_{A\times W\times Y} \ar[r]\ar@/^20pt/[rrrr]^{\psi}&
\left.\frac{A\times TX}{A\times TW}\right|_{A\times W} \times \left.\frac{B\times TZ}{B\times TY}\right|_{B\times Y} \ar[rr]_-{\psi_A\times \psi_B}&& A\times X\times B\times Z \ar[r]& A\times X\times Z
}\]
where the last map is the projection away from $B$.
\end{lemma}
\begin{proof}
The map $\psi$ is automatically smooth. It therefore suffices to prove that it is injective and that the derivative gives a splitting near the zero section.

Lets first prove injectivity. Take two points $[a_i,\theta_i,\kappa_i]$ (i=1,2) in 
\[\frac{A\times TX\times TZ}{A\times TW\times TY}\]
and assume that 
\begin{equation}\label{eq:psi}
\psi([a_1,\theta_1,\kappa_1]) = \psi([a_2,\theta_2,\kappa_2]).
\end{equation}
The map
\[\frac{A\times TX\times TZ}{A\times TW\times TY}\maplu{\psi} A\times X\times Z \map A\times X\]
sends $[a,\theta,\kappa]$ to $\varphi_A([a,\theta])$ and since $\varphi_A$ is injective, we get that
\[a_1=a_2\qquad \theta_1=\theta_2.\]
In particular the $\theta_i's$ lie in the same fiber  $T_wX$ of $TX$.  
Let
\[b= \pi(a_1,  w) = \pi(a_2,w).\]
Because of \eqref{eq:psi} and because the diagram
\[\xymatrix{
\frac{A\times TX\times TZ}{A\times TW\times TY}\ar[r]\ar[rd]_{f}\ar@/^20pt/[rrr]&
\frac{A\times TX}{A\times TW}\times \frac{B\times TZ}{B\times TY}\ar[r]\ar[d]&A\times X\times B\times Z\ar[d]\ar[r]& A\times X\times Z\\
&\frac{B\times TZ}{B\times TY}\ar[r]^{\varphi_B}&B\times Z
}\]
is commutative, we get that
\[\varphi_{B}\circ f([a_1,\theta_1,\kappa_1]) =\varphi_{B}\circ f([a_2,\theta_2,\kappa_2]).\]
Because $\varphi_B$ is injective, we get that
\[(b,\kappa_1) =f([a_1,\theta_1,\kappa_1]) = f([a_2,\theta_2,\kappa_2])=(b,\kappa_2)\]
and hence $\kappa_1=\kappa_2$. Hence $\psi$ is injective.

To show that $\psi$ is a tubular neighborhood, it suffices to show that the following composition is the identity.
\[\xymatrix{
\left.\frac{A\times TX\times TZ}{A\times TW\times TY}\right|_{A\times W\times Y}\ar[r]&
 \left.T\left(\left.\frac{A\times TX\times TZ}{A\times TW\times TY}\right|_{A\times W\times Y}\right)\right|_{0-\word{sect}}\ar[r]^-{d\psi}&\left.T(A\times X\times Z)\right|_{A\times W\times Y}
\ar[r]&
\left.\frac{A\times TX\times TZ}{A\times TW\times TY}\right|_{A\times W\times Y}
}\] 
Here the first map includes the bundle as the vertical tangent space at the zero section.
Pick an element $(a,w,y)$ of $A\times W\times Y$. Let $b=\pi(a,w)$. If we restrict to the fibre above $(a,w,y)$, we have a commutative diagram
\[\xymatrix{
\frac{T_w X}{T_wW}\times \frac{T_y Z}{T_y Y}
\ar[d]\ar[rd]\\
T_{a,x,y}(A\times X\times Z) \ar[r]\ar[d]& T_{a,w}(A\times X)\times T_{b,y}(B\times Z)\ar[dl]\\
T_{a,w}\left(\frac{A\times TZ}{A\times TW}\right) \times T_{b,y}\left(\frac{B\times TZ}{B\times TY}\right).
}\] 
Since the right hand side comes from the product of two tubular neighborhoods it gives the identity. And hence so does the left-hand side.
\end{proof}

\begin{lemma}\label{lem:composition lift}
Assume that we are in the same situation as in the previous lemma. 
Now assume that we also have pullback diagrams
\[\xymatrix{
A_0\times W_0 \ar[r]\ar[d]&A_0\times X_0\ar[d]&&B_0\times Y_0\ar[r]\ar[d]&B_0\times Z_0\ar[d]\\
A\times W\ar[r]& A\times X && B\times Y\ar[r] &B\times Z
}\]
and assume that both $\phi_A$ and $\phi_B$ induce tubular neighborhood for the top maps. Finally assume that the map
\[A\times W \map B\]
restricts to a map $A_0\times W_0\map B_0$. 
Then $\psi(\phi_A,\phi_B)$ restrict to a tubular neighborhood of 
\[A_0\times W_0\times Y_0 \map A_0\times X_0\times Z_0.\]
In fact the new tubular neighborhood is 
\[\psi(\phi_A|_{A_0\times W_0}, \phi_B|_{B_0\times Y_0})\]
\end{lemma}
\begin{proof}
We then have a diagram
\[\xymatrix{
\frac{A_0\times TX_0\times TZ_0 }{A_0\times TW_0\times TZ_0} \ar[d]\ar[r]& 
\frac{A_0\times TX_0}{A_0\times TW_0} \times \frac{B_0\times TZ_0}{B_0\times TY_0} \ar[rr]^{\phi_A|_{A_0}\times \phi_B|_{B_0}}\ar[d]&&
A_0\times X_0\times B_0\times Z_0\ar[r]\ar[d]&
A_0\times X_0\times Z_0\ar[d]\\
\frac{A\times TX\times TZ }{A\times TW\times TZ} \ar[r]& 
\frac{A\times TX}{A\times TW} \times \frac{B\times TZ}{B\times TY} \ar[rr]^{\phi_A\times \phi_B}&&A\times X\times B\times Z\ar[r]&
A\times X\times Z
}\]
and the construction of a tubular neighborhood for 
\[A\times W\times Y\map A\times X\times Z\]
restricts to one on 
\[A_0\times W_0\times Y_0\map A_0\times X_0\times Z_0.\]
Note also that lifting through this type of setting commutes with the construction. (Lifting $\phi_A$ and $\phi_B$ first or lifting the $\psi$ give the sam tubular neighborhood.)

\end{proof}

\begin{lemma}\label{lem:composition twist}
Say we have 
\[\xymatrix{
&\nu_A\ar[d]&&&\nu_B\ar[d]\\
A\times U\ar[r]\ar[ru]& A\times U\times W&& B\times V\ar[r]\ar[ru]&B\times V\times Y 
}\]
where $\nu_A$ and $\nu_B$ are bundles. Say we have picked tubular neighborhoods
\begin{eqnarray*}
\psi_A:\frac{(T\nu_A)|_{A\times U}}{TA\times TU} \cong \nu_A|_{A\times U}\oplus \frac{(TA\times  TU\times TW)|_{A\times U}}{TA\times TU} &\map& \nu_A\\
\psi_B:\frac{(T\nu_B)|_{B\times V}}{TB\times TV} \cong \nu_B|_{B\times V}\oplus \frac{(TB\times  TV\times TY)|_{B\times V}}{TB\times TV} &\map& \nu_B
\end{eqnarray*}
for the diagonal maps that live above the identity on $A\times U$ and $B\times V$. Say we have a map
\[g: A\times U\map B.\]
There gives a tubular neighborhood $\psi_A\star\psi_B$ for the diagonal map of 
\[
\xymatrix{
&\nu_A\oplus g^*\nu_B\ar[d]\\
A\times U\times V\ar[r]\ar[ru]&A\times U\times V\times W\times Y
}\]
\end{lemma}
\begin{proof}
We start by following the idea of lemma \ref{lem:propagating}. We get as far as
\[\begin{split}
&\frac{\left(T\left(\nu_A\oplus g^*\nu_B\right)\right)|_{A\times U\times V }}{TA\times TU\times TV} \cong (\nu_A\oplus g^*\nu_B)|_{A\times U\times V} \oplus \frac{\left(TA\times TU\times TV\times TW\times TY\right)|_{A\times U\times V}}{TA\times TU\times TV}\\
&\map
\left(\nu_A\times \frac{TA\times TW\times TU}{TA\times TU}\right) \times 
\left(\nu_B\times \frac{TB\times TV\times TY}{TB\times TV}\right) \\
&\map
\nu_A\times \nu_B
\end{split}\]
However now we need to be more careful than in the other construciton. Because the tubular neighborhood we are using lie above some identities, we land into the image of the embedding
\[\nu_A\oplus g^*\nu_B \maplu{Id\times g} \nu_A\times \nu_B\]
and using this, we get the desired tubular neighborhoods.
\end{proof}

\begin{lemma}\label{lem:composition twist lift}
Assume that we are in the situation of the previous lemma. Assume that we also have a diagram
\[\xymatrix{
&\nu_A|_{0}\ar[rdd]\ar[d]&&&\nu_B|_0\ar[d]\ar[rdd]\\
A_0\times U_0\ar[r]\ar[ru]\ar[rdd]& A_0\times U_0\times W_0\ar[rdd]&& B_0\times V_0\ar[r]\ar[ru]\ar[rdd]&B_0\times V_0\times Y_0 \ar[rdd]\\
&&\nu_A\ar[d]&&&\nu_B\ar[d]\\
&A\times U\ar[r]\ar[ru]& A\times U\times W&& B\times V\ar[r]\ar[ru]&B\times V\times Y 
}\]
where all squares are pullbacks. Assume that the given tubular neighborhood $\psi_A$ and $\psi_B$ restrict to tubular neighborhoods for the first diagram. The tubular neighborhood $\psi_A\star\psi_B$ then restrict to the tubular neighborhood
\[\psi_A|_0 \star \psi_B|_0\]
for the diagonal map of
\[\xymatrix{
&\nu_A|_0\oplus g_0^*\nu_B|_0\ar[d]\\
A_0\times U_0\times V_0 \ar[r]\ar[ru]&A_0\times U_0\times V_0\times W_0\times Y_0
}\]
\end{lemma}
\begin{proof}
\[\begin{split}
&\frac{\left(T\left(\nu_A\oplus g^*\nu_B\right)\right)|_{A\times U\times V }}{TA\times TU\times TV} \cong (\nu_A\oplus g^*\nu_B)|_{A\times U\times V} \oplus \frac{\left(TA\times TU\times TV\times TW\times TY\right)|_{A\times U\times V}}{TA\times TU\times TV}\\
&\map
\left(\nu_A\times \frac{TA\times TW\times TU}{TA\times TU}\right) \times 
\left(\nu_G\times \frac{TB\times TV\times TY}{TB\times TV}\right) \\
&\map
\nu_A\times \nu_B
\end{split}\]
\end{proof}

\subsection{A natural transformation for the $\cT_{\sigma}$}

Using the construction of the previous section, we can now glue the tubular neighborhoods that were used in the definition of $\cT_{\sigma}$.

\begin{lemma}
For any simplex
\[(\sigma,\tilsigma) = \left[(\G_n,\tilG_n)\smap \cdots \smap (\G_0,\tilG_0)\right]\]
of $\cG_1$, there is a natural transformation 
\[\theta : \cT_{\sigma}\times \cT_{\tilsigma} \Longrightarrow \cT_{\sigma\#\tilsigma}\]
as in
\[\xymatrix{
\Delta^n \ar@/^25pt/[rr]_{\cT_{\sigma}\times \cT_{\tilsigma}}="TT" \ar@/_25pt/[rr]^{\cT_{\sigma\#\tilsigma}}="T"&& \Top
\ar@{=>}"TT";"T"_{\theta}.
}\]
\end{lemma}
\begin{proof}
Recall that for any glue-able pair $(\G,\tilG)$, we have
\begin{alignat*}{8}
V_{\G\#\tilG} &\cong V_{\G}\du eV_{\tilG}&\cong \frac{V_{\G} \du V_{\tilG}}{\sim} &\qquad & V^{\win}_{\G\#\tilG} &\cong V^\win_{\G}&\qquad& eV_{\G \#\tilG} &\cong eV_{\G}\du eV_{\tilG}\\
E_{\G\#\tilG} &\cong E_{\G}\du eE_{\tilG} &\cong \frac{E_{\G} \du E_{\tilG}}{\sim}&\qquad & E^{\win}_{\G\#\tilG} &\cong E^\win_{\G}&\qquad& eE_{\G \#\tilG} &\cong eE_{\G}\du eE_{\tilG}.
\end{alignat*}
Fix any object
\[\alpha = (0\leq i_0 < \ldots < i_k \leq n)\]
of $\Delta^n$. To construct $\theta$, we need to define a map
\[\theta(\alpha) : \cT_{\sigma}\left(\G_{i_k}\smap \cdots \smap \G_{i_0}\right)\times \cT_{\tilsigma}\left(\tilG_{i_k}\smap \cdots \smap \tilG_{i_0}\right)
\map \cT_{\sigma\#\tilsigma}\left(\G_{i_k}\#\tilG_{i_k}\smap \cdots \smap \G_{i_0} \#\tilG_{i_0}\right).\]
Let $(x,y)$ be in the domain of this map and lets construct the element
\[z=\theta(\alpha)(x,y) \in \cT_{\sigma\#\tilsigma}\left(\G_{i_k}\#\tilG_{i_k}\smap \cdots \smap \G_{i_0} \#\tilG_{i_0}\right).\]

We first need to pick a tubular neighborhood $\varphi_{V,i_0}^z$ for the embedding
\[\rho_{z} : M^{V_{\G_{i_0}\#\tilG_{i_0}}} \map M^{V^{\win}_{\G_{i_0}\#\tilG_{i_0}}} \times W^{\left(eE_{\G_{i_k}\#\tilG_{i_k}} \du eV_{\G_{i_k}\#\tilG_{i_k}}\right) \setminus eE_{\G_{i_0} \#\tilG_{i_0}}}\]
For this, we take
\begin{eqnarray*}
A \times W = \left(M^{V^{\win}_{\G_{i_0}}}\right) \times \left(M^{eV_{\G_{i_0}}}\right) &\map& \left(M^{V^{\win}_{\G_{i_0}}}\right)\times \left(W^{(eE_{\G_{i_k}}\du eV_{i_k})\setminus eE_{\G_{i_0}}}\right)=A\times X\\
B\times Y = \left(M^{V^{\win}_{\tilG_{i_0}}}\right) \times \left(M^{eV_{\tilG_{i_0}}}\right) &\map& \left(M^{V^{\win}_{\tilG_{i_0}}}\right)\times \left(W^{(eE_{\tilG_{i_k}}\du eV_{i_k})\setminus eE_{\tilG_{i_0}}}\right)=B\times Z\end{eqnarray*}
and the map
\[A\times W =\left(M^{V^{\win}_{\G_{i_0}}}\right) \times \left(M^{eV_{\G_{i_0}}}\right) \map \left(M^{V^{\win}_{\tilG_{i_0}}}\right) = B\]
comes from the identification $\p_{\win}\tilG_{i_0} \cong \p_{\out}\G_{i_0}$. By applying the lemma \ref{lem:constructing composition}, we get a tubular neighborhood for
\[A\times W\times Y\map A\times X\times Z\]
which is exactly $\rho_z$.

Since
\[eE^{(\varphi\#\tilvarphi)_{j(j-1)}} = eE^{\varphi_{j(j-1)}} \du eE^{\tilvarphi_{j(j-1)}}\qquad eE_{\G_0\#\tilG_0} = eE_{\G_0} \du eE_{\tilG_0}\]
we can built a tubular neighborhood for 
\begin{eqnarray*}
M^{eE^{(\varphi\#\tilvarphi)_{j(j-1)}}}&\map& W^{eE^{(\varphi\#\tilvarphi)_{j(j-1)}}}\\
M^{eE_{\G_0\#\tilG_0}}&\map& W^{eE_{\G_0\#\tilG_0}}
\end{eqnarray*}
in $z$ by taking the product of the corresponding ones in $x$ and in $y$.

We now apply lemma \ref{lem:composition twist} to obtain a tubular neighborhood for
\begin{equation}\label{eqn:big rho}
M^{V_{\G_{i_0}\#\tilG_{i_0}}} \map \frac{M^{V_{\G_{i_0}\#\tilG_{i_0}}^\win} \times TW^{(eE_{\G_{i_k}\# \tilG_{i_k}}\du eV_{\G_{i_k}\#\tilG_{i_k}})\setminus eE_{\G_{i_0}\#\tilG_{i_0}}}}{TM^{V_{\G_{i_0}\#\tilG_{i_0}}}} \times M^{eH_{\G_{i_0}}}.
\end{equation}
We take 
\begin{eqnarray*}
A\times U\times W &=& M^{V^{\win}_{\G_{i_0}}} \times M^{eV_{\G_{i_0}}}\times M^{\G_{eH_{i_0}}}\\
\nu_A&=& \frac{M^{V_{\G_{i_0}}} \times TW^{(eE_{\G_{i_k}}\du eV_{\G_{i_k}})\setminus eE_{\G_{i_0}}}}{TM^{V_{\G_{i_0}}}} \times M^{eH_{\G_{i_0}}}\\
B\times V\times Y &=&M^{V^{\win}_{\tilG_{i_0}}} \times M^{eV_{\tilG_{i_0}}}\times M^{eH_{\tilG_{i_0}}}\\
\nu_B&=& \frac{M^{V_{\tilG_{i_0}}} \times TW^{(eE_{\tilG_{i_k}}\du eV_{\tilG_{i_k}})\setminus eE_{\tilG_{i_0}}}}{TM^{V_{\tilG_{i_0}}}} \times M^{eH_{\tilG_{i_0}}}
\end{eqnarray*}

Lets construct the  propagating flow $\Omega_{\rho,i_0}^z$ for the normal bundle $\lambda_{z}$ of \eqref{eqn:big rho} which is part of $z$. Since 
\[\lambda_{z} =\left.\lambda_{x}\times \lambda_y\right|_{M^{V_{i_0}\du \til{eV}_{i_0}}},\]
we first take $\Omega^x\times \Omega^y$ which is a propagating flow 
for the product and then we restrict it to the appropriate fibers.
For $z$, we take the product
\[\nabla^z_r= \nabla^x_r\times \nabla^y_r\]
to be the connection for the bundle
\[
\nu^{eE^{\varphi_r\#\tilvarphi_r}} \cong \nu^{eE^{\varphi_r}}\times \nu^{eE^{\tilvarphi_r}}.
\]

We finally need to pick a tubular neighborhood for 
\[M^{V_{\G_r\#\tilG_r}} \map 
\frac{T\left(M^{V_{\G_r\#\tilG_r}^\win} \times W^{(eE_{\G_{i_k}\#\tilG_{i_k}} \du eV_{\G_{i_k}\#\tilG_{i_k}})\setminus eE_{\G_r\#\tilG_r}}\right)}{T\left(M^{V^\win_{\G_r\#\tilG_r}}\times M^{eE^{\varphi_r\#\tilvarphi_r}\du eV_{\G_{r+1}\#\tilG_{r+1}}}\right)}.
\]
Again, we again apply lemma \eqref{lem:composition twist} with the obvious.

We still need to prove that $z$ has the appropriate lifting properties. We also need to prove that $\theta$ defines a natural transformation. Say $\om$ is a morphism
\[\beta=\{0\leq i_{j_0}<\cdots < i_{j_l}\leq n\} \maplu{\om} \alpha=\{0\leq i_0\cdots<i_{k}\leq n\}\]
of $\Delta^n$. We need to prove that the following diagram commutes.
\[\xymatrix{
\cT_\sigma(\beta) \times \cT_{\tilsigma}(\beta) 
\ar[rr]^{\cT_\sigma(\om)\times \cT_{\tilsigma}(\om)}\ar[d]^{\theta(\beta)}
&&\cT_\sigma(\alpha)\times\cT_{\tilsigma}(\alpha)\ar[d]^{\theta(\alpha)}\\
\cT_{\sigma\#\tilsigma}(\beta)
\ar[rr]^{\cT_{\sigma\#\tilsigma}(\om)}
&&\cT_{\sigma\#\tilsigma}(\alpha)
}\]
These two proofs reduce to knowing that our construction interact well with liftings. This is exactly what the  lemmas \ref{lem:composition lift} and \ref{lem:composition twist lift}.

\end{proof}

\begin{lemma}
The natural transformation 
\[\theta:\cT_{\sigma}\times \cT_{\tilsigma} \map\cT_{\sigma\#\tilsigma}\]
respects the operations. 
\end{lemma}
Lets spell out what we mean by this statement. Pick an object
\[\alpha = (0\leq i_0<\ldots < i_k\leq n)\]
in $\Delta^n$, an element
\[(x,y)\in \cT_\sigma(\alpha)\times \cT_{\tilsigma}(\alpha)\]
and $0 \leq r\leq i_0$. The elements $x$ and $y$ determine operations
\begin{eqnarray*}
\mu_{x,r}: M^{\p_{\win}\G_r}\times W_{\G_{i_k}} &\map& \Thom\left(\kappa_{\G_{i_k}}|_{M^{\G_r}}\right)\\
\mu_{y,r}: M^{\p_{\win}\tilG_r}\times W_{\tilG_{i_k}} &\map& \Thom\left(\kappa_{\tilG_{i_k}}|_{M^{\tilG_r}}\right)\\
\end{eqnarray*}
We first claim that $\mu_{y,r}$ lifts through the diagram
\[\xymatrix{
\kappa_{\G_{i_k}}|_{M^{\G_r}}\times W_{\tilG_{i_k}}\ar@{-->}[rr]^-{\tilde\mu_{y,r}}\ar[d]&&\Thom\left(\left.\kappa_{\G_{i_k}\#\tilG_{i_k}}\right|_{M^{\G_r\#\tilG_r}}\right)\ar[d]\\
M^{\G_r} \times W_{\tilG_{i_k}}\ar@{-->}[rr]^-{\mu^\prime_{y,r}}\ar[d]&&\Thom\left(\left.\kappa_{\tilG_{i_k}}\right|_{ M^{\G_r\#\tilG_r}}\right)\ar[d]\\
M^{\p_\win\tilG_r}\ar[rr]^-{\mu_{y,r}}\times W_{\tilG_{i_k}}&&\Thom\left(\left.\kappa_{\tilG_{i_k}}\right|_{M^{\tilG_r}}\right)
}\]
and that $x$, $y$ and $z=\theta(\alpha)(x,y)$ give  a commutative diagram
\[\xymatrix{
M^{\p_\win\G_r} \times W_{\G_{i_k}} \times W_{\tilG_{i_k}} \ar[rr]^{\mu_{x,r}\times Id} \ar[rrd]_{\mu_{z,r}}&& \Thom(\kappa_{\G_{i_k}}|_{M^{\G_r}})\times W_{\tilG_{i_k}} \ar[d]_{\til\mu_{y,r}}\\
&& \Thom(\kappa_{\G_0\#\tilG_0}).
}\]

\begin{proof}
Lets first show that we get a lift $\tilde\mu_{y,r}$ of $\mu_y$. For the glue-able pair $(\G_r,\tilG_r)$ we have maps of graphs
\[\xymatrix{
\p_\win \tilG_r\ar[r]\ar[d]& \tilG_r\ar[d]\\
\G_r\ar[r]& \G_r\#\tilG_r
}\]
We first consider the diagram
\[\xymatrix{
M^{\G_r\#\tilG_r}\ar[r]\ar[d] 
&M^{\G_r}\times PM^{eE_{\tilG_{r}}}\times M^{eV_{\tilG_r}}\ar[d]
&\ar[l] M^{\G_r} \times M^{eE_{\tilG_{r}}\du eV_{\tilG_{r}}} \ar[r]\ar[d]
&M^{\G_r} \times W^{eE_{\tilG_r}\du eV_{\tilG_r}}\ar[d]\\
M^{\tilG_r}\ar[r]
& M^{\p_\win \tilG_r} \times PM^{eE_{\tilG_r}} \times M^{eV_{\tilG_r}} &\ar[l] M^{\p_\win \tilG_r} \times M^{eE_{\tilG_r}\du eV_{\tilG_r}} \ar[r]
&M^{\p_\win \tilG_r} \times W^{eE_{\tilG_r}\du eV_{\tilG_r}}
}\]
Since the choices contain in $y$ lived above the identity on the incoming boundary of $\tilG_r$, our operation $\mu_{y,r}$ restricts to an operation $\mu^{\prime}_{y,r}$ on the top row. For the same reason, we can then twist this operation by $\kappa_{\G_{i_k}}|_{M^{\G_r}}$.

To prove that these choices give a compatible Thom collapse map, we consider the following diagram. Note that to shorten the notation we have dropped the $i_0$'s and the $i_k$'s as there is no confusion.  We have also used $W(\G)$, $\til{W}$ and $W^\#$ to represent the Euclidean spaces $W^{eE_{\G_{i_k}}\du eV_{\G_{i_k}}}$, $W^{{eE}_{\tilG_{i_k}}\du eV_{\tilG_{i_k}}}$ and $W^{eE_{\G_{i_k}\#\tilG_{i_k}} \du eV_{\G_{i_k}\#\tilG_{i_k}}}$.
\[\xymatrix@C=13pt{
M^{\G^\#}\ar@{=>}[r]\ar@{=>}[dr]
& M^{\G}\times PM^{\til{eE}}\times M^{\til{eV}}\ar@{=>}[d] 
&\ar[l] M^{\G}\times M^{\til{eE}\du \til{eV}}\ar@{=>}[r] \ar@{=>}[d] 
&M^{\G}\times \tilde{W}\ar@{=>}[d]\\
&M^{\p_\win\G}\times PM^{eE^\#}\times M^{ eV^\#} 
&\ar[l] M^{\p_\win\G}\times PM^{eE}\times M^{eV^\#\du \til{eE}} \ar@{=>}[r] 
&M^{\p_\win\G}\times PM^{eE}\times M^{eV}\times \tilde{W}\\
&&\ar[lu] M^{\p_\win\G} \times M^{eE^\# \du eV^\#} \ar@{=>}[rd]\ar@{=>}[r]\ar[u]
& M^{\p_\win\G}\times M^{eE\du eV}\times \tilde{W}\ar@{=>}[d]\ar[u]\\
&&&M^{\p_\win\G}\times W^\#
}\]
We will show how to construct the operation for each double arrow in a compatible way. In the end, the vertical operation will be the one associated $x$ and $\G$, the horizontal one will be  the lifted one associated to $y$ and $\tilG$ and the diagonal one the operation we have just constructed.

Since the extra terms are kept constant, we can use $x$ to pick the choices for the vertical maps. Similarly the choices for $\tilG_0$ determine the tubular neighborhood of the horizontal arrows. However for the horizontal maps, we need to use the map
\[M^{V_{\G_0}} \map M^{V_{\tilG_0}^{\win}}\]
to lift the old tubular neighborhood. To go from $M^{\p_\win\tilG}$ to $M^{\G_0}$,
we use that the tubular neighborhoods keep $M^{\til{V_0^{\win}}}$ constant.
The diagonal maps are all already constructed except for
\[M^{\G}\times M^{\tilde{eE}\du \til{eV}} \map M^{\p_{\win}\G}\times PM^{eE}\times M^{eV}\times \til{W}.\]
This one is easy since the two fat graphs are independent. We use the tubular neighborhood for
\[M^{\p_{\win}\tilG} \times M^{\til{eE}\du\til{eV}} \map M^{\p_{\win}\tilG} \times \tilde{W}.\]
And since it does not move $M^{\p_{\win}\tilG}$, it will not interfere with the choices we have made for 
\[M^{\G}\times  \map M^{\p_{\win}\G}\times PM^{eE}\times M^{eV}.\]

We claim that these choices give compatible operations. Only three positions are not trivial. Lets consider each of these separately.

Consider first the bottom triangle
\[\xymatrix{
 M^{\p_\win\G} \times M^{eE \du eV} \times M^{\til{eE} \du \til{eV}} \ar@{=>}[rd]_{(C)}\ar@{=>}[r]^{(A)}
& M^{\p_\win\G}\times M^{eE\du eV}\times \tilde{W}\ar@{=>}[d]^{(B)}\\
&M^{\p_\win\G}\times W\times \til{W}.
}\]
The point in $M^{\p_\win\G}$ determines which tubular neighborhood for
\[M^{eE\du eV}\map W(\G) \qquad M^{\til{eE}\du \til{eV}}\map \til{W}\]
and these tubular neighborhood are  completely independent. Hence the composition of Thom collapses of $(A)$ and $(B)$ is the operation corresponding to the product of these. This is what we chose for $(C)$.

Lets now consider the top right square
\[\xymatrix{
M^{\G}\times M^{\til{eE}\du \til{eV}}\ar@{=>}[r]^{(A)} \ar@{=>}[d]_{(C)} 
&M^{\G}\times \tilde{W}\ar@{=>}[d]^{(D)}\\
M^{\p_\win\G}\times PM^{eE}\times M^{eV^\#\du \til{eE}} \ar@{=>}[r]_{(B)} 
&M^{\p_\win\G}\times PM^{eE}\times M^{eV}\times \tilde{W}\\
}\]
The tubular neighborhoods for $(A)$ and $(B)$ both use the tubular neighborhood
\[M^{\til{V}^{\win}} \times \nu^{\til{eE}\du \til{eV}} \map M^{\til{V}^{\win}} \times W^{\til{eE}\du \til{eV}}\]
which is part of $x$ and leave the remainder and $M^{\til{V}^{\win}}$ untouched. On the contrary, the tubular neighborhoods for $(C)$ and $(D)$ ignore the terms of $\tilG$. In particular this square give compatible operations.

Finally lets look at the left-most triangle
\[\xymatrix{
M^{\G^\#}\ar@{=>}[r]^-{(A)}\ar@{=>}[dr]_-{(C)}
& M^{\G}\times PM^{\til{eE}}\times M^{\til{eV}}\ar@{=>}[d]^-{(B)} \\
&M^{\p_\win\G}\times PM^{eE^\#}\times M^{ eV^\#} 
 }.\]
This part of the construction is the more subtle. First, we look at the finite dimensional,
\[\xymatrix{
M^{V^\#} \ar[r]^-{(A)_{\word{fin}}}\ar[rd]_-{(C)_{\word{fin}}}&M^{V}\times M^{\til{eH}}\times M^{\til{eV}}\ar[d]^-{(B)_{\word{fin}}}\\
& M^{V^\win}\times M^{eH^\#}\times M^{eV^\#}.
}\]
We twist this with the bundle
\[\begin{split}
\frac{\left.\left(M^{V^{\win}}\times TW^{(eE_{i_k}^\#\du eV_{i_k}^\#)\setminus eE^\#}\right)\right|_{M^{V}}}{M^{V^\win} \times TM^{eV^\#}} \times M^{eH^\#}\\
& \cong \frac{M^{V^{\win}} \times TW^{(eE_{i_k}\du eV_{i_k})\setminus eE}\times TW^{(\til{eE}_{i_k} \du \til{eV}_{i_k})\setminus \til{eE}}}{M^{V^{\win}}\times TM^{eV_{i_k}}\times TM^{\til{eV}_{i_k}}} \times M^{eH}\times M^{\til{eH}}
\end{split}\]
above $M^{V^\#\du eH^\#}$. Since our tubular neighborhoods live above the identity on $M^{V^\win}$, any point $\textbf{m} \in M^{V^\win}$ gives a tubular neighborhoods for the diagram
\[\xymatrix{
M^{eV}\times M^{\til{eV}} \ar[r]\ar[rd]& M^{eV} \times  M^{\til{eV}}\times M^{\til{eH}}\ar[d]\\
&\frac{TW^{(eE_{i_k}\du eV_{i_k})\setminus eE}}{TM^{eV_{i_k}}}
\times M^{eH}\times \frac{TW^{(\til{eE_{i_k}}\du\til{eV_{i_k}})\setminus \til{eE}}}{TM^{\til{eV}_{i_k}}} \times TM^{\til{eH}}.
}\]
Now the part coming from $\G$ and the part coming from $\tilG$ are completely independent. In particular, we get the construction for $\G\#\tilG$.

Now it suffices to see how we use the propagating flows. Recall that we are only moving the extra edges and that
\[eE^\# = eE \du \til{eE}.\]
In particular, since the propagating flow for the diagonal map is the product of the two propagating flow, the movement of one edge of $\G$ only depends on the propagating flow of $\G$. And similarly for $\tilG$. 
\end{proof}

 \appendix
%-------------------------------------------------------A space of tubular neighborhood
\section{Tubular neighborhoods}

\label{app:tubular}
\subsection{Tubular neighborhoods of finite dimensional embedding}
Say we have an embedding $\rho : M\subset N$. Let
\[\nu = \frac{TN|_M}{TM}\]
be its tubular neighborhood.

\begin{definition}
A \emph{tubular neighborhood} is an embedding
\[f:\nu\map N\]
so that the restriction  
\[\xymatrix{
f|_M : M \ar[rr]^{0-sect} &&\nu \ar[r]^f&N
}\]
to the $0$-section is the inclusion $\rho$
and so that the composition
\[TM \oplus \nu \cong T\nu|_M  \maplu{df} TN|_M \map \nu\]
is the projection onto the second component.
\end{definition}

\begin{figure}
\begin{center}
\mbox{\epsfig{file=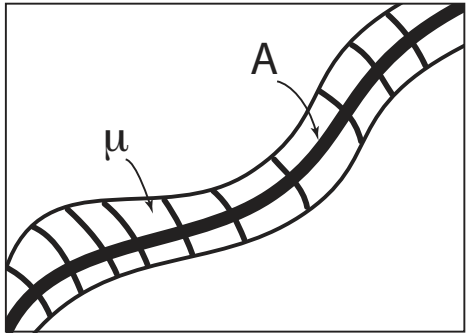, width=200pt}}
\caption{A tubular neighborhood.}\label{fig:tube}
\end{center}
\end{figure}

Assume that both $M$ and $N$ are smooth finite dimensional manifold. Assume that $M$ is compact.
Consider the space $\Tub(\rho)$ of all tubular neighborhoods of $\rho$. We topologize it as a subspace of the space of all embeddings of $\nu$ into $V$ with the $C^{\infty}$ topology.

\begin{proposition}\label{prop:tub contractible}
The space $\Tub(\rho)$ is contractible.
\end{proposition}
\begin{proof}
This is a classical result. We include a proof because it gives a more general statement which we shall use.
Fix one tubular neighborhoods $f_0: \nu \map N$ and pick a metric on the bundle $\nu$ and an extension of it to $N$. Let $V_0$ denote the image of $f_0$.

The first step is to homotope $\Tub_{f_0}(\rho)$ into the subspace $\Tub_{f_0}(\rho)$ which contains all tubular neighborhoods $f$ so that $f(\nu)\subset f_0(\nu)$. Consider the metric 
\[\|f\|_1 = \sup\{\|df(v)\| \qquad  \|v\| \leq 1\}\]
on $\Tub(\rho)$. Fix a $\delta>0$ so that
\[\|f\|_1 < \delta \implies f(B_1(\nu)) \subset f_0(\nu).\]

We now construct a homotopy of $\Tub(\rho)$ into the subspace $\Tub(\rho)_{f_0}$. Define
\[H:\Tub(\rho) \times I \map \Tub(\rho)\]
by
\begin{eqnarray*}
H(f,t) : \nu &\map& V\\
y&\maps& \begin{cases}
f\left((1-t)y + t \left(\frac{1}{\frac{\pi |y| \|f\|} {2 \delta}} \arctan\left[ \frac{ \pi |y| \|f\|}{2 \delta} \right]  \right) y \right) & |y|\neq 0 \\
f(y)& \word{otherwise.}
\end{cases}
\end{eqnarray*} 
This function is smooth because the function
\[ g(x) = \begin{cases}
\frac{1}{x} \arctan(x)&x\neq 0\\
1& x=0
\end{cases}\]
is both smooth and even. Since the derivative of the function $g$ is 1 at zero, $H(f,t)$ is a tubular neighborhood for each $t$. Because for any positive $c$
\[\lim_{x\map \infty} \frac{2}{\pi c} \arctan[\frac{\pi c}{2} x ] =  c,\]
we have that for any $y\in \nu$
\[\left(\frac{1}{\frac{\pi |y| \|f\|} {2\delta}} \arctan\left[ \frac{ \pi |y| \|f\|}{2 \delta} \right]  \right) y \in D_{\delta/\|f\|}\nu.\]
For any $f=\alpha(z)$ and for $y\in \nu$ with $|y|\neq 0$, we have that
\[H(f,1)(y) \in f(D_{\delta/\|f\|}\nu) \subset V_0\]

We will now retract all of $\Tub(\rho)_{f_0}$ onto $f_0$.  For any $f$ in $\Tub(\rho)_{f_0}$, because
$f(\nu)\subset f_0(\nu)$, we can make sense of
\[g=f_0^{-1}\cdot f : \nu\map \nu.\]
By our definition of tubular neighborhoods, $g$ is the identity on $M\subset \nu$ and whose fiberwise derivative is the identity at $M$.
Consider the map
\begin{eqnarray*}
F_{f}: \nu\times I &\map& \nu\\
(y,t)&\maps& \begin{cases}
t^{-1}g(ty) &t>0\\
Id&t=0
\end{cases}
\end{eqnarray*}
By Hirsch's argument $F_f$ is a smooth homotopy between the identity and $g$. By composing it with $f_0$, we get a isotopy between $f_0$ and $f$. This homotopy depends continuously on $f$.
\end{proof}

\subsection{Compatible Thom collapses}

Lets now consider the pull-back diagram
\[\xymatrix{
A\ar@{^(->}[r]^{\rho_A}\ar[d]&B\ar[d]\\
C\ar@{^(->}[r]^{\rho_C}&D
}\]
We assume that $B$ and $C$ are embedded in D and that they meet transversely at A. Let 
\[\mu = \frac{TB|_A}{TA} \qquad \lambda=\frac{TD|_C}{TC}\]
be the normal bundles. Note that since we have a transverse intersection, 
\[ \mu = \lambda|_A.\]

We want to give conditions on the tubular neighborhoods of $A$ in $B$ and $C$ in $D$ under which the Thom collapses along the horizontal arrows give a commutative diagram
\[\xymatrix{
B\ar[r]\ar[d]& Thom(\mu)\ar[d]\\
D\ar[r]& Thom(\lambda).
}\]
This diagram will commute if and only if the choice of tubular neighborhoods make the following commutes.
\[\xymatrix{
\mu \ar[r]\ar[d]& B\ar[d]\\
\lambda\ar[r]&D
}\]
Hence a choice of tubular neighborhood for $\rho_C$ determines completely the one for $\rho_A$. However, we do need 
\[f:\lambda\map D\]
to map the subbundle $\lambda|_A$ into $B$.
Let $\Tub(\rho_A,\rho_C)\subset  \Tub(\rho_C)$ be the space of tubular neighborhood of $\rho_C$ which induce a tubular neighborhood for $\rho_A$.

\begin{figure}
\begin{center}
\mbox{\epsfig{file=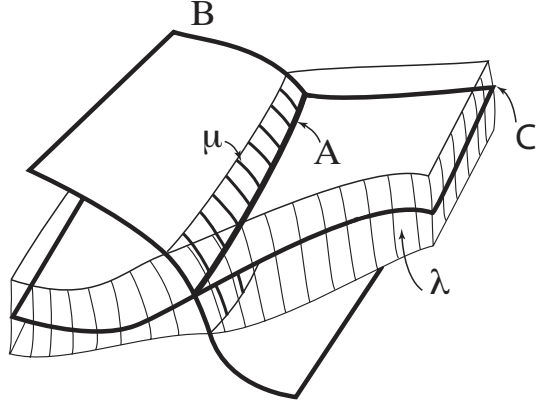, width=200pt}}
\caption{Compatible tubular neighborhoods.}\label{fig:tubeABCD}
\end{center}
\end{figure}

\begin{proposition}
The space $\Tub(\rho_A,\rho_C)$ is contractible.
\end{proposition}
\begin{proof}
This proof follows the proof of proposition \ref{prop:tub contractible}. Start with any $f_0:\lambda\map D$ so that 
\[f_0(\lambda|_A)\subset B.\] 
Both steps involved in contracting the space $\Tub(\rho_C)$ will restrict to homotopies between tubular neighborhoods which have this condtition.
\end{proof}

\subsection{Compositions of Thom collapses.}

We will also need some compatibility for compositions of embeddings
\[A \maplu{f} B\maplu{g} C.\]
Denote the normal bundles by
\[\mu = \frac{TB|_A}{TA} \qquad \lambda = \frac{TC|_B}{TB} \qquad \nu = \frac{TC|_A}{TA}\]
We have a short exact sequence
\[\mu\map \nu\map \lambda|_A.\]
This sequence splits but not naturally. Say we have chosen tubular neighborhoods 
\[ f_\mu : \mu \map B\qquad f_\lambda:\lambda\map C\qquad f_\nu:\nu \map C\]
for each of these embeddings. We get three Thom collapses
\[B\map Thom(\mu)\qquad C \map Thom(\lambda)\qquad C\map Thom(\nu).\]
We would like to compare the composition of the first two to the third. However the first two do not compose. To add to this, the target of the unexistent composition would not be the same as the target of the third one. 

We therefore consider tubular neighborhoods for the composition
\[\xymatrix{
A\ar[r]& B \ar[rr]^{0-sect}&& \lambda.
}\]
whose normal bundle is
\[\frac{\left.T\left(\lambda\right)\right|_{A}}{TA} 
\cong \frac{TB|_A\oplus \lambda|_{A}}{TA} 
\cong \mu\oplus \lambda|_A 
\]
A choice of a tubular neighborhood 
\[ \widetilde{f}_{\lambda} : \mu\oplus \lambda|_A\map \lambda \]
gives a Thom collapse
\[\xymatrix{
 \lambda \ar[r]& Thom(\mu\oplus \lambda|_A) 
 %\\&\frac{D_\epsilon \lambda}{S_{\epsilon}\lambda} \cong Thom(\lambda)\ar@{-->}[ru]
}\]
which extends to $Thom(\lambda)$.
We can then compose the two Thom collapses to get
\[C\map Thom(\lambda) \map Thom(\mu\oplus \lambda|_A)\]

From the choice  tubular neighborhood $f_\nu$ for $g$, we get identifications
\[ T(\lambda ) |_B \cong TC|_B \qquad \mu\oplus \lambda|_A \cong \nu\]
and hence we get a diagram
\begin{equation}\label{dia:composing compatibility}
\xymatrix{
C\ar[r]\ar[rrd]& Thom(\lambda)\ar[r]& Thom(\mu\oplus \lambda|_A)\ar@{<->}[d]\\
&& Thom(\nu).
}\end{equation}

\begin{definition}
We call the tubular neighborhoods $f_\mu$, $\tilde{f}_\lambda$ and $f_\nu$
compatible if they make diagram \ref{dia:composing compatibility} commutes.
Let $\Tub(\rho_{AB},\rho_{BC},\rho_{AC})$ denote the space of compatible choices of tubular neighborhood for $f$, $g$ and $g\cdot f$.
\end{definition}

\begin{proposition}
The space $Tub(\rho_{AB},\rho_{BC},\rho_{AC})$ is contractible.
\end{proposition}

\begin{figure}
\begin{center}
\mbox{\epsfig{file=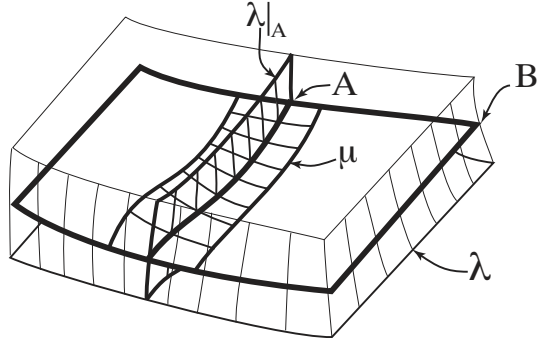, width=200pt}}
\caption{The inclusion of $A$ into $\lambda$ and then into $C$.}
\end{center}
\end{figure}

\begin{proof}
The space $\Tub(\rho_{AB},\rho_{BC},\rho_{AC})$ is homeomorphic to the product
\[\Tub((\word{0-sect}) \cdot \rho_{AB})\times \Tub(\rho_{BC}).\]
For \emph{any} choice of $\widetilde{t}_\mu$ and $t_\lambda$, there is exactly one compatible $t_\nu$, namely the composition
\[\nu \maplu{\cong} \mu\oplus \lambda|_A \maplu{\widetilde{t}_\mu} \lambda \maplu{t_\lambda} C.\]
Note that there is a pullback diagram
\[\xymatrix{
A\ar[r]\ar[d]&\lambda\ar[d]^{t_\lambda}\\
A\ar[r]&C.
}\]
This implies that once $t_\lambda$ is  fixed, the problem simplifies to the pullback case of the previous section. 
\end{proof}

\bibliographystyle{plain}
\bibliography{references}

\end{document}